\documentclass[10pt]{article}

\usepackage[dvipdf]{graphicx}
\usepackage[latin1]{inputenc}
\usepackage{latexsym}
\usepackage{stmaryrd}
\usepackage{amsmath,amssymb,amsfonts,amsthm}
\usepackage{xcolor}
\usepackage{mathrsfs}
\usepackage[colorlinks=true,citecolor=red,linkcolor=blue,urlcolor=blue,pdfstartview=FitH]{hyperref}
\providecommand{\keywords}[1]{{\textit{Keywords:}} #1}
\providecommand{\AMS}[1]{{\textit{AMS Classification:} }#1}

\newcommand{\rrbrack}{\rrbracket}

\newcommand{\nrn}{\rightarrow+\infty}
\newcommand{\xrn}{\xrightarrow}

\newcommand{\ER}{\mathbb {R}}
\newcommand{\R}{\mathbb {R}}
\newcommand{\EN}{\mathbb {N}}

\newcommand{\PE}{\mathbb {P}}
\renewcommand{\P}{\mathbb {P}}
\newcommand{\ES}{\mathbb{E}}
\newcommand{\E}{\mathbb{E}}
\newcommand{\EMM}{L}
\newcommand{\ERR}{R}

\newcommand{\psg}{( }
\newcommand{\psd}{) }

\newcommand{\cqfd}{\Box}

\newcommand{\bqn}{\begin{equation}}
\newcommand{\bqne}{\begin{equation*}}
\newcommand{\eqn}{\end{equation}}
\newcommand{\eqne}{\end{equation*}}

\newcommand{\bx}{\bar X}

\newcommand{\gamzero}{\gamma_1}
\newtheorem{theorem}{\textnormal{\bf{T\scriptsize{HEOREM}}}}
\newtheorem{prop}{\textnormal{\bf{P\scriptsize{ROPOSITION}}}}

\newtheorem{lemme}{\textnormal{\bf{L\scriptsize{EMMA}}}}

\theoremstyle{definition}
\newtheorem{definition}{\textnormal{\bf{D}\scriptsize{EFINITION}}}
\theoremstyle{remark}

\newtheorem{Remarque}{\textnormal{\bf{R\scriptsize{EMARK}}}}

\author{Gilles Pag\`es~\thanks{UPMC,  Laboratoire de Probabilit\'es et Mod\`eles al\'eatoires, UMR~7599, case 188, 4, pl. Jussieu, F-75252 Paris Cedex 5, France. E-mail: \texttt{gilles.pages@upmc.fr}}  $\;$and 
Fabien Panloup~\thanks{Institut de Math\'ematiques de Toulouse, Universit\'e Paul Sabatier \& INSA Toulouse, 135, av. de Rangueil, F-31077 Toulouse Cedex 4, France. E-mail: \texttt{fabien.panloup@math.univ-toulouse.fr}}
}

\title{Weighted Multilevel Langevin  Simulation of Invariant Measures}
\begin{document}
\maketitle

\begin{abstract}
We investigate a weighted Multilevel Richardson-Romberg extrapolation for the ergodic approximation of invariant distributions of diffusions adapted from the one  introduced in~\cite{LEPA} for regular Monte Carlo simulation.  In a first result, we prove under  weak confluence assumptions on the diffusion,  that for any integer $R\ge2$,  the procedure allows us  to attain a rate $n^{\frac{R}{2R+1}}$  whereas the original algorithm convergence is at a weak rate $n^{1/3}$. Furthermore, this is achieved without any explosion of the asymptotic variance. In a second part, under stronger confluence assumptions and  with the help of some second order expansions of the asymptotic error, we go deeper in the study by optimizing the choice of the parameters involved by the method. In particular, for a given $\varepsilon>0$, we exhibit some semi-explicit parameters for which the number of iterations of the Euler scheme required to attain a  Mean-Squared Error lower than $\varepsilon^2$ is about $\varepsilon^{-2}\log(\varepsilon^{-1})$. \smallskip

\noindent Finally, we numerically this Multilevel Langevin estimator on several examples including the simple one-dimensional Ornstein-Uhlenbeck process but also on a high dimensional  diffusion  motivated by a statistical problem. These examples  confirm the theoretical  efficiency of the method.

\end{abstract}

\medskip
\noindent \keywords{Ergodic diffusion, Invariant measure, Multilevel, Richardson-Romberg, Monte Carlo.}

\medskip
\noindent \AMS{60J60, 37M25, 65C05.}
\section{Introduction}

%

Let $(X_t)_{t\in [0,T]}$ be the unique strong solution to the stochastic differential equation ($SDE$)
\[
dX_t = b(X_t)dt +\sigma(X_t)dW_t
\]
starting at $X_0$ where $W$ is a  standard $\R^q$-valued standard Brownian motion, independent of $X_0$, both defined on a probability space $(\Omega, {\cal A}, \PE)$, where $b:\R^d\to \R^d$ and $\sigma:\R^d \to {\cal M}(d,q, \ER)$ are  locally Lipschitz continuous functions {with at most linear growth}. 
The process $(X_t)_{t\ge 0}$ is a Markov process and we denote by $\PE_{\mu}$ its distribution starting from $X_0\sim \mu$. 
Let ${\cal L}$ denote  its infinitesimal generator, defined on twice differentiable functions $g:\R^d\to \R$ by
\[
{\cal L}g = (b|\nabla g) +\frac 12 {\rm Tr}\big(\sigma^*D^2g\,\sigma\big).
\]

As soon as there exists a continuously twice differentiable {\em Lyapunov} function $V:\R^d \to \R_+$ such that 
\begin{equation}\label{eq:Lyap}
\sup_{x\in \ER^d}{\cal L}V (x) <+\infty \quad \mbox{ and }\quad \varlimsup_{|x|\to +\infty} {\cal L}V(x) <0,
\end{equation}
then there exists an invariant probability measure $\nu$ for the diffusion in the sense that $X$ is a stationary process under $\PE_{\nu}$, so that $X_t \sim  \nu$ for every t$\,\in \ER_+$. Under appropriate (hypo-)ellipticity assumptions on $\sigma$ or global confluence assumptions ({on this topic, see $e.g.$~\cite{LPPIHP}}), this invariant measure $\nu$ is unique, hence ergodic. In particular,  
\[
\PE_{\nu}(d\omega)\mbox{-}a.s. \quad \nu_t(\omega,d\xi)=\frac 1t \int_0^t\delta_{X_s(\omega)}ds\stackrel{(\ER^d)}{\Longrightarrow} \nu
\]
where $\stackrel{(\ER^d)}{\Longrightarrow} $ denotes weak convergence of distributions on $\ER^d$ {(see $e.g.$~\cite{billingsley-ergodic} or~\cite{krengel-ergodic} for background)}.
We will assume that this uniqueness holds throughout the paper. Under additional assumptions, one shows that the diffusion is {\em stable} in the sense that 
\[
\forall\, x\!\in \ER^d, \;\PE_x(d\omega)\mbox{-}a.s. \quad \nu_t(\omega,d\xi)
\stackrel{(\ER^d)}{\Longrightarrow} \nu.
\]

This $\P_x$-$a.s.$ convergence is ruled by Bhattacharya's  CLT (see~\cite{bhatta82} for detailed assumptions), namely, if $f:\ER^d\rightarrow\ER$ is such that the Poisson equation $f-\nu(f)= -{\cal L}g$ admits a solution, then 
\begin{equation}\label{eq:1}
 \sqrt{t} \big(\nu_t(\omega,f)-\nu(f)\big) \stackrel{(\ER^d)}{\Longrightarrow}  {\cal N}\big(0, \sigma^2(f)\big)
\end{equation}
with $\sigma^2(f) = \int_{\ER^d} |\sigma^*\nabla g|^2 d\nu$ where $\sigma^*$ denotes the transpose matrix of $\sigma$. 

\medskip 
In a series of papers (see $e.g.$~\cite{LP1,LP2,lemaire2,PP1,PP3,panloup1}), the above properties have been exploited in order to compute by ergodic simulation integrals  $\int fd\nu = \E_{\nu}f(X_t)$ or, more generally, $\E_{\nu} F\Big((X_t)_{t\in [0,T]}\Big)$ where $F$ is a (path-dependent) functional defined  on the space ${\cal C}([0,T], \ER^d)$ {(see also~\cite{talay} or~\cite{piccioni-scarlatti} for other references on the topic or more recently~\cite{garcia-trillos})}.

The starting idea is to mimic~\eqref{eq:1}. First we  replace the diffusion $X$ by a discretization scheme with {\em decreasing step}.  To be more precise, we
 consider,  for a given non-increasing sequence of   steps $\gamma_n>0$, $n\ge 1$, the associated Euler scheme with decreasing step defined by
\begin{equation}\label{eq:EulerDec}
\bar X_{n+1}= \bar X_n  +\gamma_{n+1} b(\bar X_n) +\sigma(\bar X_n) \big(W_{\Gamma_n}-W_{\Gamma_{n-1}}\big), \; n\ge 0, \; \bar X_0=X_0,
\end{equation}
where $\Gamma_n = \gamma_1+\ldots+ \gamma_n$, $n\ge 1$.
Then we introduce  (for technical matter to be explained further on)  a {\em weight} sequence $(\eta_n)_{n\ge 1}$ and the related  $\eta$-weighted empirical (or occupation) measures of  the above Euler scheme, namely
\[
\nu^{\eta,\gamma}_n (\omega,dx) = \frac{1}{H_n} \sum_{k=1}^n \eta_{k} \delta_{\bar X_{k-1}(\omega)}.
\]
{The computation of $\nu^{\eta,\gamma}_n(f)$ can be performed  recursively, once noted that that
\begin{equation}\label{eq:crude}
\nu^{\eta,\gamma}_n (f)= \frac{\eta_n}{H_n} f(\bar X_n) + \Big(1-\frac {\eta_n}{H_n}\Big) \nu^{\eta,\gamma}_{n-1} (f), \; \nu^{\eta, \gamma}_0(f)=0.
\end{equation}
}
 It is clear that, in order to let the scheme explore the whole state space $\ER^d$ and to let  the empirical measures take into account new values as $n$ grows, we must require  the pair $(\eta_n, \gamma_n)_{n\ge 1}$ satisfies 
\begin{equation}\label{poidspastendinfty}
H_n := \eta_1 +\cdots +\eta_n \to +\infty \quad\mbox{and}  \quad \Gamma_n:= \gamzero  +\cdots +\gamma_n \to +\infty \quad \mbox{ as } n\to +\infty.
\end{equation}

When $\eta=\gamma$, the $\gamma$-empirical measure $\nu^{\gamma,\gamma}$ is the natural counterpart of $\nu_t$ and one expects that, under natural {\em mean-reverting} assumptions similar to~\eqref{eq:Lyap} (or slightly more stringent), $\PE_x(d\omega)$-$a.s.$ $\nu^{\eta,\gamma}_n (\omega,dx)\stackrel{(\ER^d)}{\Longrightarrow} \nu$ taking advantage of the fact that the step $\gamma_n \downarrow 0$.  
The major difference with  the above continuous time pointwise Birkhoff's ergodic theorem is that, provided $b$, $\sigma$  can be computed easily,  these random measures taken against a function $f$ (computable as well) can in turn  be simulated. This opens the way to simulation based ergodic methods to compute $\nu(f)$. Note that, though we will not go deeper in that direction,   when $\nu = h.\lambda_d$ is absolutely continuous such a method appears as a probabilistic numerical scheme for the resolution of  the stationary Fokker-Planck equation ${\cal L}^*h= 0$ by providing  the values of as many integrals $\int fh\lambda_d $ as required. 

\smallskip
Let us first recall one simple convergence result for the $a.s.$ weak convergence of the weighted empirical measures $(\nu^{\eta, \gamma}_n)_{n\ge 1}$ (see Theorem V.2 borrowed and slightly adapted from~\cite{lemairethese}).

\begin{prop}\label{prop:convergence} Assume $b$ and $\sigma$ satisfy the mean-reverting assumption\\

\smallskip
$\mathbf{(S)}$: There exists a positive ${\cal C}^2$-function $V:\ER^d\rightarrow\ER_+$ and  $\rho\!\in (0,+\infty)$  such that
$$
\lim_{|x|\rightarrow+\infty} \frac{V(x)}{|x|^\rho}=+\infty, \quad |\nabla V|^2\le C V\quad \textnormal{and }\quad \sup_{x\in\ER^d}\|D^2 V(x)\|<+\infty
$$
and  there exist some   real constants $C_{b}>0$, $\alpha>0$  and $\beta\ge 0$  such that:
\begin{align*} 
&\textit{(i)}\quad |b|^2\le C_b V, \quad {\rm Tr}(\sigma\sigma^*)(x)=o\big(V(x)\big)\quad\textnormal{as $|x|\rightarrow+\infty$.  }
&\textit{(ii)}\quad \psg\nabla V| b\psd \le \beta-\alpha V.
\end{align*}
Then $(SDE$) admits at least one invariant distribution $\nu$ and  for every $x\!\in \ER^d$ and $p>0$, $\sup_n \E_x V^p(\bar X_n)<+\infty$.\smallskip

 Assume $\nu $ is the unique invariant measure of $(SDE$). If the pair $(\eta_n, \gamma_n)_{n\ge 1}$ satisfies~\eqref{poidspastendinfty}
\begin{equation}\label{eq:poids-pas}
\sum_{n\ge 2}\frac{1}{H_n} \left(\frac{\eta_n}{\gamma_n}-\frac{\eta_{n-1}}{\gamma_{n-1}}\right)_{+}<+\infty \quad  \mbox{ and }\quad \sum_{n\ge1}\left(\frac{\eta_n}{H_n\sqrt{\gamma_n}}\right)^2 <+\infty
\end{equation}
then,  $\PE_x(d\omega)$-$a.s.$ $\;\nu^{\eta,\gamma}_n (\omega,dx)\stackrel{(\ER^d)}{\Longrightarrow} \nu$. 

Moreover, $ \PE_x\mbox{-} a.s. $, for every $\nu$-$a.s.$ continuous functions $\R^d \to \ER$ with $V$-polynomial growth, 
\begin{equation}\label{eq:asCvgnu}
\nu^{\eta, n}(\omega,f)\to \nu(f) \quad \mbox{ as } \quad n \to+\infty.
\end{equation}
\end{prop}

\begin{Remarque} $\rhd$ By $V$-polynomial growth we mean that  $f= O(V^p)$  at infinity for some $p>0$.
 
\smallskip
\noindent $\rhd$ The condition $\mathbf{(S)}$ is stronger than~\eqref{eq:Lyap}. It  implies that there exists $\alpha'\!\in(0,+\infty)$ and $\beta \!\in \ER$ such that ${\cal L}V \le \beta' -\alpha' V$. In fact the  conclusions of the above proposition are also true for the continuous time occupation   measure $\nu_t(\omega)= \frac{1}{t}\int_0^t \delta_{X_s(\omega)}ds$ of the diffusion itself.

\smallskip
\noindent $\rhd$ The above result remains true under weaker Lyapunov assumptions of the following type: ${\cal L}V \le \beta' -\alpha' V^a$ with $a\in(0,1]$. For the sake of simplicity, we choose in this paper to state the results under $\mathbf{(S)}$ only but all what follows can be extended   to the weaker setting owing to additional technicalities (involving the control of the moments of the diffusion or of the Euler scheme~\eqref{eq:EulerDec}).  

\smallskip
\noindent $\rhd$ In the above proposition, the condition $\displaystyle \lim_{|x|\rightarrow+\infty} \frac{V(x)}{|x|^\rho}=+\infty$ can be relaxed into $\displaystyle \lim_{|x|\rightarrow+\infty} {V(x)}=+\infty$. For the sequel, the interest of this slightly reinforced assumption is to ensure that every function $f$ with polynomial growth has  a $V$-polynomial growth.
\end{Remarque}

\begin{definition} A pair  $(\eta_n, \gamma_n)_{n\ge 1}$ (with decreasing $\gamma_n)$ satisfying~\eqref{poidspastendinfty} and~\eqref{eq:poids-pas} is called an {\em averaging system}.
\end{definition} 

\noindent {\bf Examples.} If $\gamma_n = \gamzero n^{-a}$ and $\eta_n = \eta_1 n^{-c}$, then the pair $(\eta_n, \gamma_n)_{n\ge 1}$ is averaging as soon as $0<a<1$ and $0<c<1$.  In practice, we will extensively use that, furthermore, the  pairs of the form $(\gamma_n^{\ell}, \gamma_n)_{n\ge 1}$ are averaging for $\ell\!\in \{1,\ldots,\lceil \frac 1 a\rceil-1\}$ so that $a\ell < 1$. 

\medskip

\noindent The rate of convergence of $\nu^{\eta, \gamma}_n(f)$ toward $\nu(f)$ has also been elucidated and reads as follows (when $d=1$ and $\eta_n= \gamma_n$ for the sake of simplicity, keeping in mind that even in that setting, various averaging systems are involved): \smallskip

\noindent Set  $\Gamma^{(2)}_n = \sum_{k=1}^n \gamma_k^2$, $n\ge 1$. Assume  the Poisson Equation $f-\nu(f) =-{\cal L}g$ has a smooth enough solution and that $\displaystyle \frac{\Gamma_n^{(2)}}{\sqrt{\Gamma_n}} \to \widetilde \beta$, then  
\begin{eqnarray}\label{tclbasic}
\sqrt{\Gamma_n} \big(\nu^{\gamma, \gamma}_n(f) -\nu(f)\big) &\stackrel{(\ER)}{\longrightarrow} &{\cal N}\Big({\widetilde \beta} \nu(\Psi_2);\sigma_1^2(f) \Big)\quad \mbox{ if } \widetilde \beta \!\in [0,+\infty),\\
\frac{\Gamma_n}{\Gamma^{(2)}_n}\big(\nu^{\gamma, \gamma}_n(f) -\nu(f)\big) &\stackrel{a.s.}{\longrightarrow} &\nu(\Psi_2)\hskip 2,4cm  \mbox{ if } \widetilde \beta  =+\infty\label{tclbasic2}
\end{eqnarray}

\noindent with $\sigma_1^2(f)=\nu(| \sigma^*\nabla g|^2 \Big)=-2\nu(g.L g)$ and 
$$
\Psi_2(x):=\frac{1}{2} D^2g(x)b(x)^{\otimes 2}+\frac{1}{24} \ES[D^{(4)} g(x)(\sigma(x)U)^{\otimes 4}],\quad U\sim{\cal N}(0,I_q).
$$
When $\gamma_n = n^{-a}$ the unbiased $CLT$ ($\widetilde \beta =0$) holds for $a\!\in \big(\frac 13, 1\big]$,  the biased $CLT$  for $a = \frac 13$ and the biased convergence in probability for $a\!\in (0,\frac 13)$.

On can interpret this result as follows: if $(\gamma_n)$ decreases to $0$ fast enough ($\widetilde \beta =0$), the empirical measures $\nu^{\gamma,\gamma}_n$ behaves like the empirical measures $\nu_t$ of the diffusion. When  $(\gamma_n)$ goes to $0$ too slowly, there is a discretization effect which slows down the convergence of the empirical measure at rate $\frac{\Gamma_n}{\Gamma^{(2)}_n}$. The convergence then holds $a.s.$ (or at least in probability)  which confirms that what slows down the convergence is a bias term whose rate of decay is lower than $1/\sqrt{\Gamma_n}$. The top rate of convergence is obtained with a biased $CLT$.

We will see in Theorem~\ref{theo:CLT}  further on that, in fact,   there are many of these bias terms   which go to $0$ slower than the $CLT$ rate   for slowly decreasing steps. So killing these terms is a major issue to speed up such ergodic simulations (or Langevin Monte Carlo method) compared to  the regular Monte Carlo method.

\medskip 
The Multilevel paradigm has been  introduced by M. Giles in the  late 2000's  (2008, see~\cite{GIL}). Ever since, it  has been extensively adapted to various types of simulations (nested Monte Carlo, see~\cite{LEPA}, stochastic approximation~\cite{Frikha}) and dynamics (L\'evy driven diffusion, random maps, etc)  as a bias killer. The principle is  the following: assume that a quantity of interest to be computed does have  a representation as an expectation, say  $\E\, Y_0$, but that  the random variable  $Y_0$ cannot be simulated at a reasonable computational cost. Then one usually  approximates $Y_0$ by a family $(Y_h)_{h>0}$ of random vectors that can be simulated with a low complexity, relying on discretization schemes of the underlying dynamics. The typical situation is  the $Y_0 = f(X_T)$ or $F\big((X_t)_{t\in0,T]})$ where $(X_t)$ is a Brownian diffusion as above  and $Y_h = 
f(\bar X^n_T)$ or  $F\big((\bar X^n_t)_{t\in0,T]})$ where $(\bar X_t)_{t\in [0,T]}$ is a discretization scheme, say an  Euler or a Milstein scheme with step $h= \frac Tn\in {\mathbb H} = \big\{\frac{T}{m}, \, m\!\in \EN^*\big\}$. A multilevel estimator with depth $L\!\in \EN^*$  of $\E Y_0$ is designed by implementing a non-homogeneous Multilevel Monte Carlo (MLMC) estimator of size $N\!\in \EN^*$ of the form
\[
\frac{1}{N_1}\sum_{k=1}^{N_1}Y^{(1),k}_{{\mathbf h}} + \sum_{\ell =2}^L \frac{1}{N_{\ell}}\sum_{k=1}^{N_{\ell}}Y^{(\ell),k}_{\frac{{\mathbf h}}{M^{\ell-1}}}-Y^{(\ell),k}_{\frac{{\mathbf h}}{M^{\ell-2}}}
\]
where ${\mathbf h}\!\in \mathbb H$ is a fixed {\em coarse} step, $\big((Y^{(\ell),k}_h)_{h \in \mathbb H})\big)_{\ell=1,\ldots, L, k\ge 0}$ are independent copies of $(Y_h)_{h \in\mathbb H}$, $M\ge 2$,  is a  fixed integer and $N_1, \ldots,N_{_R}$ is an appropriate (optimized) allocation {\em policy} of the simulated paths across the levels $\ell$ such that $N_1+\cdots+N_{_R}=N$ (in practice, at a given level $\ell$, only $Y^{(\ell)}_{\frac{{\mathbf h}}{M^{\ell-1}}}$ and $Y^{(\ell)}_{\frac{{\mathbf h}}{M^{\ell-2}}}$ have to be simulated). The level $\ell=1$ is the {\em coarse} level whereas the levels $\ell\ge 2$ are the {\em refined} levels. Within a refined given level $\ell$, $Y^{(\ell),k}_{\frac{h}{M^{\ell-2}}} $ denotes the coarse scheme and $Y^{(\ell),k}_{\frac{h}{M^{\ell-1}}}$ the refined scheme. For some fixed $k$ and $\ell$, the random variables are ``consistent" in the sense that they have been simulated from the same underlying Brownian motion $W^{(\ell)}$. A   quantitative translation of this consistency is that $Y_h$ converges in (squared) quadratic norm to $Y_0$ at a $h^{\beta}$ rate, namely $\|Y_h-Y_{0}\|^2_2 \le V_1 |h|^{\beta} $, $h\!\in \mathbb H$. The parameter $\beta$ depends on $f$ or $F$ in a diffusion framework.  If $f$ or $F$  are locally Lipschitz continuous with polynomial growth (with respect to the sup norm as for $F$), $\beta = 1$. This parameter $\beta $ and the constant $V_1$ are key parameters to optimize the allocations of the paths to the various levels  (see~\cite{GIL, LEPA}).


Among other results,  M. Giles proved that if $\alpha=1$ and $\beta =1$ --~which is the standard situation in a diffusion discretized by its Euler scheme~-- when $Y_0= f(X_{_T})$, $Y_h= f(\bar X^n_{_T})$ (Euler scheme with step $h= \frac Tn$), $f$, 
$b$, $\sigma$ smooth enough (or  $\sigma$ uniformly elliptic if $f$ is simply Borel and bounded), the resulting complexity of  the optimized Multilevel Monte Carlo estimator  to attain a prescribed Mean Squared Error $\varepsilon^2$ behaves like $O\Big(\Big(\log(1/\varepsilon)/\varepsilon\Big)^{2}\Big)$ as $\varepsilon\to 0$. When $\beta >1$ (fast strong approximation like with the Milstein scheme), this rates attains $O\big(\varepsilon^{-2}\big)$ $i.e.$  the rate  of a (virtual)  unbiased simulation. The case $\beta<1$ provides even better improvements compared to a crude Monte Carlo simulation.

\smallskip
In a recent paper (see~\cite{LEPA}) a weighted version of  the above multilevel estimator has been devised to take advantage of a higher order expansion of the weak error (bias expansion) up to an order $R\!\in \EN^*$, namely
\[
\E \,Y_h = \sum_{r=1}^R c_r h^{\alpha r} + O\big( h^{\alpha (R+1)} \big),
\] 
still under the above quadratic convergence rate assumption. Then,  the so-called {\em Multilevel Richardson-Romberg estimator} (ML2R in short)  is still based on  the simulation of independent copies of 
$(Y_h)_{h \in\mathbb H}$ and reads
\[
\frac{\mathbf{W}^{(R)}_1}{N_1}\sum_{k=1}^{N_1}Y^{(1),k}_{{\mathbf h}} + \sum_{r =2}^R \frac{\mathbf{W}^{(R)}_r}{N_{r}}\sum_{k=1}^{N_{r}}Y^{(r),k}_{\frac{{\mathbf h}}{M^{r-1}}}-Y^{(r),k}_{\frac{h}{M^{r-2}}}
\]
where  the  $R$-tuple  $(\mathbf{W}^{(R)}_r)_{1\le r\le R}$ of weights has a closed form {\em entirely  determined by   $\alpha$,  and $M$} and not on $(Y_h)_{h\ge 0}$ (that means on the specific form of $f$, $b$, $\sigma$ in a diffusion framework).  For this weighted estimator,  the complexity is reduced {\em mutatis mutandis}  to $O\Big( \log(1/\varepsilon)/\varepsilon^2 \Big)$ in the setting $\beta =1$. When $\beta<1$ this   estimator  dramatically outperforms the above ``regular" multilevel method since it only differs from a (virtual) unbiased simulation by  a factor $\exp^{-\frac {1-\beta}{\alpha}\sqrt{\log(2) \log(1/\varepsilon)/2}}$   (when $M=2$) instead of $\varepsilon^{\frac{1-\beta}{\alpha}}$ with MLMC. The underlying idea for this weighted Multilevel method is to combine the multilevel paradigm with  a multistep Richardson-Romberg extrapolation introduced in~\cite{PAG}, hence its name.
We refer to~\cite{LEPA} for more precise results and proofs.

\medskip The aim of this paper is to transpose the weighted multilevel paradigm to the Langevin Monte Carlo simulation with decreasing step described above, with the issue that, in contrast with regular Monte Carlo simulation,  canceling the bias terms directly impacts the rate of convergence of the method by enlarging the range of step parameters for which a $CLT$ holds at rate $\sqrt{\Gamma_n}$ to coarser steps (so that $\Gamma_n$   goes faster to infinity where the  stationary regime takes places). So we will adapt the ML2R estimator to the occupation measure $\nu_n^{\gamma}= \nu_n^{\gamma,\gamma}$ introduced before. Like  in the regular  Monte Carlo setting, we introduce, for a function $f$,  a weighted estimator involving $\nu_n(f)$ and some correcting terms denoted by $\mu_n^{(r,M)}(f)$, $r=1,\ldots,R$  based on some pairs of coupled refined schemes (see~\eqref{eq:defcorterm} for details). Since the ergodic estimation of the invariant measure is based on only one path, the idea here  is to replace the allocation policy of realizations $N_1,\ldots,N_R$ of  the ML2R method by a {\em sizing}  policy $q_1,\ldots q_{_R}$ of the length of the coarse path (involved in $\nu_n(f)$) and those of the correcting sequences $\mu_n^{(r,M)}(f)$. 

\smallskip
In order to asymptotically kill the successive terms of the bias induced by the estimator, we will need some asymptotic expansions of the  $\nu_n$ and $\mu_n^{(r,M)}$ such as~\eqref{eq:nun11} and~\eqref{eq:tildenunrr} below. These expansions, which require the invertibility of the infinitesimal generator (or equivalently the existence of solutions to the Poisson equation) can be viewed as  the counterpart of the classical weak error/bias expansion $\ES[f(X_T)]-\ES[f(\bar{X}_T)]$ in finite horizon.  Concerning the strong convergence rate property which leads to the control of the variance of the corrective terms in the standard Multilevel method, its counterpart in our ergodic setting is the following mean confluence result which says that
 \[
 \frac{1}{\Gamma^{(2)}_n}\sum_{k=1}^n \gamma_k|X_{\Gamma_k}-\bar X_k|^2 \longrightarrow 0\quad a.s. \quad \mbox{ as } n\to +\infty.
 \] 
It  says that the $(\gamma^2,\gamma)$-empirical measure of the couple $(X,\bar X)$ concentrates on the diagonal of $\ER^d$ at rate $o\Big(\frac{\Gamma_n}{\Gamma_n^{(2)}}\Big)$. Such a property holds $e.g.$ when the diffusion itself is exponentially confluent (typically a mean-reverting Ornstein-Uhlenbeck process) under an exponential confluence property  which holds under $(\mathbf{S})$.  

 Throughout the proofs, we will work in one dimension for notational convenience. The extension to the multidimensional case  would only   generate  technicalities.\medskip 

\textbf{Outline.} The paper is organized as follows. We begin by introducing precisely the weighted empirical sequence built for the estimation of the invariant measure, called \textbf{{\bf ML2Rgodic}} and denoted by $\widetilde{\nu}_n^{R,{\bf W}}$. Then, our main results are divided in three parts. In Theorem~\ref{theo:CLT}, we obtain some CLTs for $\widetilde{\nu}_n^{R,{\bf W}}$: we show that the {\bf ML2Rgodic}-Algorithm with $R-1$ levels of corrections and an appropriate sequence $(\gamma_n)$ has an optimal rate of order $n^{\frac{R}{2R+1}}$ with an  asymptotic variance which is the same as the one of the original procedure. Then, in view of the optimization of the choices of the parameters, we exhibit in Theorem~\ref{L2theo} some first and second order asymptotic expansions of the Mean-Squared Error. Based on this result, we proceed to the optimization in Theorem~\ref{thm:optimiz} and provide some choices of the parameters involved by the algorithm which lead to a complexity of order 
$\varepsilon^{-2}\log(\frac{1}{\varepsilon})$ (instead of $\varepsilon^{-3}$ for the original procedure).  The main tools for the establishment of Theorems~\ref{theo:CLT} and~\ref{L2theo} appear in Sections~\ref{sec:proof1} and~\ref{sec:proof2}. Then,  the proofs of Theorems~\ref{theo:CLT},~\ref{L2theo} and~\ref{thm:optimiz} are achieved in Section~\ref{sec:proof3}. Finally, we end this paper by some numerical computations in Section~\ref{sec:numerics}.

\section{The Multilevel-Romberg Ergodic ({\bf ML2Rgodic}) procedure}
\subsection{Design of the {\bf ML2Rgodic} Langevin estimator}  We aim at adapting the multilevel paradigm to devise an ergodic estimator for the approximation of the invariant distribution. For a given integer $R\ge 2$, the idea is to modify the original procedure with the aim to kill the $R$ first terms of the expansion  of the discretization error without impacting too much  the simulation cost of simulation. 

Let $\gamma = (\gamma_n)_{n\ge 1}$ be a sequence of steps, and $M$ and $R$ be two integers such that $R\ge 2$ and $M\ge2$. First we consider an Euler scheme $\bar X^{(1)}=\bar X$ with decreasing step $\gamma$ associated to a standard Brownian motion $W^{(0)}= W$. We associate to this scheme  $R-1$ {\em independent}  coupled schemes $(\bar X^{(r)}, \,\bar Y^{(r,M)})$, $r=2,\ldots,R$, independent of $\bar X^{(0)}$   where 
\begin{itemize}
\item $\bar X^{(r)}$ is an Euler scheme with decreasing step $\gamma^{(r,M)}= \frac{\gamma}{M^{r-2}}$ (so that $\gamma^{(2,M)} = \gamma$) associated to a Brownian motion $W^{(r)}$.
\item $\bar Y^{(r,M)}$ is a  refined Euler scheme with decreasing step $\widetilde\gamma^{(r,M)}$ associated {\em  to the same}  Brownian motion $W^{(r)}$ where
\begin{equation}\label{eq:gamtildebb}
\forall \, m\in \{1,\ldots,M\},\quad \widetilde{\gamma}^{(r,M)}_{M(n-1)+m} = \frac{\gamma^{(r,M)}_n}{M}= \frac{\gamma_n}{M^{r-2}}, \; n\ge 1.
\end{equation}
\end{itemize}
Set, for every integers $\ell\ge 1$ and $r\ge 2$,
\bqn\label{eq:Gammalr}
\Gamma^{(\ell,r)}_n = \sum_{k=1}^n (\gamma^{(r,M)}_k)^{\ell}= M^{-(r-2)\ell} \sum_{k=1}^n \gamma^{\ell}_k =   M^{-(r-2)\ell}\Gamma^{(\ell)}_n
\eqn
where  $\Gamma_n^{(\ell)}=\Gamma_n^{(\ell,2)}=\sum_{k=1}^n\gamma_k^\ell$. Note that $\Gamma_n^{(\ell)} = \Gamma_n^{(\ell,2)}$. 

Then, we define for every $r=2,\ldots,R$ the sequence of  difference of the empirical measures of the  two schemes by
\begin{align}
\nonumber
 \mu^{(r,M)}_n(dx)&= \frac{1}{\Gamma^{(1,r)}_n} \sum_{k=1}^n \left(\left(\sum_{m=0}^{M-1}\widetilde \gamma^{(r)}_{M(k-1)+m} \delta_{\bar Y^{(r)}_{M(k-1)+m}}\right)- \gamma^{(r)}_k\delta_{\bar X^{(r)}_{k-1}}\right),\; n\ge 1,\\
 &=\frac{1}{ \Gamma^{(1,r)}_n} \sum_{k=1}^n \frac{\gamma_n}{M^{r-2}} \left(\frac{1}{M}\sum_{m=0}^{M-1}\delta_{\bar Y^{(r)}_{M(k-1)+m}}-\delta_{\bar X^{(r)}_{k-1}}\right),\; n\ge 1.\label{eq:defcorterm}
  \end{align}
The expected weak limit of $\mu^{(r,M)}_n(f)$ is $0$ as a difference of occupation measures of two Euler schemes with decreasing step. Thus, this empirical measure plays the role of a correcting term.

\smallskip  
Now, let $q_1,\ldots,q_R$ denote some positive real numbers, called {\em re-sizers} from now on,  satisfying
$$
\forall r\in\{1,\ldots,R\},\quad 0<q_r<1,\quad q_1+\ldots+q_R=1,
$$
and, for a given integer $n\ge 1$, set 
$$
n_r=\lfloor q_r n\rfloor, \; r=1,\ldots,R.
$$
 Let $f:\ER^d\rightarrow\ER$ be a smooth function, coboundary for the infinitesimal generator ${\cal L}$ (existence of solutions to the Poisson equation $f-\nu(f)= {\cal L}(g))$. Under some appropriate assumptions (including weak confluence)  we can prove in a sense made precise later on (see Propositions~\ref {prop:Multistep}$(b)$ and~\ref{prop:dvpcorrectiveterms}$(b)$) that the sequences $(\nu_{n_1}(f))_{n\ge 1}$ and 
$(\mu_{n_r}^{(r,M)}(f))_{n\ge 1}$ satisfy the following asymptotic generic type-expansions:

\begin{eqnarray}\label{eq:nun11}
\nu_{n_1}(f) &=& \nu(f) + \sum_{\ell=2}^{R+1} \frac{\Gamma_{n_1}^{(\ell)}}{\Gamma_{n_1}}\nu(\Psi_\ell) + \frac{{\mathbf{M}_n}}{\Gamma_n}+ o\Big(\frac{1}{\sqrt{\Gamma_{n_1}}}\wedge\frac{\Gamma_{n_1}^{(R+1)}}{\Gamma_{n_1}}\Big)\\
\label{eq:tildenunrr} \mu^{(r,M)}_{n_r}(f) & = &  \sum_{\ell=2}^{R+1}M^{(r-2)(1-\ell)}(M^{1-\ell}-1)\frac{\Gamma^{(\ell)}_{n_r}}{\Gamma_{n_r}}\nu(\Psi_\ell)+   o\Big(\frac{1}{\sqrt{\Gamma_{n_r}}}\wedge\frac{\Gamma_{n_r}^{(R+1)}}{\Gamma_{n_r}}\Big),
\end{eqnarray}
where $(\mathbf{M}_n)_{n\ge1}$ is a martingale and $(\Psi_\ell)_{\ell\ge 1}$ is  a sequence of functions made precise further on. At this stage, the reader can remark that there is no martingale term in the main part of the second expansion. This point, which is strongly linked with the weak confluence assumption $\mathbf{(C_w)}$ introduced below,  can be understood as follows: the martingale term induced by 
$\mu_n^{(r,M)}$ is asymptotically negligible against the one of $\nu_{n_1}(f)$.  In a rough sense, this means that if we build an appropriate combination of 
$\nu_{n_1}(f)$ and $\mu^{(r,M)}_{n_r}(f)$, $r=1,\ldots,R$, we will be able to kill the bias error without growing the asymptotic variance. But a numerical computation holds in a finite (non-asymptotic) setting so that this heuristic needs to be refined in practice. One of the objectives of the paper is thus to go deeper in the study of the expansion in order to be able to propose an efficient and potentially optimized method of approximation of the invariant distribution.

\medskip
\noindent \textbf{The {\bf ML2Rgodic}-algorithm:} As mentioned before, the first step toward our {\bf ML2Rgodic} estimator is to design  an appropriate combination of the formerly defined empirical measures in order to ``kill'' the bias. Furthermore, we require that this combination  does not depend  upon the size  $n$ of the estimator. We thus define a sequence of empirical measures denoted by $(\widetilde{\nu}_n^{(R,W)})_{n\ge1}$  by:
\begin{equation}\label{defML2Rgodic}
\widetilde{\nu}_n^{(R,{\bf W})}={\bf W}_1 \nu_{n_1}+\sum_{r=2}^R {\bf W}_r \mu_{n_r}^{(r,M)},\quad n\ge1,
\end{equation}
where ${\bf W}=({\bf W}_r)_{r=1}^R$ is a sequence of real numbers. For the sake of simplicity, we do not mention the dependency  of $\widetilde{\nu}_n^{(R,{\bf W})}$in $M$ and $\gamma$. Also, let us remark that the weights  ${\bf W}_r$ clearly depend on $R$ and will sometimes be  denoted ${\bf W}_r^{(R)}$ in order to recall this dependence when necessary. Let us now specify ${\bf W}$. First, by~\eqref{eq:nun11} and~\eqref{eq:tildenunrr}, one remarks that it is necessary to assume that  ${\bf W}_1=1$ in order to ensure the convergence towards $\nu$.

\medskip
Let us now consider the construction of ${\bf W}_2,\ldots,{\bf W}_R$. To this end, we   consider  from now on   step sequences with polynomial decay 
\begin{equation}\label{eq:step}
\gamma_k= \gamzero  k^{-a} \; \mbox{ with }\quad \gamzero >0,\; a\in(0,1).
\end{equation}
Then by plugging the expansions of the bias resulting from~\eqref{eq:nun11} and~\eqref{eq:tildenunrr} in the definition~\eqref{defML2Rgodic} of the  {\bf ML2Rgodic} estimator we derive that 
\begin{align*} 
\E&\big(\widetilde\nu^{(R,\mathbf W)}_n\big) =  {\mathbf W} _1\E\nu_{n_1}(f) + \sum_{r=2}^{R} \mathbf W_r\E\, \mu_{n_r}^{(r,M)}(f)\\
&=  {\mathbf W} _1 \nu (f) + \sum_{\ell=2}^{R+1} \left[  {\mathbf W} _1  \frac{\Gamma_{n_1}^{(\ell)}}{\Gamma_{n_1}}+\sum_{r=2}^R 
\mathbf W_r M^{(r-2)(1-\ell)}\big(M^{1-\ell}-1\big)  \frac{\Gamma_{n_r}^{(\ell)}}{\Gamma_{n_r}}
\right] \nu(\Psi_\ell) 
+o\Big(\frac{\Gamma^{(R+1)}_n}{\Gamma_n}\Big)  \\
&\approx {\mathbf W} _1 \nu (f) + \sum_{\ell=2}^{R+1} \frac{\Gamma_{n}^{(\ell)}}{\Gamma_{n}} \nu(\Psi_\ell) \left[  {\mathbf W} _1 q_1^{-a(\ell-1)}+\sum_{r=2}^R 
\mathbf W_r M^{(r-2)(1-\ell)}\big(M^{1-\ell}-1\big)   q_r^{-a(\ell-1)}
\right]\\
&\qquad +o\Big(\frac{\Gamma^{(R+1)}_n}{\Gamma_n}\Big)
\end{align*}
where the notation $\approx$ is used to keep in mind that one implicitly assumes that  
$   \frac{\Gamma_{n_r}^{(\ell)}}{\Gamma_{n_r}}-q_r^{-a(\ell-1)}  \frac{\Gamma_{n}^{(\ell)}}{\Gamma_{n}}$ is negligible (see further on the proof of Theorems~\ref{theo:CLT} and~\ref{L2theo}). Then as soon as the weights $(\mathbf{W}_r)_{1\le r\le R}$ are solutions to the linear system
\begin{equation}\label{eq:theSystem}
{\mathbf W} _1= 1,\quad {\mathbf W} _1 q_1^{-a(\ell-1)} +(M^{1-\ell}-1)\sum_{r=2}^R {\mathbf W} _r M^{-(r-2)(\ell-1)}q_r ^{-a(\ell-1)}= 0,\;\quad\ell=2,\ldots,R,
\end{equation}
the bias is ``killed'' up to order $R$ and reads
\[
\E\big(\widetilde\nu^{(R,\mathbf W)}_n\big) \approx  \frac{ 1-a}{1-a(R+1) }\,\gamzero ^R \,\nu(\Psi_{R+1}) \widetilde{\mathbf W}_{_{R+1}}n^{-aR} +o\Big(\frac{\Gamma^{(R+1)}_n}{\Gamma_n}\Big)
\]
where we set, more generally,
\begin{equation}\label{eq:Wtilde}
\widetilde{{\mathbf W}}_{_{R+i}}= {\mathbf W} _1 q_1^{-a(R+i)} +( M^{-R-i+1}-1)\sum_{r=2}^R {\mathbf W} _r  M^{-(r-2)(R+i-1)}q_r ^{-a(R+i)},\,i\ge 0.
\end{equation}
 The main difference at this stage with the regular weighted Multilevel estimator is that these weights {\em depend on} the re-sizers $q_r$ which will make a complete optimization  of these allocation parameters out of reach. 

In the following lemma  the linear system~\eqref{eq:theSystem} is solved. In short, it shows that  the  weights  are uniquely defined  provided the re-sizers $q_r$  satisfy $\frac{q_r}{ M^{r/a}}\neq \frac{q_s}{ M^{s/a}}$, $s\neq r$. Note that these weights depend on  the exponent $a$ (and the $(q_r)$) but not on $\gamzero $.

Another important point is that, by contrast with the regular weighted Multilevel Monte Carlo setting, this system in its general form {\em is not a regular Vandermonde system} though it shows some similarities. In fact it can be related to a sequence of $(R-1)\times(R-1)$-Vandermonde systems with closed solutions.  A notable exception to this situation occurs in the very special of  {\em uniform} re-sizers    $q_r = \frac 1R$, $r=1, \ldots,R$ where 
%
  we retrieve exactly the weights of the  regular Monte Carlo $ML2R$ introduced in~\cite{LEPA}.  For a given depth $R>1$, the closed form of  $({\bf W}_i)_{i=2}^R$ (keeping in mind that ${\bf W}_1=1$) is given by the following lemma.
  

\begin{lemme}\label{lem:WWtilde}  $(a)$ {\em General re-sizers}: If ${\bf q}:=(q_1,\ldots,q_R)\!\in {\cal S}_{_R}:=\big\{(x_1,\ldots,x_R)\!\in(0,+\infty)^R, \sum_{i=1}^R x_i=1\big\}$ and satisfies $\frac{q_r}{ M^{r/a}}\neq \frac{q_s}{M^{s/a}}$, $s\neq r$, then the above system~\eqref{eq:theSystem} has a unique solution given by

\begin{equation}\label{eq:Wr} {\mathbf W}^{(R)} _r = M^{r-2}\left(\frac{q_r}{q_{1}}\right)^a\sum_{k\ge 0} \frac{1}{M^k}    \prod_{s=2,s\neq r}^R\frac{1-M^{s-2-k}(q_s/q_1)^a}{1-M^{s-r}(q_s/q_r)^a}, \; r=2,\ldots,R.  
\end{equation}

Moreover, the coefficients $\widetilde{{\mathbf W}}^{(R)}_{_{R+i}}$, $i=1,2$, as defined in~\eqref{eq:Wtilde} read 
\begin{equation}\label{eq:Wtilde1}
\widetilde{{\mathbf W}}^{(R)}_{_{R+1}}= \frac{(1- M^{-R})}{q_1^{aR} }\sum_{k\ge 0}\frac{1}{ M^{kR}} \prod_{r=0}^{R-2}\Big(1- M^{k-r}\Big(\frac{q_1}{q_{r+2}}\Big)^a\Big).
\end{equation}
and
\begin{equation}\label{eq:Wtilde2}
\widetilde {\mathbf W}^{(R)}_{_{R+2}} = \frac{(1-M^{-(R+1)})}{q_1^{a(R+1)}}\sum_{k\ge 0} \frac{1}{ M^{k(R+1)}}\Big(1+\sum_{r=0}^{R-2}M^{k-r}\Big(\frac{q_1}{q_{r+2}}\Big)^a\Big) \prod_{r=0}^{R-2}\Big(1- M^{k-r}\Big(\frac{q_1}{q_{r+2}}\Big)^a\Big).
\end{equation}

\noindent $(b)$ {\em Uniform re-sizers}: If $q_r =\frac 1R$, $r=1,\ldots,R$,
the following simpler closed form holds for  the weights $\mathbf{W}^{(R)}_r$:
\begin{equation}\label{eq:quniform0}
\mathbf{W}^{(R)}_r =  \mathbf{w}^{(R)}_r+\cdots+ \mathbf{w}^{(R)}_{_R}, \; r=1,\ldots,R 
\end{equation}
with 
\begin{equation}\label{eq:quniform}
 \mathbf{w}^{(R)}_r = \prod_{s=1, s\neq r}^R\frac{ M^{-(s-1)}}{ M^{-(s-1)}- M^{-(r-1)}}=  \prod_{s=1, s\neq r}^R\frac{1}{1- M^{s-r}}, \; r=1,\ldots,R. 
\end{equation}
These  weights $( \mathbf{W}^{R}_r )_{r=1,\ldots,R}$, $R\ge 1$, are bounded $i.e.$ $\sup_{r=1,\ldots,R, R\ge 1}| \mathbf{W}^{(R)}_r | <+\infty$. Furthermore
\begin{equation}\label{eq:quniformtilde}
\widetilde {\mathbf W}^{(R)}_{_{R+1}}
=  (-1)^{R-1}  R^{aR}  M^{-\frac{R(R-1)}{2}}\quad\mbox{ and }\quad 
\widetilde {\mathbf W}^{(R)}_{_{R+2}}
=  (-1)^{R}  R^{a(R+1)}  M^{-\frac{R(R-1)}{2}}\frac{1-M^R}{1-M^{-1}}.
\end{equation}
\end{lemme}
\medskip
 The proof is postponed to  Appendix~\ref{App:A}.

\bigskip
\noindent {\bf Examples.} $\bullet$ $R=2$: ${\bf W}^{(2)}_1=1$, ${\bf W}^{(2)}_2=\frac{M}{M-1}\left(\frac{q_2}{q_1}\right)^a$

\smallskip
\noindent  $\bullet$  $R=3$: 
$$
({\bf W}^{(3)}_2,{\bf W}^{(3)}_3)=\frac{M}{M-1}\left(\left(\frac{q_2}{q_1}\right)^a\frac{1-\frac{M^2}{M+1}\left(\frac{q_3}{q_1}\right)^a}{1-M \left(\frac{q_3}{q_2}\right)^a}, \left(\frac{q_3}{q_1}\right)^a\frac{1-\frac{M^2}{M+1}\left(\frac{q_2}{q_1}\right)^a}{1-M^{-1}\left(\frac{q_2}{q_3}\right)^a}\right).
$$

When there is no ambiguity the superscript $^{(R)}$ will be dropped in the notations $\mathbf W^{(R)}$,   $\mathbf{w}^{(R)}_r$ and ${\mathbf W}^{(R)}_{_{R+1}}$. In the sequel, $\widetilde{\nu}_n^{(R,{\bf W})}$ will be always defined with ${\bf W}$ satisfying~\eqref{eq:theSystem} or~\eqref{eq:Wr}.

%


\medskip
\noindent \textbf{Assumptions.} We introduce below the assumptions for the first theorem. As recalled in the introduction, the study of the rate of convergence brings into play
the Poisson equation related to the SDE. In this paper where we are going deeper in the expansion of the error, we will need to use it successively. For the sake of simplicity, we thus assume the following (strong) assumption: 

\medskip
\noindent $\mathbf{(P)}:$ For every ${\cal C}^\infty$ function $f$, there exists a unique ({up to an additive constant}) ${\cal C}^\infty$-function $g$,  such that 
$f-\nu(f)=-{\cal L}g$. Furthermore, if $f$ is a function with polynomial growth, then $g$ also is.

\smallskip {For instance, it can be shown that, when $\sigma$ is bounded and uniformly elliptic (in the sense that $(\sigma\sigma^*(x)x|x)\ge \lambda_0|x|^2$ for some  $\lambda_0>0$), when Assumption $\mathbf{(S)}$ is in force and $f$, $b$ and $\sigma$ are smooth  have polynomial growth as well as  their derivatives,  then $\mathbf{(P)}$ holds true.  Actually, we first recall that under the ellipticity and Lyapunov assumptions, the semi-group converges exponentially fast towards $\nu$ (in total variation) so that  $g(x)=\int_0^{\infty} P_s f(x)-\nu(f) ds$ is well-defined and it is classical background that $g$ is  the {unique (up to a constant)} solution to the Poisson equation $f-\nu(f)=-{\cal L}g$ (see $e.g.$~\cite{Parver1}). Then, by~\cite[Theorem 6.17]{gilbarg}, under uniform ellipticity, $g$ is in fact ${\cal C}^\infty$ as soon as $f$, $b$ and $\sigma$ are. The polynomial growth  of  $g$ and $\nabla g$ has been proved in~\cite[Theorem 1]{Parver1}. The property is obtained through the a priori estimate, see Equation~(9.40) in~\cite{gilbarg}, which in fact also holds for $D^2g$. Then, we can  establish by induction that all  the partial derivatives of $g$  have a  polynomial growth. Assume it is true up to order $k$. First note that $u=\partial_{i_1,\ldots,i_{k-1}} g$ is a solution to ${\cal L} u=-f_{g}$ where $f_{g}$ is a function which depends on $f$, $b$ and $\sigma$ and their first order partial derivatives and some derivatives of $g$ up to order $k$.  Hence,  $f_g$ has polynomial growth and the {\em a priori} error bound (9.40) in~\cite{gilbarg} for the second order partial derivatives of  $u$ yields the polynomial growth of the partial derivatives $\partial_{i_1,\ldots,i_{k+1}} g$.}



\smallskip
The second additional assumption has been introduced  in~\cite{PP3} and deeply studied~\cite{LPPIHP}: it requires the diffusion to be \textit{weakly confluent}, $i.e.$ that two paths of the diffusion, with different initial values, but driven by the same Brownian motion, asymptotically cluster in a weak (or statistical) sense as follows: let $(X_t,Y_t)_{t\ge0}$ be the \textit{duplicated diffusion} (or {\em two-point motion}) associated with the diffusion ($SDE$) by 
\begin{equation}\label{eq:dupdif2}
\begin{cases}
d X_t= b(X_t) dt+\sigma(X_t) dW_t\\
dY_t = b(Y_t) dt+\sigma(Y_t) dW_t, 
\end{cases}
\end{equation}
where $X_0,\,Y_0$ are two starting values independent of $W$. If $\nu$ is  an invariant distribution for ($SDE$), $\nu_{\Delta}:=\nu\circ (x\mapsto(x,x))^{-1}$ is trivially invariant for the couple $(X,Y)$.  The  diffusion ($SDE)$ is said {\em weakly confluent}  if $\nu_{\Delta}$ is {\em the only invariant distribution}  for $(X,Y)$ (which implies implicitly  that $\nu$ itself is the unique invariant distribution of ($SDE$)). In the sequel,  this assumption is referred to as 

\medskip
\noindent $\mathbf{(C_w)}$: ($SDE$) is weakly confluent.

\begin{Remarque} $\rhd$ Under slight additional assumptions on the stability of ($SDE$), it can be shown (see~\cite{LPPIHP}) that, if $\mathbf{(C_w)}$ holds, the diffusion is statistically confluent in the sense that
$$
\frac{1}{t}\int_0^t\delta_{(X_s,Y_s)} ds\stackrel{(\ER^{2d})}{\Longrightarrow }\nu_{\Delta}\quad a.s.\quad\textnormal{as $t\rightarrow+\infty$}.
$$
 
\smallskip
\noindent $\rhd$ For the empirical measure $\widetilde{\nu}_n^{(R,{\bf W})}$,  the role of $\mathbf{(C_w)}$ is to ensure that the empirical measures $\mu_n^{r,L}$, built with  some  differences of 
schemes $\bar{X}_n^{(r)}$ and ${\bar Y}_n^{(r)}$ have a negligible asymptotic variance (with respect to that of $\nu_n$). This property will be made precise in Section~\ref{sec:dominatingmartingales}.
\end{Remarque}
We are now in position to state the first main theorem.

\begin{theorem}[CLT] \label{theo:CLT} Assume $\mathbf{(S)}$, $\mathbf{(P)}$ and $\mathbf{(C_w)}$. Let $(R,M)\!\in ({\mathbb N}^*\setminus\{1\})^2$ and let $(\mathbf W_r)_{1\le r\le R}$ denote the $R$-tuple of weights defined by~\eqref{eq:Wr}. Let  $q=(q_r)_{1\le r\le R}\!\in \cal S_R$ be an $R$-tuple of re-sizers satisfying $\frac{q_r}{M^r}\neq \frac{q_s}{M^s}$, $s\neq r$. Let $\gamma_n=\gamzero  n^{-a}$, $n\!\in \EN^*$, $a\in(0,1/R)$,  be a discretization step sequence.  Let $f:\ER^d\to\ER$ be a ${\cal C}^\infty$-function and denote by $g$ the solution to $f-\nu(f)=-{\cal L}g$. Let ${\bf W}=(\mathbf{W}_r)_{r=1,\ldots,R}$ be  defined by~\eqref{eq:Wr}.  

\smallskip
\noindent  $(a)$ If  $a\!\in \big( \frac{1}{2R+1}, \frac 1R\big)$, then  
\[
n^{\frac{1-a}{2}}\left(\widetilde{\nu}_n^{(R,\bf W)}(f)- \int_{\ER}fd\nu\right)\stackrel{(\ER)}{\Longrightarrow} {\cal N}\Big(0;\sigma_f^2(a,q,R)\Big)\quad \mbox{ as }\quad n\to +\infty
\]
with    
\begin{equation}\label{eq:sigma2f}
\sigma_f^2(a,q,R) = \frac{1-a}{\gamzero }\frac{\sigma_1^2(f)}{q_1^{1-a}}\quad \mbox{ with }\quad \sigma_1^2(f)=\nu(|\sigma^*\nabla g|^2).
\end{equation}

\noindent $(b)$ If $a=\frac{1}{2R+1}$,  the $CLT$ holds at an optimal rate towards a biased Gaussian distribution, namely
\[
n^{\frac{R}{2R+1}}\Big(\widetilde{\nu}_n^{(R,\bf W)}(f)- \int_{\ER^d}fd\nu\Big)\stackrel{(\ER)}{\Longrightarrow} {\cal N}\Big(m_f(q,R);\sigma_f^2(q,R) \Big)\quad \mbox{ as }\quad n\to +\infty
\]
 
 
 \noindent  with $\sigma_f^2(q,R):=\sigma_f^2(\frac{1}{2R+1},q,R)$ and  $m_f(q,R) :=  2 \gamzero ^{R}\widetilde{\mathbf W}_{_{R+1}}c_{R+1}$
where $\widetilde{\mathbf W}_{_{R+1}}$ is given by~\eqref{eq:Wtilde}
and $c_{R+1}=\nu(\Psi_{_{R+1}})$, $\Psi_{_{R+1}}$ being  a ${\cal C}^\infty$-function with polynomial growth (whose explicit expression in the one-dimensional case is given by~\eqref{eq:Phiell}). 

 
\smallskip
\noindent $(c)$ If $a\!\in \big(0,\frac{1}{2R+1}\big)$, then 
$$
n^{a R}\left(\widetilde{\nu}_n^{(R,\bf W)}(f)- \int_{\ER}fd\nu\right)\xrightarrow{\PE}m_f(a,q,R)\quad \textnormal{as } \quad n\rightarrow+\infty
$$
with 
\begin{equation}\label{eq: biaisresiduel}
m_f(a, q,R) :=  \frac{1-a}{1-a(R+1)}\gamzero ^R \widetilde  {\mathbf W}_{_{R+1}}c_{R+1}.
\end{equation}
%
\end{theorem} 

\begin{Remarque} Note that the definitions of $m_f(a,q,R)$ and $m_f(q,R)$ in the above claims $(b)$ and $(c)$ are consistent since $m_f(q,R) = m_f(a,q,R)$ when $a=\frac{1}{2R+1}$. 
\end{Remarque}

\medskip
From an asymptotic point of view, the above result says in particular that when $R$ grows, the optimal rate of convergence tends to $n^{\frac 12}$
without increasing the (asymptotic) variance. However, from a non-asymptotic point of view, one has certainly to go deeper in the result to try to optimize the choice of the parameters. This implies to take into account   the effect of the choice of ${\bf q}$, $M$  and $R$ on the residual bias term, the variance and on  the computational cost. This is the purpose of the next paragraph.

\smallskip
\noindent \textbf{$L^2$-expansions of the error.}
The aim of this part is to study the quadratic error to prepare the optimization of the parameter of the multilevel  estimator ($a,q,R,n$) algorithm subject to a prescribed quadratic error $\varepsilon>0$. To this end, we will not only provide a  re-formulation of  Theorem~\ref{theo:CLT} in quadratic norm,  we will also go deeper in the study of the asymptotic error. In particular, in the previous result, the variance induced by the correcting terms $\mu_n^{R,M}$ does not appear and we would like to quantify it.  We will also need to control the residual error terms not only in $n$ but also with respect to the depth $R$, since this parameter is intended to go to $+\infty$ in the optimization phase. This will lead us to carry out the expansion to the order $R+2$ and not $R$ or $R+1$ like in the above theorem and to introduce a second and more constraining confluence assumption denoted by $\mathbf{(C_s)}$:
 
\medskip
\noindent $\mathbf{(C_s)}:$ There exists $\alpha>0$ and a positive matrix $S$ such that for every $x,y\in\ER^d$,
$$
(b(x)-b(y)|x-y)_{_S}+\frac{1}{2}\|\sigma(x)-\sigma(y)\|^2_{_S}\le -\alpha \|x-y\|^2_{_S}
$$
where  $(\,.\,|\,.\,)_{_S}$ and by $|\,.\,|_{_S}$ stand for the inner product and norm on $\ER^d$ defined   by $(x|y)_{_S}=(x|Sy)$ and $|x|_{_S}^2=(x|x)_{_S}$, and for $A\!\in {\cal M}(d,d,\ER)$, $\| A\|_{_S}^2={\rm Tr}(A^*SA)$.

\smallskip
Furthermore to get closer to practical aspects, we only consider the optimal case $a=\bar a =1/(2R+1)$ which clearly provides the highest possible  rate of convergence for a given complexity. Finally we will focus on the uniform re-sizing vector $q_r = \frac 1R$, $r=1,\ldots,R$. They turn out to be most likely rate optimal and, as emphasized in Remark~\ref{Rem:5.10},   in that case the first term of the bias of the {\bf ML2Rgodic}  estimator  {\em does vanish} whereas for other choices of vectors $q$ a residual bias (at rate $O(n^{-1-\bar a}))$  still remains. Though theoretically negligible, it turns out to have a strong numerical impact on simulations. 

%
%

\begin{theorem}[Mean Squared Error for $a=\bar a =  \frac{1}{2R+1}$]\label{L2theo} $(a)$ Suppose that the assumptions of the previous theorem hold and let $a=\frac{1}{2R+1}$. Then, 
%
%
%
%
%
$$
\big\|\widetilde{\nu}^{(R,{\bf W})}_n(f)-\nu(f)\big\|_2^2=n^{-\frac{2R}{2R+1}}\left(\sigma_f^2(q,R)+ m_f^2(q,R)+o(1)\right)\; \textnormal{ as }\; n\rightarrow+\infty.$$

\smallskip
\noindent $(b)$ If, furthermore,  $\mathbf{(C_s)}$ holds 
$$
\big\|\widetilde{\nu}^{(R,{\bf W})}_n(f)-\nu(f)\big\|_2^2=n^{-\frac{2R}{2R+1}}\left(\sigma_f^2(q,R)+ m_f^2(q,R)\right)
+\frac{1}{n}\left(\widetilde{\sigma}^2_f(q,R)+\widetilde m_f (q,R)+o(1)\right) \textnormal{ as }\; n\rightarrow+\infty
$$
where, on the one hand 
\begin{equation}\label{eq:contribsecondordreterme}
\widetilde{\sigma}^2_f(q,R)=\frac{1}{q_1}\sigma_{2,1}^2(f)+\left(1-\frac{1}{M}\right)\Psi(R,M)\sigma_{2,2}^2(f)
\end{equation}
with  
\begin{equation}\label{eq:PsiRM}
\Psi(R,M)= \frac{4R^2}{4R^2-1}\sum_{r=2}^R (\mathbf{W}^{(R)}_r)^2
\end{equation}
and $\sigma_{2,1}^2(f)$ and $\sigma_{2,2}^2(f)$ are  some variance terms    explicitly defined  further on by~\eqref{eq:contribsecondordreterme1} and~\eqref{eq:contribsecondordreterme2} in Propositions~\ref{prop:contribmartnun} and~\ref{prop:contribmartmun} respectively. On the other hand $\widetilde m_f(q,R)$ is given by 
\[
\widetilde m_f(q,R)=\frac{8R}{R-1}c_{_{R+1}} c_{_{R+2}}  \gamzero ^{2R+1}\widetilde{\mathbf W}_{_{R+1}}^{(R)}\widetilde{\mathbf W}_{_{R+2}}^{(R)}.  
\]
 
\noindent $(c)$ If furthermore the  re-sizers are uniform, namely $q_r= \bar q_r = \frac1R$, $r=1,\ldots,R$, then the weights $\mathbf W_r^{(R)} $ are given by~\eqref{eq:quniform} and $\widetilde{\mathbf W}_{_{R+1}}^{(R)} $ and $\widetilde{\mathbf W}_{_{R+2}}^{(R)} $ by~\eqref{eq:quniformtilde} so that
\begin{equation}\label{eq: lesbiais}
 \widetilde m_f(\bar q,R)= -  \frac{4R}{R-1}c_{_{R+1}} c_{_{R+2}}  \gamzero ^{2R+1}R \,M^{-R(R-1)}\frac{1-M^{-R}}{1-M^{-1}}. 
\end{equation}

%
\end{theorem}

\subsection{Optimization procedure}

It remains to optimize the parameters to minimize the complexity of the estimator for a given prescribed {\em mean square error} ($MSE$). In view of the above Theorem~\ref{theo:CLT}, it is clear that the parameter $a$ should be settled at
$$
a=\bar a= \frac{1}{2R+1}.
$$


We start from Theorem~\ref{L2theo}$(b)$ with
$$
a =\bar a= \frac{1}{2R+1}\quad\mbox{ and }\quad q_r =\bar q_r:= \frac 1R, \;r=1,\ldots,R.
$$
Then the weights $\mathbf{W}_{r}$, $r=1\ldots,R$   and $\widetilde {\bf W}_{_{R+1}}$ are given by~\eqref{eq:quniform} and~\eqref{eq:quniformtilde} (those coming out in standard multilevel Monte Carlo $e.g.$ in the case of the approximation of a diffusion by its Euler scheme).

We denote by $\varpi = (R, \gamzero ,n,M)\!\in   \Pi= \EN^*\times (0,+\infty)\times \EN^*\times\EN^*$ the remaining set of free simulation parameters that we wish to optimize.
With this specification for  $a$ and the allocation vector $\bar q$, the $MSE(\varpi)$ reads 
\begin{equation}\label{eq:MSE}
\big\|\nu_n^{R,\mathbf{W}}-\nu(f)\big\|_2^2{=} \frac{1}{n^{\frac{2R}{2R+1}}} \Big(\sigma^2_f\big(\bar a,\bar q,R)+m^2_f(\bar a,\bar q,R)\Big)+\frac 1n \Big( \widetilde \sigma^2_f(\bar a,\bar q,R) +\widetilde m_f(q,R)+o( 1)\Big) \;
\end{equation}
as $n$ goes to $\infty$ where, owing to~\eqref{eq: biaisresiduel},~\eqref{eq: lesbiais},~\eqref{eq:sigma2f} and ~\eqref{eq:contribsecondordreterme},  
\begin{eqnarray*}
m_f(\bar a,\bar q,R) &=& 
2\gamzero ^R   (-1)^{R-1}R^{\frac{R}{2R+1}}M^{-\frac{R(R-1)}{2}}c_{_{R+1}},\\
 \widetilde m_f(\bar q,R)&=& -  \frac{8R}{R-1}c_{_{R+1}} c_{_{R+2}}  \gamzero ^{2R+1}R \,M^{-R(R-1)}\frac{1-M^R}{1-M^{-1}},
\\
\sigma^2_f(\bar a, \bar q,R)&=& 
=\frac{2R}{2R+1}R^{\frac{2R}{2R+1}}\sigma^2_1(f)\gamzero ^{-1},\\
\widetilde \sigma^2_f(\bar a, \bar q,R)&=& 
=R\left[\sigma_{2,1}(f)^2 + \left(1-\frac{1}{M}\right)\Psi(R,M)\sigma^2_{2,2}(f)\right].
\end{eqnarray*}
 

\noindent On the other hand, the complexity $K(\varpi,n,M)$ of the multilevel Langevin estimator  devised in~\eqref{defML2Rgodic} reads
\begin{eqnarray*}
K(\varpi,n,M) &=& n\big(q_1+(M+1)(q_2+\cdots+q_{_R})\big)\kappa_0 \\
&=& n\big(1+M(1-q_1)\big)\kappa_0 =  n\Big(1+M\Big(1-\frac{ 1}{R}\Big)\Big)\kappa_0
\end{eqnarray*}
where $\kappa_0$ denotes the unitary computational cost of one iteration of an Euler scheme.

\medskip To calibrate the above parameter $\varpi$, we want to minimize the complexity subject to a prescribed $RMSE$ $\varepsilon>0$, that is  solving the  constrained optimization problem:
\[
 \inf_{MSE(\varpi)\le \varepsilon^2} K(\varpi).
\]
To state the main result of this section, whose proof is postponed to Section~\ref{sec:proof3},   we need to  introduce  a function related to the weights $\mathbf{W}^{(R)}_r$ and on the depth of the simulation. 
We know from Lemma~\ref{lem:WWtilde}$(c)$  that $\sup_{1\le r\le R, R\ge 2}|\mathbf{W}^{(R)}_r|<+\infty$. Consequently, $M$ being fixed,  $\Psi(R,M)= O(R)$ as $R\to+\infty$ (where $\Psi$ is defined by \eqref{eq:PsiRM}). This  leads us to define 
\begin{equation}\label{eq:BoldPsi}
\mathbf{\Psi}(M)= \sup_{R\ge 1} \frac{\Psi(R,M)}{R}. 
\end{equation}
We refer to Table~\ref{valeursdePsi} for some numerical values of  $\Psi$ and ${\bf \Psi}$.  
 \begin{theorem} \label{thm:optimiz} Under the assumptions of Theorem~\ref{L2theo} and if, furthermore,
 $\displaystyle{\lim _{R\to +\infty}\frac{1}{R}\Big|\frac{c_{_{R+1}}}{c_{_{R}}}\Big|=0}$ and 
  $|c_{_R}|^{\frac 1R}\to \widetilde c \!\in (0, +\infty)$,
then
\[
 \inf_{MSE(\varpi)\le \varepsilon^2, \varpi \in \Pi} \hskip -0,5cm K(\varpi)  \precsim K(f,M) . \varepsilon^{-2}  \Big(\log \Big(\frac{1}{\varepsilon}\Big) \Big)\quad\mbox{as}\quad \varepsilon \to 0,
\]
where
\begin{equation}\label{eq:complexite}
K(f,M)=   \frac{2\kappa_0 (M+1)}{\log M}\left( \frac{  (M-1)\mathbf{\Psi}(M)}{\widetilde c\, \theta_1(f)}+1\right)\widetilde c\, \sigma_1^2(f)
\end{equation}
with
 $\theta_1(f)=\frac{\sigma_1^2(f)}{\sigma_{2,2}^2(f)}$.

\noindent $(b)$ The above bound can be achieved by the (sub-)optimal $\varpi^*$ given by $q^*  =  \frac{\mbox{\bf  1}}{R} $, $R^*=R(\varepsilon,M)=\lceil x(\varepsilon,M)\rceil$ where  $x(\varepsilon,M)$ is the unique solution to the equation  $\frac{\log(M)}{2} x(x-1)+x\log x +\log({\varepsilon})=0$ and 

\[
\gamma^*(\varepsilon, M)  =  \Big(\frac{2R}{2R+1}\Big)^{\frac{1}{2R+1}}(8R)^{-\frac{1}{2R+1}}|c_{_{R+1}}|^{-\frac{2}{2R+1}}\sigma_1^2(f)^{\frac{1}{2R+1}}M^{\frac{R(R-1)}{2R+1}}.
\]
Furthermore, 
as $\varepsilon\rightarrow0$,
\[
x(\varepsilon,M)= \sqrt{\frac{2\log \big(\frac{1}{\varepsilon}\big)}{\log M}}
 -\frac{\log_{(2)}\! \big(\frac{1}{\varepsilon}\big)}{2\log M} +\frac 12 + \frac{\log(\log M)-\log2}{2\log M}+O\left( \frac{ \log_{(2)}\!\big( 1/\varepsilon\big)   }{\sqrt{\log\big( 1/\varepsilon\big)}}\right) \quad  \mbox{ as } \varepsilon \to 0
\]
and 
the (minimal) number of iterations $n(\varepsilon,M)$ necessary to attain an MSE lower than $\varepsilon^2$ satisfies
\begin{equation}\label{eq:sizesimu}
n(\varepsilon,M)\precsim \frac{2}{\log M}\left( \frac{(M-1)\mathbf{\Psi}(M)}{\widetilde c\, \theta_1(f)}+1\right){\sigma_1^2(f)} \varepsilon^{-2}\log \Big(\frac{1}{\varepsilon}\Big) \quad  \mbox{ as } \varepsilon \to 0.
\end{equation}

 \end{theorem}
 
 \smallskip
  \begin{Remarque} Though difficult to check in practice, note that the assumptions on the sequence $(c_r)_{r\ge 1}$ are satisfied as soon as 
  \[
  \lim_{R\to +\infty}\Big|\frac{c_{_{R+1}}}{c_{_{R}}}\Big|=\widetilde c\!\in (0, +\infty).
  \]
\end{Remarque}

 \smallskip\begin{Remarque} \label{valuesofRepsilon}Note that the choice of $R(\varepsilon,M)$ does not depend on the parameters. In Table~\ref{valeursdexepsilon}, we give the values of $x(\varepsilon,M)$ for several choices of $M$ and $\varepsilon$. As expected, one can check that $R(\varepsilon,M)$ increases very slowly when $\varepsilon$ decreases. 
\begin{table}[htbp]
\begin{center}
\begin{tabular}{|c||c|c|c|c|}
\hline
&$\varepsilon=10^{-1}$ & $\varepsilon=10^{-2}$ &$\varepsilon=10^{-3}$&$\varepsilon=10^{-4}$\\
\hline
\hline
$M=2$&2.08    &2.79 &   3.38 &   3.89\\
\hline
$M=3$&1.94 &    2.56&    3.06 &   3.50\\
\hline
$M=4$ &   1.87 &   2.44&    2.90&    3.30\\
\hline
\end{tabular}
\end{center}
\caption{\label{valeursdexepsilon} Values of $x(\varepsilon,M)$}
\end{table}


\end{Remarque}
\begin{Remarque} A remarkable  point to be noted is that we retrieve the same asymptotic rate as that obtained with the original ML2R Monte Carlo simulation at finite horizon, that is for the computation of expectations $\ES\,f(X_T)$ where $X=(X_t)_{t\in [0,T]}$ is a standard diffusion discretized by its Euler scheme.\end{Remarque}
 
 Practical aspects are investigated in the practitioners' corner (see Section~\ref{sec:practicorner}) especially how to calibrate the parameters which are involved in the definition of $\varpi^*$.

\section{Expansion of the error}\label{sec:proof1}
For the sake of simplicity, the proofs are detailed in dimension $1$. In the following subsections, we begin by decomposing  the quantity $\nu_n^{\gamma,\eta}(f)-\nu(f)$  for a given smooth  coboundary function $f$  ($i.e.$ such that the  Poisson equation $f-\nu(f)=-{\cal L}g $ has a smooth enough solution) and for a general weight sequence $(\eta_n)$. Then, in the next subsections, we successively propose some expansions of the error, $\nu_n^\gamma(f)-\nu(f)$   for the original  sequence $(\nu_n^\gamma(f))_{n\ge 1}$ (implemented on the coarse level) and for the sequences of correcting empirical measures
$(\mu_n^{(r,M)}(f))$ for $r=2,\ldots,R$ defined in~\eqref{eq:defcorterm} and corresponding to the successive refined levels of our estimator.

{Note that by expansion, we mean an expansion of the bias of our estimators (level by level then globally) until we reach an order at which we reach a martingale term involved in the weak rate of convergence.}

\subsection{Higher order expansion of  $\nu_n^\gamma(f)-\nu(f)$ (coarse level)}
\bigskip

For every integer $n\ge 1$, for every sequence $(v_n)_{n\ge1}$, we set $\Delta v_n=v_n-v_{n-1}$. We will also use the following notations:
$$
\mbox{ $U_n=\gamma_n^{-\frac{1}{2}}(W_{\Gamma_n}-\Gamma_{n-1})\stackrel{d}{=} {\cal N}\big(0;I_q\big)$ and  $\rho_{m}=\ES[U_1^m]$, $m\!\in \EN$.}
$$
\begin{lemme}\label{lemme:decomp1} Let $\EMM\in\EN$. Assume that  $f-\nu(f)=-{\cal L}g$ where $g$ is a ${\cal C}^{2\EMM+3}$-function. Then, for every integer $n\ge1$,
\begin{equation}\label{eq:Deltag}
\Delta g(\bar{X}_n)=-\gamma_n(f(\bar{X}_{n-1})-\nu(f))+ \left[\sum_{\ell=2}^{\EMM+1} \gamma_n^\ell \varphi_{\ell}(f)(\bar{X}_{n-1})\right]+\sum_{i=1}^3 \Delta M_n^{(i,g)}+\Delta R_{n,L}^{(1,g)}+\Delta R_{n,L}^{(2,g)}+\Delta R_{n,L}^{(3,g)}
\end{equation}
where
\begin{align*}
&\varphi_{\ell}(f)(x)=\sum_{(m_1,m_2),m_1+\frac{m_2}{2}=\ell}g^{(m_1+m_2)}(x) 
\frac{\rho_{m_2}}{ m_1!\,m_2!} b^{m_1}(x)\sigma^{m_2}(x)\\
&\Delta M_n^{(1,g)}=\sqrt{\gamma_n}(g'\sigma)(\bar{X}_{n-1})U_{n},\quad \Delta M_n^{(2,g)}=\frac{1}{2}\gamma_ng''(\bar{X}_{n-1})\sigma^{2}(\bar{X}_{n-1})\left[U_{n}^2-1\right],\\
&\Delta M_n^{(3,g)}=\gamma_n^{\frac{3}{2}}\left(\frac 12 g''(\bar{X}_{n-1})b(\bar{X}_{n-1})\sigma(\bar{X}_{n-1}) U_n+\frac 16 g^{(3)}(\bar{X}_{n-1})\sigma^{3}(\bar{X}_{n-1})U_{n}^3\right),\\
& \Delta R_{n,\EMM}^{(1,g)}=\sum_{\ell=2}^{2\EMM+1}\gamma_n^{\ell+\frac{1}{2}}\sum_{(m_1,m_2),m_1+\frac{m_2}{2}=\ell+\frac{1}{2}}g^{(m_1+m_2)}(\bar{X}_{n-1}) \frac{1}{ m_1!\,m_2!} b^{m_1}(\bar{X}_{n-1})\sigma^{m_2}(\bar{X}_{n-1})U_{n}^{m_2}\\
&\qquad\qquad+\sum_{\ell=2}^{2\EMM+1}\gamma_n^{\ell}\sum_{(m_1,m_2),m_1+\frac{m_2}{2}=\ell}g^{(m_1+m_2)}(\bar{X}_{n-1}) \frac{1}{ m_1!\,m_2!} b^{m_1}(\bar{X}_{n-1})\sigma^{m_2}(\bar{X}_{n-1})[U_{n}^{m_2}-\rho_{m_2}],\\
&\Delta R_{n,\EMM}^{(2,g)}=\sum_{\ell=\EMM+2}^{2\EMM+2}\gamma_n^{\ell}\sum_{(m_1,m_2),m_1+\frac{m_2}{2}=\ell}g^{(m_1+m_2)}(\bar{X}_{n-1})\frac{\rho_{m_2}}{ m_1!\,m_2!}b^{m_1}(\bar{X}_{n-1})\sigma^{m_2}(\bar{X}_{n-1})\\
&\Delta R_{n,L}^{(3,g)}=g^{(2\EMM+3)}(\xi_n)(\gamma_n b(\bar{X}_{n-1})+\sqrt{\gamma_n}\sigma(\bar{X}_{n-1})U_n)^{2\EMM+3},\; \xi_n\in[\bar{X}_{n-1},\bar{X}_n].
\end{align*}
As a consequence,
\begin{equation}\label{eq:decomp2}
\begin{split}
\nu_n^{\eta,\gamma}(f)-\nu(f)=& -\frac{1}{H_n}\sum_{k=1}^n \frac{\eta_k}{\gamma_k}\Delta g(\bar{X}_k)+\sum_{\ell=2}^{\EMM+1}\frac{\sum_{k=1}^n \eta_k\gamma_k^{\ell-1}}{H_n}\nu_n^{\eta\gamma^{\ell-1},\gamma}(\varphi_{\ell}(f))\\
&+
\frac{1}{H_n}\sum_{k=1}^n\frac{\eta_k}{\gamma_k} \left(\sum_{i=1}^3\Delta M_k^{(i,g)}+\Delta R_{k,L}^{(1,g)}+\Delta R_{k,L}^{(2,g)}+\Delta R_{k,L}^{(3,g)}\right).
\end{split}
\end{equation}
\end{lemme}
\begin{proof} By the Taylor formula with order $2L+2$, we have for every $x$ and $y$ in $\ER^d$,
$$g(x+y)-g(x)=\sum_{\ell=1}^{2L+2}\frac{1}{k!} g^{(k)}(x) y^{k}+ g^{(2L+3)}(\xi) y^{2L+3}$$
where $\xi\in[x,x+y]$. Then, if  $y=\gamma b(x)+\sqrt{\gamma}\sigma(x) u$ with $u\in\ER^d$,
$$\frac{1}{k!}y^k= \sum_{m_1+m_2=k} \frac{1}{m_1!m_2!} \gamma^{m_1+\frac{m_2}{2}}b^{m_1}(x)\sigma^{m_2}(x)u^{m_2}.$$
The decomposition of $\Delta g(x)$ easily follows  by separating  odd and even $m_2$  and by remarking that 
\begin{align*}
g'(x)y+\frac{1}{2} g''(x)y^2=-\gamma {\cal L} g(x)+\sqrt{\gamma}\sigma(x)u+\frac{1}{2}\gamma \sigma^2(x)(u^2-1)+\frac{1}{2}g''(x)\left(\gamma^2 b^2(x)+2\gamma^{\frac 32}\sigma(x)u\right).
\end{align*} 
Since 
$$\nu_n^{\eta,\gamma}(f)-\nu(f)=\frac{1}{H_n}\sum_{k=1}^n\frac{\eta_k}{\gamma_k}\left(\gamma_k(f(\bar{X}_{k-1})-\nu(f)\right),$$
the second part of the lemma is a direct consequence.
\end{proof}

For notational convenience, we will denote by  ${\cal Q}f$ in what follows  the  solution of the Poisson equation $f-\nu(f)= {-}{\cal L}\big({\cal Q}f\big)$  satisfying $\nu({\cal Q}f)=0$. {(Under Assumption $\mathbf{(P)}$,  ${\cal Q}f$ is well-defined)}.
 
\begin{definition} $(a)$ Under Assumption $\mathbf{(P)}$, one may define a mapping $\varphi^{[1]}_{\ell}(.)$  from ${\cal C}^\infty(\ER,\ER)$ into itself defined for every $ f\in{\cal C^\infty(\ER,\ER)}$ by
\begin{equation}\label{eq:LesVarphi}
 \varphi^{[1]}_{\ell}(f)(.)=\sum_{(m_1,m_2),m_1+\frac{m_2}{2}=\ell}\frac{\rho_{m_2} }{ m_1!\,m_2!}b^{m_1}(.)\sigma^{m_2}(.)({\cal Q}f)^{(m_1+m_2)}(.)
\end{equation}
where 
$h^{(k)}$ denotes the $k^{th}$ derivative of a function $h$. Then,  for every $\ell\!\in\EN$, one sets $ \varphi_{\ell}^{[m]}=  \varphi_{\ell}^{[m-1]}\circ  \varphi_{\ell}^{[1]} $.  To alleviate notations, we will often write  $\varphi_m(f)$ instead of $ \varphi_m^{[1]}(f)$ in what follows.

\smallskip 
\noindent $(b)$ Still under Assumption $\mathbf{(P)}$, we  define the mappings $\Psi_{\ell}$, $\ell\!\in \EN^*$, 
\begin{equation}\label{eq:Phiell}
\Psi_\ell
=\sum_{k=1}^{\ell-1}\sum_{\underset{m_1+\ldots+m_k=\ell+k-1}{(m_1,\ldots,m_k)\in  \llbracket 2,\ell\rrbrack^k,}}\varphi_{m_1}\circ\ldots \circ \varphi_{m_k}.
\end{equation}
\end{definition}  
For example,  note that 
$$
\Psi_2=\varphi_2,\quad\Psi_3=\varphi_3+\varphi_2^{[2]}\quad\textnormal{and}\quad \Psi_4=\varphi_4+\varphi_3\circ\varphi_2+\varphi_2\circ\varphi_3+\varphi_2^{[3]}.
$$
We have the following expansions of the error, depending on the averaging properties of the step sequence $\gamma$. 
%

\begin{prop}[Bias error expansion for the coarse level] \label{prop:Multistep} Assume $\mathbf{(S)}$, $\mathbf{(P)}$ (and uniqueness of the invariant distribution $\nu$). Let $R\!\in \EN$, $R\ge 2$ and let $f\!\in{\cal C}^\infty(\ER,\ER)$ with polynomial growth
and $g={\cal Q} f$.

\smallskip
\noindent $(a)$ If $(\gamma_n^{\ell}, \gamma_n)_{n\ge 1}$ is averaging for every $\ell\!\in \{1,\ldots,\ERR\}$, 
$$ 
\nu_n^{\gamma}(\omega,f)-\nu(f)-\sum_{\ell=2}^{\ERR}\frac{\Gamma_n^{(\ell)} }{\Gamma_n}\nu\big(\Psi_{\ell}(f)\big)=\frac{M_n^{(1,g)}}{\Gamma_n}+o_{L^2}\left(\frac{\sqrt{\Gamma_n}\vee\Gamma_n^{(\ERR)}}{\Gamma_n}\right).
$$
\noindent $(b)$ If, furthermore, the pair $(\gamma_n^{R+1}, \gamma_n)_{n\ge 1}$ is averaging, 
$$ \nu_n^{\gamma}(\omega,f)-\nu(f)-\sum_{\ell=2}^{\ERR}\frac{\Gamma_n^{(\ell)} }{\Gamma_n}\nu(\Psi_{\ell}(f))=\frac{M_n^{(1,g)}}{\Gamma_n}+\frac{\Gamma_n^{(\ERR+1)} }{\Gamma_n}\nu\big(\Psi_{\ERR+1}(f)\big)+o_{L^2}\left(\frac{\sqrt{\Gamma_n}\vee\Gamma_n^{(\ERR+1)}}{\Gamma_n}\right).$$
(c) The following sharper expansion also holds when $(\gamma_n^{R+2}, \gamma_n)_{n\ge 1}$ is averaging 
\begin{align*}
\nu_n^{\gamma}(\omega,f)-\nu(f)-\sum_{\ell=2}^{\ERR}\frac{\Gamma_n^{(\ell)} }{\Gamma_n}&\nu(\Psi_{\ell}(f))=\frac{M_n^{(1,g)}+N_n}{\Gamma_n}\\
&+\frac{\Gamma_n^{(\ERR+1)} }{\Gamma_n}\nu\big(\Psi_{\ERR+1}(f)\big)+
\frac{\Gamma_n^{(\ERR+2)} }{\Gamma_n}\nu\big(\Psi_{\ERR+2}(f)\big)+o_{L^2}\left(\frac{\sqrt{\Gamma_n^{(3)}}\vee\Gamma_n^{(\ERR+2)}}{\Gamma_n}\right),
\end{align*}
where $N_0=0$ and 
$$\Delta N_n=\Delta M_k^{2,g}+\Delta M_k^{3,g}+\gamma_k^{\frac{3}{2}}(\sigma g_2')(\bar{X}_{k-1})U_k,$$
with $g_2={\cal Q}(\varphi_2(f))$, $i.e.$ the solution to $\varphi_2(f)-\nu(\varphi_2(f))=-{\cal L} g_2$.
\end{prop}
\begin{Remarque} The  first expansion is adapted to the proof of Theorem~\ref{theo:CLT}$(a)$, the second one to  Theorem~\ref{theo:CLT}$(b)$ and $(c)$  and Theorem~\ref{L2theo}$(a)$. Statement $(c)$ is written in view of Theorem~\ref{L2theo}$(b)$ where  one needs to handle the second order term of the asymptotic expansion of the $MSE$. Note that the bias term of order $R+2$ in $(c)$ will contribute to $\widetilde m_f(\bar q,R)$ in Theorem~\ref{L2theo}$(b)$. At this stage, it can be justified by the following remark:  when $a=1/(2R+1)$,
$$\frac{\Gamma_n^{(\ERR+1)} }{\Gamma_n}\frac{\Gamma_n^{(\ERR+2)} }{\Gamma_n}\overset{n\rightarrow+\infty}{\sim} \left(\frac{2R}{2R+1}\right)^2 \frac{1}{n}.$$
As concerns the contribution of the martingale correction $\Delta N_n$, we refer to   Proposition~\ref{prop:contribmartnun} for details. Finally, remark that all the negligible terms are given with the $L^2$-norm. For Theorem~\ref{theo:CLT},  ``$o_{\PE}$'' is enough.
\end{Remarque}

%
%
%

\begin{proof}  $(a)$ and $(b)$: Let $R\ge 2$ be an integer. Let us consider the decomposition given by~\eqref{eq:Deltag} in Lemma~\ref{lemme:decomp1}. When $(\gamma_n)_{n\ge 1}=\eta=(\gamma_n)_{n\ge 1}$, $L=R$ and $g={\cal Q}f$, we get  
\begin{align}
\nu_n^\gamma(f)-\nu(f)-\sum_{\ell=1}^{\ERR}\frac{\Gamma_n^{(\ell)} }{\Gamma_n}\nu(\varphi_{\ell}(f))=& \frac{g(\bar{X}_0)-g(\bar{X}_n)}{\Gamma_n} +\sum_{\ell=2}^{\ERR}\frac{\Gamma_n^{(\ell)}}{\Gamma_n}\left(\nu_n^{\gamma^{\ell}, \gamma}(\varphi_{\ell}(f))-\nu(\varphi_{\ell}(f))\right)\label{eq:54}\\
&+\frac{\Gamma_n^{(\ERR+1)}}{\Gamma_n}\nu_n^{ \gamma^{\ERR+1},\gamma}(\varphi_{\ERR+1}(f))+\frac{M_n^{1,g}}{\Gamma_n}\nonumber\\
&+\frac{1}{\Gamma_n}\sum_{k=1}^n \left(\sum_{i=2}^3\Delta M_k^{(i,g)}+\Delta R_{k,R}^{(1,g)}+\Delta R_{k,R}^{(2,g)}+\Delta R_{k,R}^{(3,g)}\right).\nonumber
\end{align}
By Lemma~\ref{lemme:gesttermesreste}$(i)$ applied with $(\eta_n)=(\gamma_n)$,
$$
\Big\|\frac{g(\bar{X}_0)-g(\bar{X}_n)}{\Gamma_n}\Big\|_2\le \frac{C}{\Gamma_n}$$
As well, by Lemma~\ref{lemme:gesttermesreste}$(ii)$applied for different choices of $(\theta_n)$, $h$ and $(Z_n)_{n\ge1}$, we have 
$$\Big\|\frac{1}{\Gamma_n}\sum_{k=1}^n \left(\Delta M_k^{(2,g)}+\Delta M_k^{(3,g)}+\Delta R_{k,R}^{(1,g)}\right)\Big\|_2\le C\frac{\sqrt{\Gamma_n^{(2)}}}{\Gamma_n}.$$
Finally, Lemma~\ref{lemme:gesttermesreste}$(iii)$ and $(iv)$ are  adapted to manage $\Delta R_{k,R}^{(2,g)}$ and $\Delta R_{k,R}^{(3,g)}$ respectively. This yields 
$$\Big\|\frac{1}{\Gamma_n}\sum_{k=1}^n \left(\Delta R_{k,R}^{(2,g)}+\Delta R_{k,R}^{(3,g)}\right)\Big\|_2\le C \left(\frac{\Gamma_n^{(R+2)}}{\Gamma_n}+ \frac{\Gamma_n^{(R+\frac{3}{2})}}{\Gamma_n}\right)\le C\frac{\Gamma_n^{(R+\frac{3}{2})}}{\Gamma_n}.$$
The above terms are thus negligible in expansions $(a)$ and $(b)$. As concerns $\nu_n^{ \gamma^{\ERR+1},\gamma}(\varphi_{\ERR+1}(f))$, one can  deduce from the polynomial growth of $\varphi_{\ERR+1}(f)$ and from~\eqref{controleLP22} that there exists $C>0$ such that
$$\forall n\ge1,\quad \Big\|\nu_n^{ \gamma^{\ERR+1},\gamma}(\varphi_{\ERR+1}(f)) \Big\|_2\le C.$$
This means that this term is negligible in the expansion $(a)$. In $(b)$, $(\gamma^{\ERR+1}_n,\gamma_n)$ is averaging so that
by Proposition~\ref{prop:convergence},  
$$\nu_n^{ \gamma^{\ERR+1},\gamma}(\varphi_{\ERR+1}(f))\xrightarrow{n\rightarrow+\infty}\nu(\varphi_{\ERR+1}(f))\quad a.s.$$
But using again~\eqref{controleLP22}, one checks that there is a $\delta>0$ such that   $(\|\nu_n^{ \gamma^{\ERR+1},\gamma}(\varphi_{\ERR+1}(f))\|_{2+\delta})_n$ is a bounded sequence. Thus,
an uniform integrability argument yields that
$$\nu_n^{ \gamma^{\ERR+1},\gamma}(\varphi_{\ERR+1}(f))\xrightarrow{n\rightarrow+\infty}\nu(\varphi_{\ERR+1}(f))\quad \textnormal{in $L^2$}.$$
But for any $\ell$, $\varphi_{\ell}$ is the component corresponding to $k=1$ in the definition~\eqref{eq:Phiell} of $\Psi_\ell$.   In $(b)$, $\nu(\varphi_{\ERR+1}(f))$ will thus contribute to $\nu(\Psi_{R+1})$. As well,  the terms $\nu(\varphi_\ell)$, $\ell=2,\ldots,R$ exhibited in this first expansion   will certainly contribute to $\nu(\Psi_\ell)$, $\ell=2,\ldots,R$. \smallskip

\noindent Now, we focus on the second bias term of the right-hand side  of~\eqref{eq:54}. More precisely, for each $\ell\in\{2,\ldots,\ERR\}$, we have  to repeat the previous procedure: we apply the expansion~\eqref{eq:Deltag} of Lemma~\ref{lemme:decomp1}
with $\eta=(\gamma_n^\ell)_{n\ge1}$, $L=R-\ell+1$, $f_\ell=\varphi_\ell$ and $g_\ell={\cal Q}\varphi_\ell$ (defined above). After several transformations, this yields

\begin{align}
\frac{\Gamma_n^{(\ell)}}{\Gamma_n}&\left(\nu_n^{\gamma^\ell,\gamma}(f_\ell)-\nu(f_\ell)\right)-\sum_{m=2}^{R-\ell+1}\frac{\Gamma_n^{(\ell+m-1)}}{\Gamma_n}\nu(\varphi^{[1]}_m\circ\varphi^{[1]}_{\ell}(f))= -\frac{1}{\Gamma_n}\sum_{k=1}^n \gamma_k^{\ell-1}\Delta  {\cal Q}\varphi_\ell(\bar{X}_k)\nonumber\\
&+\sum_{m=2}^{R-\ell+1}\frac{\Gamma_n^{(\ell+m-1)}}{\Gamma_n}\left(\nu_n^{\gamma^{\ell+m-1}, \gamma}-\nu\right)(\varphi_m\circ\varphi_{\ell}(f))\label{diuouis}\\
&+\frac{\Gamma_n^{(R+1)}}{\Gamma_n}\nu_n^{\gamma^{R+1}, \gamma}(\varphi_{R-\ell+2}\circ\varphi_{\ell}(f))\nonumber\\
&+\frac{1}{\Gamma_n}\sum_{k=1}^n\gamma_k^{\ell-1} \left(\sum_{i=1}^3\Delta M_k^{(i,g_\ell)}+\Delta R_{k,R-\ell+1}^{(i,g_\ell)}\right).\nonumber
\end{align}
Applying again Lemma~\ref{lemme:gesttermesreste} allows us to control the $L^2$-norm of the negligible terms:
$$\Big\|\frac{1}{\Gamma_n}\sum_{k=1}^n \gamma_k^{\ell-1}\Delta  {\cal Q}\varphi_\ell(\bar{X}_k)\Big\|_2\le \frac{C \gamma_1^{\ell-1}}{\Gamma_n}$$
and 
$$\Big\|\frac{1}{\Gamma_n}\sum_{k=1}^n\gamma_k^{\ell-1} \left(\sum_{i=1}^3\Delta M_k^{(i,g_\ell)}+\Delta R_{k,R-\ell+1}^{(i,g_\ell)}\right)\Big\|_2 \le C\frac{\sqrt{\Gamma_n^{(2\ell-1)}}\vee\Gamma_n^{(R+\frac{3}{2}+\ell-1)}}{\Gamma_n}.$$
Again, the penultimate term of the previous decomposition is negligible for expansion $(a)$ and satisfies the following convergence property when $(\gamma^{R+1},\gamma)$ is averaging:
$$
\frac{\Gamma_n}{\Gamma_n^{(R+1)}}\left(\frac{\Gamma_n^{(R+1)}}{\Gamma_n}\nu_n^{\gamma, \gamma^{L+1}}\big(\varphi^{[1]}_{R-\ell+2}\circ\varphi^{[1]}_{\ell}(f)\big)\right)
\xrn{n\nrn}\nu\big(\varphi^{[1]}_{L-\ell+2}\circ\varphi^{[1]}_{\ell}(f)\big)\quad \textnormal{$a.s.$ and in $L^2$.}
$$
This brings a second ``contribution'' to $\nu(\Psi_{R+1})$. 


\noindent Finally, it remains to consider for every $\ell\in\{2,\ldots,R\}$ each term  of~\eqref{diuouis}. Setting $\ell=m_1$, $m=m_2$,   the sequel of the proof  consists in repeating the procedure until $k:=\inf\{i\,:\, m_1+\ldots m_i=R+i\}$. The result follows.

\smallskip
\noindent $(c)$ The proof  is based on the same principle but is slightly more involved since we aim at keeping all the terms which are going to play a role in the second order expansion of Theorem~\ref{L2theo}$(b)$. This implies to start the previous proof with $L=R+1$ (and in the second step with $L=R-\ell+2$). Furthermore, the main other difference comes from the martingale component. As a complement of $M_n^{(1,g)}$, one also keeps whole the martingale terms whose $L^2$-norm is not negligible with respect to $\sqrt{\frac{\Gamma_n^{(3)}}{\Gamma_n}}$. In short, this corresponds to the martingale increments with a factor $\gamma_k$ or $\gamma_k^{\frac 32}$. This yields the two martingale increments $\Delta M_k^{(2,g)}$ and $\Delta M_k^{(3,g)}$ of the first expansion but also the dominating martingale increment of the second expansion above : $\gamma_k\Delta M_k^{(1,g_\ell)}$. The result follows.
%
\end{proof}
\begin{lemme}\label{lemme:gesttermesreste} Assume $\mathbf{(S)}$.  Let $h$ be a smooth function with polynomial growth.   We know from Proposition~\ref{prop:convergence} that, for every $p\!\in (0,+\infty)$,
\begin{equation}\label{controleLP22} 
C_{h,p}= \sup_{n\ge1}\|h(X_n)\|_p<+\infty.
\end{equation}
Then, 

\smallskip
\noindent (i) If $(\eta_n/\gamma_n)_{n\ge1}$ is a non-increasing sequence of real numbers, 
$$ 
\Big\|\sum_{k=1}^n \frac{\eta_k}{\gamma_k}\Delta h(\bar{X}_k)\Big\|_2\le C_{h,2}\frac{\eta_1}{\gamma_1}.
$$
(ii) If $(Z_k)_{k\ge1}$ is a sequence of $i.i.d$ centered random variables with finite variance, then for any deterministic sequence $(\theta_k)_{k\ge0}$,
$$\Big\| \sum_{k=1}^n \theta_k h(\bx_{k-1})Z_k\Big\|_2\le C_{h,2}\|Z_1\|_2 \sqrt{\sum_{k=1}^n \theta_k^{2}}.$$
(iii) For any sequence $(\theta_k)_{k\ge1}$ of real numbers,
$$
\Big\| \sum_{k=1}^n\theta_k h(\bx_{k-1})\Big\|_2\le C_{h,2} \sum_{k=1}^n|\theta_k|
$$
(iv) For any sequence $(\theta_k)_{k\ge1}$ of real numbers and any $r>0$, there exists a real constant $C= C_{r,b,\sigma, h,\gamma}$  such that 
$$
\Big\| \sum_{k=1}^n\theta_k \sup_{u\in[0,1]} |h(\bx_{k-1}+u\Delta \bx_{k})||\Delta \bx_k|^r\Big\|_2\le C\sum_{k=1}^n |\theta_k|\gamma_k^{\frac r2}.
$$
\end{lemme}
\begin{proof} Using that $(\eta_n/\gamma_n)_{n\ge1}$ is a non-increasing sequence, we have
$$
\left|\sum_{k=1}^n \frac{\eta_k}{\gamma_k}\Delta h(\bar{X}_k)\right|=\frac{\eta_1}{\gamma_1} |h(\bar{X}_0)|+\sum_{k=1}^{n-1}\left(\frac{\eta_{k}}{\gamma_{k}}-\frac{\eta_{k+1}}{\gamma_{k+1}}\right) |h(\bx_{k})|+\frac{\eta_n}{\gamma_n} |h(\bar{X}_n)|
$$
so that
\[
\left\|\sum_{k=1}^n \frac{\eta_k}{\gamma_k}\Delta h(\bar{X}_k)\right\|_2\le C_{h,2}\left(\frac{\eta_1}{\gamma_1}   +\sum_{k=1}^{n-1}\left(\frac{\eta_{k}}{\gamma_{k}}-\frac{\eta_{k+1}}{\gamma_{k+1}}\right) +\frac{\eta_n}{\gamma_n} \right)= C_{h,2}\frac{\eta_1}{\gamma_1}.
\]
This concludes the proof of $(i)$. Items $(ii)$ and $(iii)$ are straightforward consequences of the fact that $\sup_{n\ge1}\ES[|h(X_n)|^2]<+\infty$. For $(iv)$, the polynomial growth of $h$ implies that there exists $p>0$ and a constant $C>0$ such that for any $x,y\in\ER^d$ ,
$$
\sup_{u\in [0,1]}|h(x+uy)|\le C(1+|x|^p+|y|^p).$$
Using that $b$ and $\sigma$ are sub-linear functions and  Minkowski's Inequality
$$	
\Big\|\sup_{u\in[0,1]} |h(\bx_{k-1}+u\Delta \bx_{k})||\Delta \bx_k|^r\Big\|_2\le C \big(1+  \| |X_{k-1}|^{p}\|_4+\| |\Delta X_k|^{p}\|_4\big) \||\Delta X_{k}|^{r}\|_4\le  \tilde{C}\gamma_k^{\frac{r}{2}}
$$ 
The last statement follows using again Minkowski's Inequality.
\end{proof}
\subsection{Error expansion of  the  correcting levels}
For a given sequence $\gamma:=(\gamma_n)$, let us denote by $(\bar{X}_k)_{k\ge 0}$ and $(\bar{Y}_k)_{k\ge 0}$ the two Euler schemes of the diffusion $(X_t)_{t\ge 0}$ driven by the same Brownian motion $W$ and with the step sequences $(\gamma_n)$ and 
$(\gamma_n/M)$ respectively.  We then define a sequence of empirical measures  $(\mu_n^{M,\gamma})$ by
\begin{align*}
 \mu^{M,\gamma}_n(dx)= \frac{1}{\Gamma_n} \sum_{k=1}^n \left(\left(\sum_{m=0}^{M-1}\frac{\gamma_k}{M} \delta_{{\bar Y}_{M(k-1)+m}}\right)- \gamma_k\delta_{{\bar X}_{k-1}}\right),\; n\ge 1.
\end{align*}
By the definition~\eqref{eq:defcorterm}, one first notes that for $r=2,\ldots,R$, $\mu^{(r,M)}_n= \mu^{M,\gamma^{(r)}}_n$   built with the Euler schemes  $\bar X^{(r)}$ and $\bar Y^{(r)}$ (keep in mind that $\gamma^{(r)}_k=\frac{\gamma_k}{M^{r-2}}$). As a consequence,    expanding   $(\mu^{M,\gamma}_n(f))_{n\ge 1}$ will elucidate the behavior of the refined levels in the {\bf ML2Rgodic} procedure.

\smallskip
\noindent In the proposition below, we thus state a result similar to Proposition~\ref{prop:Multistep} but for the  sequence $(\mu_n^{M,\gamma}(f))_{n\ge1}$.
\begin{prop}[Bias error expansion for the refined levels] \label{prop:dvpcorrectiveterms} Assume $\mathbf{(S)}$, $\mathbf{(P)}$ and uniqueness of the invariant distribution $\nu$ of the diffusion is unique. Let $\ERR\!\in \EN^*$, $\ERR\ge 2$ and let $f\!\in{\cal C}^\infty(\ER,\ER)$ with polynomial growth and let $g={\cal Q}f$.   

\smallskip
\noindent $(a)$ Assume  that for every $\ell\!\in \{1,\ldots,\ERR\}$, the pair $(\gamma_n^{\ell}, \gamma_n)_{n\ge 1}$ is averaging. Then,
$$ 
\mu_n^{M,\gamma}(f)-\sum_{\ell=2}^{\ERR}(M^{1-\ell}-1)\frac{\Gamma_n^{(\ell)} }{\Gamma_n}\nu(\Psi_{\ell}(f))=-\frac{{\cal M}_n(\sigma g')}{\Gamma_n}+o_{L^2}\left(\frac{\sqrt{\Gamma_n}\vee\Gamma_n^{(\ERR)}}{\Gamma_n}\right)
$$
where for a   Borel function $\varphi:\ER^d\rightarrow\ER$
\[
{\cal M}_n(\varphi) = \sum_{k=1}^n \varphi(\bar X_{k-1}) \big(W_{\Gamma_{k}}-W_{\Gamma_{k-1}}\big)
-  \sum_{m=0}^{M-1}\varphi({\bar Y}_{M(k-1)+m})\big(W_{ \Gamma_{k-1+\frac{m+1}{M}} } -W_{\Gamma_{k-1+\frac mM}}\big).
 \]
\noindent $(b)$ If furthermore, the pair $(\gamma_n^{R+1}, \gamma_n)_{n\ge 1}$ is averaging, then the following sharper expansion also holds:
\begin{eqnarray*}
 \mu_n^{M,\gamma}(\omega,f)-\sum_{\ell=2}^{\ERR}(M^{1-\ell}-1)\frac{\Gamma_n^{(\ell)} }{\Gamma_n}\nu(\Psi_{\ell}(f))&=&-\frac{{\cal M}_n(\sigma g')}{\Gamma_n}+\left(M^{-R}-1\right)\frac{\Gamma_n^{(\ERR+1)} }{\Gamma_n}\nu\big(\Psi_{\ERR+1}(f)\big)\\
 &&+o_{L^2}\left(\frac{\sqrt{\Gamma_n^{(2)}}\vee\Gamma_n^{(\ERR+1)}}{\Gamma_n}\right).
\end{eqnarray*}
$(c)$ The following sharper expansion also holds when $(\gamma_n^{R+2}, \gamma_n)_{n\ge 1}$ is averaging :
\begin{align*}
&\mu_n^{M,\gamma}(f)-\sum_{\ell=2}^{\ERR}(M^{1-\ell}-1)\frac{\Gamma_n^{(\ell)} }{\Gamma_n}\nu\big(\Psi_{\ell}(f)\big)=-\frac{{\cal M}_n(\sigma g')+{\cal N}_n(\frac{1}{2}\sigma^2 g'')}{\Gamma_n}\\
&+\left(M^{-R}-1\right)\frac{\Gamma_n^{(\ERR+1)} }{\Gamma_n}\nu\big(\Psi_{\ERR+1}(f)\big) +\left(M^{-R-1}-1\right)\frac{\Gamma_n^{(\ERR+2)} }{\Gamma_n}\nu\big(\Psi_{\ERR+2}(f)\big)
+o_{L^2}\left(\frac{\sqrt{\Gamma_n^{(2)}}\vee\Gamma_n^{(\ERR+2)}}{\Gamma_n}\right),
\end{align*}
where, for a   Borel function $\varphi:\ER^d\rightarrow\ER$,
\[
{\cal N}_n(\varphi) = \sum_{k=1}^n \varphi(\bar X_{k-1}) \Big(\big(W_{\Gamma_{k}}-W_{\Gamma_{k-1}}\big)^2-\gamma_k\Big)
-  \sum_{m=0}^{M-1}\varphi({\bar Y}_{M(k-1)+m})\left(\big(W_{ \Gamma_{k-1+\frac{m+1}{M}} } -W_{\Gamma_{k-1+\frac mM}}\big)^2-\frac{\gamma_k}{M}\right).
 \]

\end{prop}
\begin{proof}  With the notation introduced in~\eqref{eq:gamtildebb}. Set 
$$\nu_n^{\tilde{\gamma}^{2,M}}(\bar{Y},f)=\left(\sum_{k=1}^{n}\tilde{\gamma}_k^{2,M}\right)^{-1}\sum_{k=1}^n\tilde{\gamma}_k^{2,M} \delta_{\bar{Y}_{k-1}}.$$
 One can check that for every $n\ge1$, 
$$ \mu_n^{M,\gamma}(\omega,f)=\left(\nu_{nM}^{\tilde{\gamma}^{2,M}}(\bar{Y},f)-\nu(f)\right)-\left(\nu_n^\gamma(f)-\nu(f)\right).$$
For $(a)$ and $(b)$, it remains now to apply  Proposition~\ref{prop:Multistep}$(a)$ and $(b)$ to both terms in the  right-hand side of the above equation (with step $\tilde{\gamma}^{2,M}$ for $\nu_{nM}^{\tilde{\gamma}^{2,M}}(\bar{Y},f)$). The result follows by concatenating martingale components  and by noting that for any integer $\ell\ge 2$,
$$\frac{\sum_{k=1}^{nM}(\tilde{\gamma}_k^{2,M})^\ell}{\sum_{k=1}^{nM}\tilde{\gamma}_k^{2,M}}= \frac{M^{1-\ell}\Gamma_n^{(\ell)}}{\Gamma_n}.$$
For the proof of $(c)$, the only difference with  Proposition~\ref{prop:Multistep}$(c)$ is that one only keeps the martingale increment $ M_n^{(2,g)}$ of the corrective term $N_n$. More precisely, the terms of $N_n$ appearing with a factor $\gamma_n^{\frac{3}{2}}$
are here viewed as negligible terms. Using Lemma~\ref{lemme:gesttermesreste}$(ii)$, one easily check that these martingale corrections are bounded in $L_2$ by $\sqrt{\Gamma_n^{(3)}}/{\Gamma_n}$ (which is $o(\sqrt{\Gamma_n^{(2)}}/{\Gamma_n})$).
\end{proof}
\begin{Remarque} The fact that we keep less martingale terms in Expansion $(c)$ can be understood as follows: in section~\eqref{prop:contribmartmun}, we will show that  the apparently dominating martingale component ${\cal M}_n(\sigma g')$
is in fact negligible at the first order of the expansion under confluence assumptions. This implies that the covariance terms induced by the product of this martingale and the martingale corrections appearing with a factor $\gamma_k^{\frac{3}{2}}$ in $N_n$ (see Proposition~\ref{prop:Multistep}) will be also negligible at a second order.
\end{Remarque}
\section{Rate of convergence for the  dominating martingales}\label{sec:dominatingmartingales}\label{sec:proof2}
In the continuity of Propositions~\ref{prop:Multistep} and~\ref{prop:dvpcorrectiveterms}, we now propose to elucidate the weak or $L^2$ rate of convergence of  the dominating martingales , that is the martingales coming out in the   above error expansions established in the former section. 
\subsection{The dominating martingale term involved in  $\nu_n^\gamma(f)-\nu(f)$}
We begin by stating some asymptotic results for the first and second order martingales  $(M_n^{(1,g)})_{n\ge 1}$ and $(N_n)_{n\ge 1}$ which appear in the expansions of Proposition~\ref{prop:Multistep}. The associated statements describe the asymptotic martingale contributions of the first (dominating) term of the {\bf ML2Rgodic} procedure. With the view to Theorem~\ref{theo:CLT}, the first statement concerns the convergence in distribution of the dominating martingale $(M_n^{(1,g)})_{n\ge 1}$ whereas the second and third ones  are crucial steps in the proof of Theorem~\ref{L2theo} $(a)$ and $(b)$ respectively.

\begin{prop} \label{prop:contribmartnun}Assume $\mathbf{(S)}$ and $\mathbf{(P)}$.  Let $g={\cal Q}f$.Then,

\smallskip
\noindent $(a)$ 
\[
\frac{1}{\sqrt{\Gamma_n}}  M_n^{(1,g)} \stackrel{(\R)}{\Longrightarrow}{\cal N}\Big(0;\int_{\R}(\sigma g')^2d\nu\Big).
\]

\noindent $(b)$ 
\[
\ES\left[\frac{(M_n^{(1,g)})^2}{\Gamma_n}\right]=\int_{\R}(\sigma g')^2d\nu+o(1)\quad\textnormal{as $n\rightarrow+\infty$}.
\]
\noindent $(c)$ If $(\gamma_n,\gamma_n^2)$ is averaging, 
\[
\ES\left[\frac{(M_n^{(1,g)}+ N_n)^2}{\Gamma_n}\right]=\int_{\R}(\sigma g')^2d\nu+\frac{\Gamma_n^{(2)}}{\Gamma_n}\Big(\sigma_{2,1}^2(f)+o(1)\Big)
\quad\textnormal{as $n\rightarrow+\infty$},
\]
where 
\begin{equation}\label{eq:contribsecondordreterme1}
\sigma_{2,1}^2(f)=\int_{\ER}\left[ \varphi_2((\sigma g')^2)+\frac{1}{2}(\sigma^2 g'')^2+(\sigma g')\big(g^{(3)}\sigma^3+2(\sigma g_2')\big)\right]d\nu,
\end{equation}
where $g_2={\cal Q}\varphi_2(f)$,  $i.e.$ the  solution to $\varphi_2(f)-\nu(\varphi_2(f))=-{\cal L} g_2$.
\end{prop}
\begin{Remarque} If $\gamma_n=\gamzero  n^{-\frac{1}{2R+1}}$, 
$$
\frac{1}{\Gamma_n}\overset{n\rightarrow+\infty}{\sim}\frac{2R}{(2R+1)\gamzero } n^{-\frac{2R}{2R+1}}\quad\textnormal{and}\quad \frac{\Gamma_n^{(2)}}{\Gamma_n}\overset{n\rightarrow+\infty}{\sim}\frac{2R}{(2R+1)n}.
$$
One thus retrieves the orders of the expansions established in Theorem~\ref{L2theo}.
\end{Remarque}
\begin{proof} $(a)$ Using Proposition~\ref{prop:convergence},
\begin{equation}\label{eq:pscrochet}
\frac{\langle M^{(1,g)}\rangle_n}{\Gamma_n}=\nu_n^\gamma\big((\sigma g')^2\big)\xrightarrow{n\rightarrow+\infty} \nu\big((\sigma g')^2\big)\quad a.s.
\end{equation}
Furthermore, by Cauchy-Schwarz inequality and~\eqref{controleLP22}, we have for every $\varepsilon>0$, 
$$
\sum_{k=1}^n\ES\left[ (\Delta M_k^{(1,g)})^21_{(\Delta M_k^{(1,g)})^2>\varepsilon}\right]\le \frac{1}{\varepsilon^2}\sum_{k=1}^n\ES\big[ (\Delta M_k^{(1,g)})^4\big]\le C\frac{\Gamma_n^{(2)}}{\Gamma_n^2}\xrightarrow{n\rightarrow+\infty}0.
$$
This second convergence implies that the so-called \textit{Lindeberg condition} is fulfilled. Then,  $(a)$ is a consequence of the CLT for martingale arrays (see~\cite[Corollary 3.1]{hall-heyde}).

\smallskip
\noindent $(b)$  By Jensen inequality, for a given function $f$,
$$
\ES\big[(\nu_n^\gamma(f))^2\big]\le \ES\big[\nu_n^\gamma(f^2)\big]
$$
and it follows again from Proposition~\ref{prop:convergence} and from the fact that $\sigma g'$ has (at most)  polynomial growth that 
\begin{equation}\label{eq:EIArg}
\sup_n\ES\big[(\nu_n^\gamma((\sigma g')^2))^2\big]\le \sup_n\ES\big[1+|\bar X_n|^r\big]<+\infty.
\end{equation}
owing to  $\mathbf{(S)}$ and~\eqref{controleLP22}. 
As a consequence, $\big(\nu_n^\gamma((\sigma g')^2)\big)_{n\ge 1}$ is a uniformly integrable sequence so that the convergence of  $(\nu_n^\gamma((\sigma g')^2))$ toward $\nu((\sigma g')^2)$
also holds in $L^1$. The second statement then follows from~\eqref{eq:pscrochet}.

\smallskip
\noindent $(c)$ First, using that $\ES[U_n(U_n^2-1)]=0$ and that $\ES[U_n^4]=1$, one can check that 
\begin{equation*}
\frac{1}{\Gamma_n}\ES\big[(M_n^{(1,g)}+N_n)^2\big]=\ES\big[\nu_n^{\gamma,\gamma}((\sigma g')^2)\big]+ \frac{\Gamma_n^{(2)}}{\Gamma_n}\ES\big[\nu_n^{\gamma^2,\gamma}(F)\big],
\end{equation*}
where 
$$
F(x)=\Big[\frac{1}{2}(\sigma^2 g'')^2+(\sigma g')(g^{(3)}\sigma^3+2(\sigma g_2'))\Big](x).
$$
On the one hand, since $(\gamma_n^2,\gamma_n)_{n\ge1}$ is averaging, we deduce from Proposition~\ref{prop:convergence} that 
$$\nu_n^{\gamma^2,\gamma}(F)\xrn{n\nrn}\nu(F)\quad a.s.$$
But using uniform integrability arguments similar to~\eqref{eq:EIArg}, the convergence also holds in $L^1$. 
On the other hand, let us focus on $\ES[\nu_n^{\gamma,\gamma}((\sigma g')^2)]$. We set $h=(\sigma g')^2.$ Using Proposition~\ref{prop:Multistep}$(a)$ (and the fact that $\Psi_2=\varphi_2$) with 
$R=2$, we have
$$ 
\nu_n^{\gamma}(h)-\nu(h)=\frac{M_n^{(1,{\cal Q}h)}}{\Gamma_n}+\frac{\Gamma_n^{(2)} }{\Gamma_n}\nu\big(\varphi_2(h)\big)+o_{L^2}\left(\frac{\sqrt{\Gamma_n}\vee\Gamma_n^{(2)}}{\Gamma_n}\right).
$$
By~\eqref{eq:pscrochet}, we deduce that 
$$
\ES\left[\frac{(M_n^{(1,g)}+N_n)^2}{\Gamma_n}\right]=\int_{\R}(\sigma g')^2d\nu +\frac{\Gamma_n^{(2)}}{\Gamma_n}\Big(\nu(\varphi_2(h)+F)+o(1)\Big).
$$
The last statement follows.
%
\end{proof}

\subsection{The dominating martingale in the error expansion of $(\mu_n^{M,\gamma}(f))_{n\ge 1}$}
In this section, we focus on the behavior of the martingale terms involved by the refined levels of the {\bf ML2Rgodic} procedure. Thus, this corresponds to the variance induced by this procedure. On a finite horizon, 
Euler schemes are pathwise close (in an $L^2$-sense for instance)  and this property implies one of the  important features of multilevel procedures: reducing the bias without increasing significantly  the variance. As mentioned before, on a long run scale, such a property is not true in general. More precisely, without additional assumptions, the martingale $({\cal M}_n)_{n\ge 1}$ defined in Proposition~\ref{prop:dvpcorrectiveterms} is {\em a priori} not negligible
compared to the one induced by the first term of the {\bf ML2Rgodic} procedure. However, this turns out to be true in presence of an  asymptotic confluence assumption. This is the first statement of the next proposition. In the second one, we go deeper in the 
analysis of the martingale contribution of $(\mu_n^{M,\gamma}(f))_{n\ge 1}$  under   a stronger confluence assumption. The second property will contribute only to Theorem~\ref{L2theo}$(b)$.

\begin{prop} \label{prop:contribmartmun} Assume $\mathbf{(S)}$ and $\mathbf{(P)}$. Let $h_1$ and $h_2$ be  locally Lipschitz functions with polynomial growth.

\smallskip
\noindent $(a)$ If $\mathbf{(C_w)}$ holds, then $\left(\frac{{\cal M}_n(h_1)}{\sqrt{\Gamma_n}}\right)_{n\ge 1}$ converges to $0$ in $L^2$.

\smallskip
\noindent $(b$) Assume $\mathbf{(C_s)}$ holds and that $(\gamma_n,\gamma_n^2)_n$ is averaging. Assume that $h_1$ is ${\cal C}^2$ and that $h_1$ and its derivatives have  polynomial growth. Then, the martingales 
$({\cal M}_n(h_1))$ and $({\cal N}_n(h_2))$ are orthogonal and
\[
\frac{1}{\Gamma_n^{(2)}}  \ES\left[\left({\cal M}_n(h_1)+{\cal N}_n(h_2)\right)^2\right]\xrightarrow{n\rightarrow+\infty}
\left(1-\frac{1}{ M}\right)\left[\frac 12\int_{\R}(h'_1\sigma)^2d\nu+2\int h_2^2 d\nu\right].
\]
In particular, when $h_1=\sigma g'$ and $h_2=\frac{1}{2} \sigma^2 g''$ (with $g={\cal Q}f$),  this variance is denoted by  $\sigma_{2,2}^2(f)$ which subsequently reads
\begin{equation}\label{eq:contribsecondordreterme2}
\sigma_{2,2}^2(f)= \left[\frac 12\int_{\R}(h'_1\sigma)^2d\nu+2\int h_2^2 d\nu\right]=\int \sigma^2\left((\sigma g'')^2+\sigma\sigma' g' g''+\frac 12 (\sigma'g')^2\right) d\nu.\end{equation}
\end{prop}

%

\begin{proof} $(a)$ Set $\varphi=h_1$. First, using that $\bar X$ and $\bar{Y}$ are built with the same Wiener increments, 
$$
{\langle {\cal M}(\varphi)\rangle_n}=\sum_{k=1}^n\frac{\gamma_k}{M} \sum_{m=0}^{M-1}\left(\varphi(\bar X_{k-1})-\varphi(\bar{Y}_{M(k-1)+m})\right)^2
$$
so that
$$
\frac{\langle {\cal M}(\varphi)\rangle_n}{\Gamma_n}=M\sum_{m=0}^{M-1}\hat{\nu}_n^{\gamma,m}(\hat \varphi^2)
$$
where $\hat{\nu}_n^{\gamma,m}(f)=\frac{1}{\Gamma_n}\sum_{k=1}^n \gamma_k f(\bar X_{k-1},\bar{Y}_{M(k-1)+m})$ and 
$\hat \varphi(x,y)=\varphi(x)-\varphi(y)$. With similar arguments as for the proof of Proposition~\ref{prop:convergence}, for every $m\in\{0,\ldots,M-1\}$, $(\hat{\nu}_n^{\gamma,m})_n$ converges $a.s.$ to the unique invariant distribution of the duplicated diffusion $\nu_\Delta$ (since Assumption $\mathbf{(C_w)}$ holds). By uniform integrability arguments, one can check that the convergence holds  along continuous functions with polynomial growth so that 
$$
\hat{\nu}_n^{\gamma,m}(\hat \varphi^2)\xrightarrow{n\rightarrow+\infty}\int \left(\varphi(x)-\varphi(y)\right)^2 \nu_\Delta(dx,dy)=0\quad a.s. 
$$
Again with uniform integrability arguments (using that $\sup_n\ES[|\bar X_n|^r]<+\infty$ for every positive $r$), one can check that $\ES[\hat{\nu}_n^{\gamma,m}(\hat \varphi^2)]\xrn{n\nrn}0$. It follows that 
$\ES\Big[\frac{\langle {\cal M}(\varphi)\rangle_n}{\Gamma_n}\Big]\xrn{n\nrn}0.$

\smallskip
\noindent $(b)$ The proof of this statement is the purpose of the end of the section. First, remark that the orthogonality of ${\cal M}(h_1)$ and  ${\cal N}(h_2)$ follows from independency of the increments of the Brownian motion and from the fact that 
for every $s<t$, $\ES[(W_t-W_s)((W_t-W_s)^2-(t-s))]=0$. Then, it remains to study these two martingales separately. In Lemma~\ref{lem:mdevarphi}, we go deeper in the study of the long run behavior of the martingale ${\cal M}(h_1)$ under Assumption $\mathbf{(C_s)}$ and in Lemma~\ref{lem:ndevarphi}, we investigate the one of the martingale ${\cal N}(h_2)$.
\end{proof}

\subsubsection{Long run behavior of ${\cal M}(\varphi)$ under strong confluence.} 
\begin{lemme}\label{lem:mdevarphi} Under the assumptions of Proposition~\ref{prop:contribmartmun}$(b)$, 
\[
\frac{1}{\Gamma_n^{(2)}}  \ES\left[{\cal M}_n(h_1)^2\right]\xrightarrow{n\rightarrow+\infty}
\frac 12 \Big(1-\frac{1}{ M}\Big)\int_{\R}(h_1'\sigma)^2d\nu.
\]
\end{lemme}

\begin{proof} We temporarily write $\varphi$ instead of $h_1$.

\smallskip
\noindent {\sc Step~1}: We decompose ${\cal M}(\varphi)$ as the sum of  terms involving the limiting diffusion process $X$:
\begin{eqnarray*}
{\cal M}(\varphi)&= & {\cal M}^{(1)} - \sum_{m=0}^{M-1} {\cal M}^{(2,m)}+\sum_{m=1}^{m-1}{\cal M}^{(3,m)}\\
%
\mbox{where}\qquad\qquad\qquad{\cal M}^{(1)}_n &=& \sum_{k=1}^n \Big(\varphi(\bar X_{k-1}) - \varphi(X_{\Gamma_{k-1}})\Big)\Delta W_{\Gamma_{k}},\\
{\cal M}^{(2,m)}_n &=& \sum_{k=1}^n \Big(\varphi(\bar Y_{M(k-1)+m}) - \varphi(X_{\Gamma_{k-1+\frac{m}{M}}})\Big)(W_{\Gamma_{k-1+\frac{m+1}{M}}}-W_{\Gamma_{k-1+\frac{m}{M}}}),\\
{\cal M}^{(3,m)}_n&=&  \sum_{k=1}^n \Big(  \varphi(X_{ \Gamma_{k-1+\frac{m}{M} }})-\varphi(X_{\Gamma_{k-1}})\Big) (W_{\Gamma_{k-1+\frac{m+1}{M} }}-W_{\Gamma_{k-1+\frac{m}{M} }}).
\end{eqnarray*}

We first deal with ${\cal M}^{(1)}$ whose predictable bracket  given by
\begin{eqnarray*}
\langle {\cal M}^{(1)}\rangle_n & \le & \sum_{k=1}^n \gamma_k \big(\varphi(\bar X_{k-1})-\varphi(X_{\Gamma_{k-1}})\big)^2\\
& \le & [\varphi]_{\rm Lip} \sum_{k=1}^n \gamma_k \big|\bar X_{k-1}- X_{\Gamma_{k-1}}\big|^2.
\end{eqnarray*}

Let $A^{(2)}$ be the infinitesimal generator of the duplicated diffusion $(X^x_t, X^{x'}_t)_{t\ge 0}$ and let us denote by $\widetilde{b}:\ER^d\times\ER^d\rightarrow\ER^d\times\ER^d$ and $\widetilde{\sigma}:\ER^d\times\ER^d\rightarrow\mathbb{M}_{2d,2q}$ the associated drift and diffusion coefficients. If we temporarily set $S(x,y)= (x-y)^2$, then 
\[
{A}^{(2)}S(x,y)= (b(x)-b(y))(x-y) +\frac 12\big(\sigma(x)-\sigma(y\big))^2.
\]
and  $\mathbf{(C_s)}$ reads, ${A}^{(2)}S \le -\alpha \, S$ or equivalently $ 0\le S \le - \frac{1}{\alpha} {A}^{(2)}S$. 

Now, by mimicking the proof of~\eqref{tclbasic2} (where the result has been established for functions of the Euler scheme alone), we get that, as soon as $\frac{\sqrt{\Gamma_n}}{\Gamma^{(2)}_n}\to 0$, for every smooth function $f:\ER^d\times\ER^d\rightarrow\ER$
\[
\frac{1}{\Gamma^{(2)}_n} \sum_{k=1}^n \gamma_k {A}^{(2)} f (X_{\Gamma_{k-1}}, \bar X_{k-1}) \stackrel{a.s.}{\longrightarrow} m(f)=\nu_\Delta\Big(\frac{1}{2} D^2 f(.).\widetilde{b(.)}^{\otimes 2}\Big)+\frac{1}{24} \ES\big[D^{(4)} f(.)(\sigma(.)U)^{\otimes 4}\big]
\]
where $U\sim{\cal N}(0,I_{q})$ and $\nu_{\Delta}$ is the image of $\nu$ on the diagonal of $\R^2$ (which is the unique invariant distribution of the duplicated diffusion).
Straightforward computations show that $m(S)=0$ since $\nabla S(x,y) = 2\begin{pmatrix}x-y\cr y-x\end{pmatrix}$, $D^{(2)}S(x,y) = 2\Big[\begin{array}{cc}1&-1\\-1&1\end{array}\Big]$ and $D^{(\ell)}S\equiv 0$, $\ell\ge 3$. Thus, taking advantage of the strong confluence, we derive that 
\[
\lim_n \frac{1}{\Gamma^{(2)}_n} \sum_{k=1}^n \gamma_k \big(X_{\Gamma_{k-1}}- \bar X_{k-1}\big)^2 \le -\frac{1}{\alpha} m(S)=0 \quad a.s. 
\]
%
%
%
Uniform integrability arguments imply that the above convergence also holds in $L^1$. Thus,
$$\ES\left[\frac{\langle {\cal M}^{(1)}\rangle_n }{\Gamma^{(2)}_n}\right]{\longrightarrow} \,0\quad\textnormal{ as $\quad n\to +\infty$}.$$ 

\smallskip
\noindent The same method of proof shows a similar result for ${\cal M}^{(2,m)}$, $m=0,\ldots,M-1$ (by considering the  scheme $(\bar Y_{M k+m})_{k\ge 0}$ and the filtration ${\cal G}_{k}^m= {\cal F}^W_{\Gamma_{k-1+\frac mM}}$). It follows that $  \ES\left[\frac{\langle {\cal M}^{(2,m)}\rangle_n }{\Gamma^{(2)}_n}\right]{\longrightarrow}\, 0$ as $n\to +\infty$.

\medskip
\noindent {\sc Step~2}: Now we deal with ${\cal M}^{(3,m)}$, $m=1,\ldots,M-1$.  First we compute the predictable bracket 
\[
\langle{\cal M}^{(3,m)}\rangle_n =\frac 1M\sum_{k=1}^n \gamma_k\big(\varphi(X_{ \Gamma_{k-1+\frac{m}{M} }})-\varphi(X_{\Gamma_{k-1}})\big)^2.
\]
Then, we decompose
\begin{align*}
\varphi(X_{ \Gamma_{k-1+\frac{m}{M} }})-&\varphi(X_{\Gamma_{k-1}})= \underbrace{\varphi'(X_{\Gamma_{k-1}})(X_{ \Gamma_{k-1+\frac{m}{M} } }-X_{\Gamma_{k-1}})}_{(a)_k}\\
&+ \underbrace{ \big(\varphi'(\Xi_{k-1})-\varphi'(X_{\Gamma_{k-1}})\big)(X_{ \Gamma_{k-1+\frac{m}{M} }}-X_{\Gamma_{k-1}})}_{(b)_k},\; \Xi_{k-1}\!\in (X_{\Gamma_{k-1}},X_{ \Gamma_{k-1+\frac{m}{M} } }).
\end{align*}
Let us deal first with $(b)_k$.  The function $\varphi''$ being with polynomial growth, there exists some positive $C$ and  $p$ such that for every $x$ and $y$ in $\ER^d$,
$$|\varphi'(x+y)-\varphi'(x)|\le C(1+|x|^p+|y|^p)|y|.$$

Thus, 
\[
  \frac{1}{\Gamma_n^{(2)} } \sum_{k=1}^n  \gamma_k (b)_k^2 \le \frac{C}{\Gamma_n^{(2)} }  \sum_{k=1}^n \gamma_k (X_{ \Gamma_{k-1+\frac{m}{M} } }-X_{\Gamma_{k-1}})^{4}(1+|X_{\Gamma_{k-1}})|^{2p})(1+|U_k|^{2p}).
\]
Using that $\sup_t\ES[| X_t^x|^r]<+\infty$, one easily checks that for every $r\ge 2$,
$$\sup_k\ES[|X_{ \Gamma_{k-1+\frac{m}{M} } }-X_{\Gamma_{k-1}}|^{r}]\le C\gamma_k^{\frac r2}$$ 
so that with the help of the Cauchy-Schwarz inequality,
$$\lim_n \frac{1}{\Gamma_n^{(2)} } \sum_{k=1}^n  \gamma_k \ES (b)_k^2 =0.$$
%
%

\smallskip For $(a)_k$ we write
\[
(a)_k= (\varphi'\sigma)\big(X_{\Gamma_{k-1}}\big)\big(W_{ \Gamma_{k-1+\frac{m}{M} }}-W_{\Gamma_{k-1}}\big) + (\widetilde a)_k
\]
where 
\[
 (\widetilde a)_k =  \varphi' \big(X_{\Gamma_{k-1}}\big)\left(\int_{\Gamma_{k-1}}^{ \Gamma_{k-1+\frac{m}{M} }}b(X_s)ds+ \int_{\Gamma_{k-1}}^{ \Gamma_{k-1+\frac{m}{M} }}\big(\sigma(X_s)-\sigma(X_{\Gamma_{k-1}})\big)dW_s\right).
\]

It is clear, owing  to Doob's Inequality,  that 
\[
\E  (\widetilde a)_k^2\le \|\varphi'\|_{\sup}^2\left({\gamma^2_k}  \sup_{t\ge 0} \E|b(X_t)|^2 +{\gamma_k}[\sigma]^2_{\rm Lip}\E \Big( \sup_{t\in [X_{\Gamma_{k-1}}, \Gamma_{k-\frac 12} )}|X_s-X_{\Gamma_{k-1}}|^2 \Big) \right)\le C_{b,\sigma,\varphi} \gamma^2_k. 
\]
Then  $\displaystyle   \frac{1}{\Gamma_n^{(2)} } \sum_{k=1}^n  \gamma_k (\widetilde a)_k^2\stackrel{L^1}{\longrightarrow} 0$ as above.

The last term of interest is again a martingale increment. We note that 
\[
\E \left((\varphi'\sigma)^2\big(X_{\Gamma_{k-1}}\big)\big(W_{ \Gamma_{k-1+\frac{m}{M} }}-W_{\Gamma_{k-1}}\big)^2\,|\, {\cal F}^W_{\Gamma_{k-1}}\right) = \frac{m\gamma_k}{M} (\varphi'\sigma)(X_{\Gamma_{k-1}})^2.
\]
The sequence $(\gamma_n, \gamma^2_n)_{n\ge 1}$ being averaging, 
\[
\frac{1}{\Gamma^{(2)}_n }\sum_{k=1}^n {\gamma^2_k} (\varphi'\sigma)(X_{\Gamma_{k-1}})^2\stackrel{a.s.}{\longrightarrow} \int_{\R} (\varphi'\sigma)^2d\nu\quad \mbox{ as }\quad n\to +\infty.
\]
Uniform integrability arguments imply that $\frac{1}{\Gamma^{(2)}_n }\sum_{k=1}^n {\gamma^2_k}\, \ES[(\varphi'\sigma)(X_{\Gamma_{k-1}})^2]{\longrightarrow} \int_{\R} (\varphi'\sigma)^2d\nu$
and one deduces that
$$  \ES\left[\frac{\langle {\cal M}^{(3,m)}\rangle_n }{\Gamma^{(2)}_n}\right]{\longrightarrow}\,\frac m{M^2} \int_{\R} (\varphi'\sigma)^2d\nu.$$
 The result then follows from the orthogonality of the martingales  ${\cal M}^{(3,m)}$, $m=1,\ldots,M-1$ (the fact that the  martingales ${\cal M}^1$ and ${\cal M}^{2,m}$ are negligible also implies by Schwarz's Inequality that so is  their cross product).
%
%
%
\end{proof}



\subsubsection{Long run  behavior of ${\cal N}(h_2)$.} 
We consider now the martingale
\begin{eqnarray*}
{\cal N}_n({h}_2) &= &{\cal N}^1_n-\sum_{m=0}^{M-1} {\cal N}^{2,m}\\
\mbox{where  }\hskip 3 cm  {\cal N}^1_n&=&  \sum_{k=1}^n h_2(\bar X_{k-1}) \big((W_{\Gamma_{k}}-W_{\Gamma_{k-1}})^2-\gamma_k\big)\\
{\cal N}^{2,m}_n &=&  \sum_{k=1}^n h_2(\bar Y_{M(k-1)+m})\big((W_{ \Gamma_{k-1+\frac{m+1}{M} }}-W_{ \Gamma_{k-1+\frac{m}{M} }})^2- \frac{\gamma_k}{M} \big).\hskip 1 cm
\end{eqnarray*}

\begin{lemme}\label{lem:ndevarphi} Under Assumptions of Proposition~\ref{prop:contribmartmun}$(b)$, 
\[
\frac{1}{\Gamma_n^{(2)}}  \ES\left[{\cal N}_n(h_2)^2\right]\xrightarrow{n\rightarrow+\infty}
2\left(1-\frac{1}{M}\right)\int_{\R}h_2^2d\nu.
\]
\end{lemme}
%
%
%
\begin{proof} Like in the previous proof, we write $\varphi$ instead of $h_2$. We focus on the asymptotic behavior of $\langle {\cal N} \rangle_n$.

\smallskip
\noindent  First, noting that for a random variable $Z\sim {\cal N}\big(0;1\big)$, $\E \big((Z^2-1)^2\big)=2$, we get since $(\gamma_n,\gamma^2_n)_{n\ge 1}$ is averaging,

\begin{equation}\label{eq:brack1}
\frac{\langle {\cal N}^1\rangle_n}{\Gamma^{(2)}_n} = \frac{2}{\Gamma^{(2)}_n}  \sum_{k=1}^n\gamma_k^2 \varphi^2(\bar X_{k-1})\longrightarrow2\int _{\ER} \varphi^2d\nu \quad  a.s. \quad  \mbox{ as }  \quad n\to +\infty.
\end{equation}
likewise one shows that for $m=0,\dots,M-1$ that  
\[
\frac{\langle {\cal N}^{2,m}\rangle_n}{\Gamma^{(2)}_n} \longrightarrow \frac 2{M^2}\int _{\ER} \varphi^2d\nu \quad a.s. \quad  \mbox{ as }  \quad n\to +\infty.
\]
By uniform integrability arguments, the above convergence extends to the expectations. Second, we focus on the ``slanted'' brackets.  Let us set $\Delta_{m,k}=(W_{ \Gamma_{k+\frac{m+1}{M} }}-W_{ \Gamma_{k+\frac{m}{M} }})^2- \gamma_k/M$.  Using the chaining rule for conditional expectations, we note that, for every $m\neq m'$,
\[
\ES_{k-1}\left(\varphi (\bar Y_{M(k-1)}+m)\varphi(\bar Y_{M(k-1)+m'})\Delta_{m,k-1}\Delta_{m',k-1}\right)=0
\]
so that $\langle {\cal N}^{2,m}, {\cal N}^{2,m'}\rangle_n\equiv 0$. 

\smallskip
\noindent Now, let us compute $\langle {\cal N}^1, {\cal N}^{2,m'}\rangle_n$ where $m\in\{0,\ldots,M-1\}$ and $({\cal N}^1, {\cal N}^{2,m'})$ is viewed as a couple of $({\cal F}_k)$-martingales.
Writing the increment $W_{\Gamma_k}-W_{\Gamma_{k-1}}$ as follows:
$$W_{\Gamma_k}-W_{\Gamma_{k-1}}=\big(W_{\Gamma_k}-W_{\Gamma_{k-1+\frac{m+1}{M}}}\big)+ \big(W_{\Gamma_{k-1+\frac{m+1}{M}}}-W_{\Gamma_{k-1+\frac{m}{M}}}\big)+\big(W_{\Gamma_{k-1+\frac{m}{M}}}-W_{\Gamma_{k-1}}\big)$$
and using some standard properties of the increments of the Brownian Motion, one can check that 
\[
\langle {\cal N}^1, {\cal N}^{2,m'}\rangle_n = \frac 2{M^2} \sum_{k=1}^n {\gamma^2_k}\varphi(\bar X_{k-1})\ES\left(\varphi(\bar Y_{M(k-1)+m})|{\cal F}_{k-1}\right).
\]
Using second order Taylor expansions of $\varphi$ between $\varphi(\bar Y_{M(k-1)+\ell-1})$  and  $\varphi(\bar Y_{M(k-1)+\ell})$ for $\ell=1,\ldots,m$, combined with the fact that $\sup_j\ES[|\bar Y_{j}|^r]<+\infty$ for every $r>0$, one derives 
\[
\langle {\cal N}^1, {\cal N}^{2,m'}\rangle_n = \frac 2{M^2} \sum_{k=1}^n {\gamma^2_k}\left(\varphi(\bar X_{k-1})\varphi(\bar Y_{M(k-1)})+O_{L^1}(\gamma_k)\right)=\frac{\Gamma_n^{(2)}}{2M^2}\hat{\nu}_n^{\gamma,\gamma^2}(\varphi\otimes\varphi)+O_{L^1}(\Gamma_n^{(3)})
\]
where $\hat{\nu}_n^{\gamma,\gamma^2}(f)=\frac{1}{\Gamma_n^{(2)}}\sum_{k=1}^n f(\bar X_{k-1},\bar Y_{M(k-1)})$. Thus, the sequence $(\hat{\nu}_n^{\gamma,\gamma^2})_n$ of empirical measures associated to  the duplicated diffusion~\eqref{eq:dupdif2} has a   unique invariant distribution  $\nu_\Delta$. By an adaptation of the proof of Proposition~\ref{prop:convergence}, it can thus be proved that
$$
\hat{\nu}_n^{\gamma,\gamma^2}(\varphi\otimes\varphi)\xrn{n\nrn}\nu_\Delta(\varphi\otimes\varphi)=\int\varphi^2d\nu.
$$
Once again, by a uniform integrability argument (and using what precedes), one obtains
\[
\frac{1}{\Gamma_n^{(2)}}\ES[\langle {\cal N}^1, {\cal N}^{2,m'}\rangle_n]\xrn{n\nrn}\frac 2{M^2}\int\varphi^2d\nu.
\]
As a conclusion of the previous convergences, one deduces that
\[
\frac{1}{\Gamma_n^{(2)}}\ES\left[\Big\langle {\cal N}^1-\sum_{m=0}^{M-1} {\cal N}^{2,m'}\Big\rangle_n\right]\xrn{n\nrn}\left(2+\frac{2}{M^2}(M-2M)\right)\int\varphi^2d\nu=2\left(1-\frac{1}{M}\right)\int\varphi^2d\nu.
\]
\end{proof}
\section{Proofs of the main theorems ($CLT$ and optimization)}\label{sec:proof3}
Owing to the results established in the previous sections, we are now in position to prove the three main results: Theorems~\ref{theo:CLT},~\ref{L2theo} and~\ref{thm:optimiz}. 
First keep in mind that in these theorems the step sequence reads $\gamma_n= \gamzero  n^{-a}$ for some $\gamzero >0$ and $a\!\in (0,1)$.

\subsection{\bf Proof of Theorem~\ref{theo:CLT}.}  We mainly detail the proof of Theorem~\ref{theo:CLT}$(b)$ and we will only give some elements of the ones of $(a)$ and $(c)$ (which are based on the same principle) at the end of this section. \smallskip

\noindent First, by~\eqref{defML2Rgodic}, one reminds that $\widetilde{\nu}_n^{(R,{\bf W})}$ is a linear combination of $ \nu_{n_1}$ and of $\mu_{n_r}^{(r,M)}$ with $n_r=\lfloor q_r n\rfloor$, $r=2,\ldots,R$, . For $\nu_{n_1}(f)$ and $\mu_{n_2}^{(2,M)}(f)$, 
we will make use of the expansions given in Propositions~\ref{prop:Multistep}$(b)$  and~\ref{prop:dvpcorrectiveterms}$(b)$ respectively. For  $\mu_{n_r}^{(r,M)}(f)$, $r=3,\dots,R$, as defined by~\eqref{eq:tildenunrr}, we apply Proposition~\ref{prop:dvpcorrectiveterms}$(b)$
with step sequence $(\gamma_n/M^{r-2})_{n\ge 1}$.  More precisely, by  \eqref{eq:Gammalr},
$$(M^{1-\ell}-1)\frac{\Gamma_{n_r}^{(\ell,r)}}{\Gamma_{n_r}^{(1,r)}}=m_{r,\ell}\frac{\Gamma_{n_r}^{(\ell)} }{\Gamma_{n_r}}\quad\textnormal{with   $\;m_{r,\ell}=(M^{1-\ell}-1)M^{-(r-2)(\ell-1)}$,}$$
so that by Proposition~\ref{prop:dvpcorrectiveterms}$(b)$, we have for every $r\in\{2,\ldots,R\}$,
\begin{align*}\nonumber \mu_{n_r}^{(r,M)}(f)-\sum_{\ell=2}^{\ERR}m_{r,\ell} \frac{\Gamma_{n_r}^{(\ell)} }{\Gamma_{n_r}}\nu(\Psi_{\ell}(f))&=&c_{_{R+1}}m_{r,R+1}\frac{\Gamma_{n_r}^{(\ERR+1)} }{\Gamma_{n_r}}
\label{eq:munrDL}
-\frac{M^{r-2}{\cal M}_{n_r}^{(r)}(\sigma g')}{\Gamma_{n_r}}+o_{L^2}\left(\frac{\sqrt{\Gamma_{n_r}}\vee \Gamma_{n_r}^{(R+1)}}{\Gamma_{n_r}}\right)
\end{align*}
where 
 $({\cal M}_{n_r}^{(r)})_{n\ge 1}$ is defined similarly to ${\cal M}_n$ but with the step sequence $(\gamma_n/M^{r-2})_{n\ge 1}$ In particular, $(\bar X_n,\bar Y_{M n+m})$ is now a couple of Euler schemes with step sequences $(\gamma_n/M^{r-2})_{n\ge 1}$ and $(\gamma_n/M^{r-1})_{n\ge 1}$ respectively.

 %
It follows from the expansions of {\em order $R+1$} of each term of $\widetilde{\nu}_n^{(R,{\bf W})}$ established  in Propositions~\ref{prop:Multistep}$(b)$  and~\ref{prop:dvpcorrectiveterms}$(b) $ respectively that 
\begin{equation}\label{roorzpopopa}
\begin{split}
\widetilde{\nu}_n^{(R,{\bf W})}(f)-\nu(f)&= \nu_{n_1}(f)-\nu(f)+\sum_{r=2}^R {\bf W}_r \mu_{n_r}^{(r,M)}(f)\\
&= c_{_{R+1}}\widetilde{\mathbf{W}}_{_{R+1}}\frac{\Gamma_{n}^{(R+1)} }{\Gamma_{n}}+\frac{M_{n_1}^{(1,g)}}{\Gamma_{n_1}}-\sum_{r=2}^R {\bf W}_r \frac{M^{r-2}{\cal M}_{n_r}^{(r)}(\sigma g')}{\Gamma_{n_r}}\\
&\quad+ {\rm Bias}^{(1)}\!(a,R,q,n)+{\rm Bias}^{(2)}\!(a,R,q,n)+o_{L^2}\left(\frac{\sqrt{\Gamma_{n}}\vee \Gamma_n^{(R+1)}}{\Gamma_{n}}\right)
\end{split}
\end{equation}
where ${\rm Bias}^{(1)}(a,R,q,n)$ is defined in Lemma~\ref{lem:bias1}$(b)$ and 
\begin{equation*}
{\rm Bias}^{(2)}(a,R,q,n) 
= c_{_{R+1}}  \Big[ {\mathbf W} _1
\left(\frac{\Gamma_{n_1}^{(\ERR+1)} }{\Gamma_{n_1}}-q_1^{-aR}\frac{\Gamma_{n}^{(R+1)}}{\Gamma_n}\right)
 +\sum_{r=2}^R {\mathbf W} _r m_{r,R+1}  \left(\frac{\Gamma_{n_r}^{(\ERR+1)} }{\Gamma_{n_r}}-q_r^{-aR}\frac{\Gamma_{n}^{(R+1)}}{\Gamma_n}\right)
 \Big].
\end{equation*}
By Lemma~\ref{lem:bias1}
\begin{equation}\label{sdjdoijdoi}
\Big|{\rm Bias}^{(1)}(a,R,q,n)\Big|+\Big|{\rm Bias}^{(2)}(a,R,q,n)\Big|\le \frac{C}{n^{1-a}}=o\left(\frac{1}{\sqrt{\Gamma_n}}\right).
\end{equation}

As concerns the martingale components, one deduces from Propositions~\ref{prop:contribmartnun}$(a)$ and~\ref{prop:contribmartmun}$(a)$ that 
 \[
{\sqrt{\Gamma_{n_1}}}  \left(\frac{M_{n_1}^{(g)}}{\Gamma_{n_1}}-\sum_{r=2}^R {\bf W}_r \frac{M^{r-2}{\cal M}_{n_r}^{(r)}(\sigma g')}{\Gamma_{n_r}}\right)\stackrel{(\R)}{\Longrightarrow}{\cal N}\Big(0;\int_{\R}(\sigma g')^2d\nu\Big).
\]
Theorem~\ref{theo:CLT}$(b)$ then follows by Slutsky Theorem and the following remarks:
$$\Gamma_{n_1}\overset{n\rightarrow+\infty}{\sim} \frac{\gamma_1q_1^{1-a}}{1-a} n^{1-a},\quad \Gamma_n^{(R+1)}{\Gamma_n}=\frac{1-a}{1-a(R+1)} \gamma_1^{R} n^{aR}$$
and that when $a=\frac{1}{2R+1}$,
$$ 1-a=2 aR=\frac{2R}{2R+1}\quad\textnormal{and} \quad \frac{1-a}{1-a(R+1)}=2.$$

For the proof of Theorem~\ref{theo:CLT}$(c)$, the only difference comes from the fact that the martingale component becomes negligible since $1-a>2aR$ when $a\in(0,(2R+1)^{-1})$ so that $(\widetilde{\nu}_n^{(R,{\bf W})})_{n\ge 1}$ converges in probability towards
$m_f(a,q,R)$. Finally,  the proof of Theorem~\ref{theo:CLT}$(a)$ follows the same lines but with the help of the expansions of  Propositions~\ref{prop:Multistep}$(a)$  and~\ref{prop:dvpcorrectiveterms}$(a)$.

\subsection{\bf Proof of Theorem~\ref{L2theo}.}
\noindent $(a)$ is an $L^2$-version of Theorem~\ref{theo:CLT}$(b)$ so that it relies on  the same decomposition. More precisely,  it is a direct consequence of \eqref{roorzpopopa} and
\eqref{sdjdoijdoi} combined with  Propositions~\ref{prop:contribmartnun}$(b)$ and~\ref{prop:contribmartmun}$(a)$.

\medskip
 \noindent 

\noindent Claim~$(b)$ is based on the (sharper) second expansions of Propositions~\ref{prop:Multistep}$(c)$  and~\ref{prop:dvpcorrectiveterms}$(c)$ up to order $R+2$. More precisely, using the same strategy as in \eqref{roorzpopopa}, one obtains

\begin{equation*}
\begin{split}
(\widetilde{\nu}_n^{(R,{\bf W})}(f)-&\nu(f))=  c_{_{R+1}}\widetilde{\mathbf{W}}_{_{R+1}}\frac{\Gamma_{n}^{(R+1)} }{\Gamma_{n}}+ c_{_{R+2}}\widetilde{\mathbf{W}}_{_{R+2}}\frac{\Gamma_{n}^{(R+2)} }{\Gamma_{n}}\\
&+\frac{M_{n_1}^{(1,g)}+N_{n_1}}{\Gamma_{n_1}}-\sum_{r=2}^R {\bf W}_r \frac{M^{r-2}\left({\cal M}_{n_r}^{(r)}(\sigma g')+{\cal N}_{n_r}^{(r)}(\frac{1}{2}\sigma^2 g'')\right)}{\Gamma_{n_r}}\\
&\quad+ \sum_{i=1}^3 {\rm Bias}^{(i)}\!(a,R,q,n)+\eta_n^{(1)}+
\eta_n^{(2)}
\end{split}
\end{equation*}
where $\widetilde{\mathbf{W}}_{_{R+2}}$ is defined by \eqref{eq:Wtilde} (and explicitly given by \eqref{eq:Wtilde2}), ${\rm Bias}^{(3)}$ is given by
\begin{equation*}
{\rm Bias}^{(3)}(a,R,q,n) 
= c_{_{R+2}}  \left[ {\mathbf W} _1
\left(\frac{\Gamma_{n_1}^{(\ERR+2)} }{\Gamma_{n_1}}-q_1^{-aR}\frac{\Gamma_{n}^{(R+2)}}{\Gamma_n}\right)
 +\sum_{r=2}^R {\mathbf W} _r m_{r,R+2}  \left(\frac{\Gamma_{n_r}^{(\ERR+2)} }{\Gamma_{n_r}}-q_r^{-aR}\frac{\Gamma_{n}^{(R+2)}}{\Gamma_n}\right)
 \right]
\end{equation*}
 and $\eta_n^{(1)}$ (resp. $\eta_n^{(2)}$) denotes a remainder term induced by the coarse level (resp. by the levels $r=2,\ldots,R$). By
 Propositions~\ref{prop:Multistep}$(c)$  and~\ref{prop:dvpcorrectiveterms}$(c)$, one obtains when $a=1/(2R+1)$,
 $$\|\eta_n^{(1)}\|_2=o(n^{-\frac{R+1}{2R+1}})\quad\textnormal{and}\quad \eta_n^{(2)}={\cal S}_n+o(n^{-\frac{R+1}{2R+1}}),$$
 where ${\cal S}_n$ is a centered random variable independent of $M_{n_1}^{(1,g)}$ and $N_{n_1}$ and such that $\ES[{\cal S}_n^2]=o(\frac{\Gamma_n^{(2)}}{\Gamma_n^2})=o(\frac{1}{n})$ (In fact, for $\eta_n^{(2)}$, one is slightly more precise than in Proposition~\ref{prop:dvpcorrectiveterms}$(c)$ by separating the martingale component and the bias component in the $o_{L_2}$).\smallskip

\noindent On the other hand, we obtain similarly to \eqref{sdjdoijdoi}:
\begin{equation*}
\sum_{i=1}^3|{\rm Bias}^{(i)}(a,R,q,n)|\le \frac{C}{n^{1-a}}=\frac{C}{n^{\frac{2R}{2R+1}}}.
\end{equation*}
With the help of these properties (and from the independence of the strata), we  deduce that 
\begin{align*}
\|(&\widetilde{\nu}_n^{(R,{\bf W})}(f)-\nu(f))\|_2^2=\left( c_{_{R+1}}\widetilde{\mathbf{W}}_{_{R+1}}\frac{\Gamma_{n}^{(R+1)} }{\Gamma_{n}}+ c_{_{R+2}}\widetilde{\mathbf{W}}_{_{R+2}}\frac{\Gamma_{n}^{(R+2)} }{\Gamma_{n}}\right)^2\\
&+\ES\left[\left(\frac{M_{n_1}^{(1,g)}+N_{n_1}}{\Gamma_{n_1}}\right)^2\right]+\sum_{r=2}^R{\bf W}_r^2 M^{2(r-2)} \ES\left[\left(\frac{{\cal M}_{n_r}^{(r)}(\sigma g')+{\cal N}_{n_r}^{(r)}(\frac{1}{2}\sigma^2 g'')}{\Gamma_{n_r}}\right)^2\right]+o\left(\frac{1}{n}\right).
\end{align*}
The result is then a consequence of  Propositions~\ref{prop:contribmartnun}$(c)$ and~\ref{prop:contribmartmun}$(c)$ combined with the following expansion available for any $\rho\in(0,1)$: $\sum_{k=1}^n k^{-\rho}= (1-\rho)^{-1} n^{1-\rho}+O(1)$ (see \eqref{eq:Gamma-a}). In particular, it is worth noting that  when $a=1/(2R+1)$,
$$\frac{\Gamma_n^{(R+1)}\Gamma_n^{(R+2)}}{\Gamma_n^2}\overset{n\rightarrow+\infty}\sim \frac{4R}{R-1}\frac{\gamzero^{2R+1}}{n},$$
which induces the rectangular term $\widetilde{m}_f(q,R)$.

 \subsection{Proof of Theorem~\ref{thm:optimiz}}
 {\sc Step~1}({\em Optimization of the   step parameter $\gamzero $}): This step is devoted to the optimization of the starting step $\gamzero $, in order  to {\em equalize the impact of the bias and of  the variance} in the first term of the expansion of the $MSE$ in~\eqref{eq:MSE}. It amounts to solving the elementary minimization problem
\[
\min_{\gamzero >0} \left[ \sigma^2_f(\varpi)+m^2_f(\varpi)=R^{\frac{2R}{2R+1}}\Big( \frac{2R}{2R+1}\sigma^2_1(f)\gamzero ^{-1} +4\gamzero ^{2R}  M^{-R(R-1)}c^2_{_{R+1}}\Big) \right].
\] 
\[
\] 
We rely on the following elementary lemma (whose proof is left to the reader).

\begin{lemme}Let $A,B,R>0$. Then,
\[
u^*:={\rm argmin}_{u>0}\left[Au^{-1} +B u^{2R}\right] =\left( \frac{A}{2RB}\right)^{\frac{1}{2R+1}}
\] 
and

\[
\min_{u>0}\left[Au^{-1} +B u^{2R}\right] =  (2R+1)B(u^*)^{2R}=A^{\frac{2R}{2R+1}}B^{\frac{1}{2R+1}} \big(2R\big)^{\frac{1}{2R+1}}\left(1+\frac{1}{2R}\right).
\]

\end{lemme}

\noindent Consequently, 
\[
\min_{\gamzero >0} \left[ \sigma^2_f(q,R)+m^2_f(q,R) \right] = \left(2^{\frac1R}R(2R+1)^{\frac{1}{2R}}M^{-\frac{R-1}{2}}\sigma_1^2(f)  |c_{_{R+1}}|^{\frac 1R}\right)^{\frac{2R}{2R+1}}
\]
attained at $\gamzero ^*= \gamzero ^*(R,M)$ given by 
\begin{equation}\label{eq:gamma_0}
\gamzero ^*= \Big(\frac{2R}{2R+1}\Big)^{\frac{1}{2R+1}}(8R)^{-\frac{1}{2R+1}}|c_{_{R+1}}|^{-\frac{2}{2R+1}}\sigma_1(f)^{\frac{2}{2R+1}}M^{\frac{R(R-1)}{2R+1}}.
\end{equation}

\smallskip
\noindent {\sc Step~2} ({\em Optimization of the size of the coarse level}): We introduce an auxiliary allocation parameter $\rho\!\in (0,1)$ to dispatch the target global $MSE$ $\varepsilon^2$  
so that the contribution of the first and the second  term in the right hand side of~\eqref{eq:MSE}   are $\rho \varepsilon^2$  and  $(1-\rho)\varepsilon^2$ respectively. The first of these two equalities reads 
\[
n^{-\frac{2R }{2R+1}}\left[ \sigma^2_f(\varpi)+m^2_f(\varpi)\right]\le  \rho \varepsilon^2
\]
where the step parameter $\gamzero = \gamma^*_0(R, M)$ is given by~\eqref{eq:gamma_0}. One straightforwardly derives that 
\begin{equation}\label{eq:N_opt}
n= n(\varepsilon, R, M,\rho) = \left\lceil\rho^{-(1+\frac{1}{2R})}\mu(R)  R\,\sigma_1^2(f)M^{-\frac{R-1}{2}}\varepsilon^{-2-\frac 1R}\right\rceil
\end{equation}
where 
\[
\mu(R)=2^{\frac 1R}(2R+1)^{\frac{1}{2R}}|c_{_{R+1}}|^{\frac 1R} \longrightarrow \widetilde c\quad\mbox{ as } \quad R\to+\infty.
\]

\noindent {\sc Step~3} (\!{\em Calibrating the depth $R$}): To calibrate $R=R(\varepsilon)$, we will now deal with the second term $\frac{\widetilde \sigma^2_f + \widetilde m_f}{n}$ of the $MSE$ expansion~\eqref{eq:MSE}. Since we have no clue on the sign of the residual bias term   $\widetilde m_f(\bar q,R,\gamzero )$,  we will replace it by its absolute value.  Moreover, we can plug in its formula the above expression~\eqref{eq:gamma_0} of the optimal step size $\gamzero ^*(R,M)$ which yields
\[
|\widetilde m_f(\bar q,R,\gamzero ^*)|= \mbox{\bf 1}_{\{c_{_{R+1}}\neq 0\}} \frac{|c_{_{R+2}}|}{|c_{_{R+1}}|} \frac{R}{R-1}\frac{1-M^{-R}}{1-M^{-1}}\sigma_1^2(f).
\]
Consequently, using the function $\mathbf{\Psi}$ introduced in~\eqref{eq:BoldPsi} and the obvious fact that $1-M^{-R}\le 1$,  this second  term will  be upper-bounded by $(1-\rho)\varepsilon^2 $ as soon as 
\begin{equation}\label{eq:OptTerm2}
\frac{R}{n(\varepsilon)}\left(\eta(f,R,M)\sigma_1^2(f)+\sigma^2_{2,1}(f)+\mathbf{\Psi}(M)R\left(1-\frac 1M\right)\sigma_{2,2}^2(f)\right)\le (1-\rho)\varepsilon^2 .
\end{equation}
where $\eta(f,R,M)= \mbox{\bf 1}_{\{c_{_{R+1}}\neq 0\}} \frac{|c_{_{R+2}}|}{|c_{_{R+1}}|} \frac{1}{(R-1)(1-M^{-1})}\to 0$ as $R\to+ \infty$ owing to the assumption made on the sequence $(c_r)_{r\ge 1}$.  

Given the expression obtained for $n(\varepsilon, R, M,\rho) $, this inequality is satisfied in turn as soon as 
\[
\sigma^2_{2,1}(f)+\mathbf{\Psi}(M)R\left(1-\frac 1M\right)\,\sigma_{2,2}^2(f) \le (1-\rho)  \rho^{-(1+\frac{1}{2R})}\mu(R) \, \sigma_1^2(f)M^{-\frac{R-1}{2}}\varepsilon^{-\frac 1R},
\]
or equivalently
\begin{equation}\label{eq:condrho}
\varepsilon^{\frac 1R} M^{\frac{R-1}{2}}R \le \frac{1-\rho}{\rho}\rho^{-\frac{1}{2R}}\frac{\mu(R) \theta_1(f)}{\left(1-\frac 1M\right)\mathbf{\Psi}(M)+R^{-1}\big(\theta_2(f)+\eta(f,R,M)\big)},
\end{equation}
where
\[
\theta_1(f)=\frac{\sigma_1^2(f)}{\sigma_{2,2}^2(f)} \quad \mbox{ and }\quad  \theta_2(f) =\frac{\sigma_{2,1}^2(f)}{\sigma_{2,2}^2(f)} 
\]
Note that under the assumptions made on the sequence $(c_r)_{r\ge 1}$, $\theta_3(f)$, In order to ensure the above condition,   we begin by rewriting the left-hand side as follows:
\begin{equation}\label{eq:logepsir}
 \varepsilon^{\frac 1R} M^{\frac{R-1}{2}}R=\exp\left( \frac{1}{R}\left(\frac{\log M}{2}R(R-1)+R\log R+\log {\varepsilon}\right)\right)
 \end{equation}
and will apply the next lemma with $\delta=(\log M)/2$ and $R=\lceil x(\varepsilon)\rceil$:
\begin{lemme} Let $\delta\!\in (0,+\infty)$. Then, for every  $\varepsilon\!\in (0,1]$, there exists a unique $x(\varepsilon)\!\in (1,+\infty)$ solution to
\[
\delta x(x-1)+x\log x +\log({\varepsilon})=0.
\]
The function $\varepsilon \mapsto x(\varepsilon)$ is increasing and satisfies  
\begin{equation}\label{eq:xeps1}
\lim_{\varepsilon\to 0} x(\varepsilon)=+\infty,\quad x(\varepsilon) \le \frac 12+ \sqrt{\frac{\log\big(\frac{1}{\varepsilon}\big)}{\delta}+\frac 14}
\end{equation}
 and 
\begin{equation}\label{eq:xeps2}
x(\varepsilon) = \sqrt{\frac{\log\big(\frac{1}{\varepsilon}\big)}{\delta}}  -\frac{\log_{(2)} \!\big(\frac{1}{\varepsilon}\big)}{4\delta} +\frac 12 +\frac{\log \delta}{4\delta}  
+O\left( \frac{ \log_{(2)}\!\big( 1/\varepsilon\big)   }{\sqrt{\log\big( 1/\varepsilon\big)}}\right)\quad \mbox{ as }\quad \varepsilon \to 0
\end{equation}
where $\log_{(2)} x = \log\log x$, $x>1$.
\end{lemme}

\noindent {\bf Proof.} The function $h : (\varepsilon,x)\mapsto \delta x(x-1)+x\log x +\log {\varepsilon} $ defined on $(0,1)\times [1,+\infty)$ is continuous,  increasing in both $\varepsilon$ and $x$, $h(\varepsilon,1)=  \log {\varepsilon} \le 0$ and $\lim_{x\to +\infty} h(\varepsilon,x)= +\infty$ which ensures the existence of a unique solution $x(\varepsilon)\!\in [1,+\infty)$ to the equation $h(\varepsilon,x)=0$. The monotony of $x(\varepsilon)$ follows from that of $h$. Its limit at infinity follows from the fact that $\lim_{\varepsilon\to 0}h(\varepsilon,x)=+\infty$ and  the inequality in~\eqref{eq:xeps1} is a consequence of the fact that $\delta x(\varepsilon)^2-\delta x(\varepsilon) - \log\big(\frac{1}{\varepsilon}\big)\le 0$ as $x(\varepsilon)\ge 1$.  For the expansion, we first note that $x(\varepsilon)$ satisfies
the second order equation
$$\delta x(\varepsilon)^2 +b x(\varepsilon) - \log\Big(\frac{1}{\varepsilon}\Big)= 0$$
with $b=\log \big(x(\varepsilon)/\alpha\big)$ where $ \alpha=\exp(\delta) $ so that
\begin{equation}\label{eq:aux0}
x(\varepsilon) = \sqrt{\frac{\log\big(\frac{1}{\varepsilon}\big)}{\delta}}\left(\sqrt{1+ \frac{(\log \big(x(\varepsilon)/\alpha\big))^2}{4\delta\log\big(\frac{1}{\varepsilon}\big)}}-\frac{\log \big(x(\varepsilon)/\alpha\big)}{2\sqrt{\delta\log\big(\frac{1}{\varepsilon}\big)} }\right).
\end{equation}
 We derive from the inequality in Equation~\eqref{eq:xeps1} that, for small enough $\varepsilon $,
\[
0\le \frac{\log \big(x(\varepsilon)/\alpha\big)}{\sqrt{\log\big(\frac{1}{\varepsilon}\big)}}= O\left(\frac{\log_{(2)}\!\big(1/\varepsilon\big) }{\sqrt{\log\big(1/\varepsilon\big)}}\right)=o(1)\quad \mbox{ as }\quad \varepsilon \to 0.
\]
Consequently, we derive from~\eqref{eq:aux0} that
\begin{equation}\label{eq:Expan0}
x(\varepsilon)=  \sqrt{\frac{\log\big(\frac{1}{\varepsilon}\big)}{\delta}}\left(1+ O\left(\frac{\log_{(2)}\!\big( 1/\varepsilon\big) }{\sqrt{\log\big( 1/\varepsilon\big)}}\right)\right)
\end{equation}
so that $\displaystyle \log  x(\varepsilon)  =\frac 12\left(\log_{(2)}\!\big(1/\varepsilon\big)-\log \delta\right)+  O\left(\frac{\log_{(2)}\!\big( 1/\varepsilon\big) }{\sqrt{\log\big( 1/\varepsilon\big)}}\right)$. 
Plugging this  back  into~\eqref{eq:Expan0} yields
\begin{eqnarray*}
x(\varepsilon)&= &  \sqrt{\frac{\log\big(\frac{1}{\varepsilon}\big)}{\delta}} \left(1-\frac{\log_{(2)} \!\big(\frac{1}{\varepsilon}\big)-\log \delta -2\, \delta}{4\sqrt{\delta\log \big(\frac{1}{\varepsilon}\big)} } +O\left( \frac{ \log_{(2)}\!\big( 1/\varepsilon\big)   }{\log\big( 1/\varepsilon\big)}\right) 
\right)\\
&=&  \sqrt{\frac{\log\big(\frac{1}{\varepsilon}\big)}{\delta}}  -\frac{\log_{(2)} \!\big(\frac{1}{\varepsilon}\big)}{4\delta}+\frac 12 +\frac{\log \delta}{4\delta}  +O\left( \frac{ \log_{(2)}\!\big( 1/\varepsilon\big)   }{\sqrt{\log\big( 1/\varepsilon\big)}}\right) .\quad \cqfd
\end{eqnarray*}

Now let  $x(\varepsilon,M)$ be the solution of the above equation where $\delta=\delta(M)= \frac{\log M}{2}$. We have

\[
 x(\varepsilon,M)=  \sqrt{\frac{2\log \big(\frac{1}{\varepsilon}\big)}{\log M}}
 -\frac{\log_{(2)}\! \big(\frac{1}{\varepsilon}\big)}{2\log M} +\frac 12 + \frac{\log(\log M)-\log2}{2\log M}+O\left( \frac{ \log_{(2)}\!\big( 1/\varepsilon\big)   }{\sqrt{\log\big( 1/\varepsilon\big)}}\right).
\]

Now, we set
\[
R(\varepsilon) = R(\varepsilon,M) = \lceil x(\varepsilon,M)\rceil.
\]
We derive  from the above lemma the following useful estimates for $R(\varepsilon)$:
$$
R(\varepsilon)\sim \sqrt{\frac{2\log \big(\frac{1}{\varepsilon}\big)}{\log M}}
\xrightarrow{\varepsilon\rightarrow0} +\infty\quad  \textnormal{and}\quad   R(\varepsilon) \le\frac 32+ \sqrt{\frac{2\log\big(\frac{1}{\varepsilon}\big)}{\log M}+\frac 14}.
$$

Now, it follows from the very definitions of $x(\varepsilon,M)$ and $R(\varepsilon)$  that
  $$
  h(\varepsilon, R(\varepsilon))\ge h(\varepsilon,x(\varepsilon,M))=0\ge h(\varepsilon,R(\varepsilon)-1).
  $$   
where $h$ is defined in the proof of the previous lemma.   Plugging these inequalities into~\eqref{eq:logepsir} yields 
\begin{equation}\label{eq:ineqR}
1\le \varepsilon^{\frac{1}{R(\varepsilon)}} M^{\frac{R(\varepsilon)-1}{2}}R(\varepsilon)\le M\left(1-\frac{1}{R(\varepsilon)}\right)^{-1+\frac{1}{R(\varepsilon)}} \left(\frac{R(\varepsilon)}{M}\right)^{\frac{1}{R(\varepsilon)}}.
\end{equation}
The above  inequality on the right implies  that~\eqref{eq:condrho} will be true as soon as   $\rho=\rho(\varepsilon,M) \!\in (0,1)$ satisfies
\[
\frac{1-\rho}{\rho}\rho^{-\frac{1}{2R(\varepsilon)}}\ge M\left(1-\frac{1}{R(\varepsilon)}\right)^{-1+\frac{1}{R(\varepsilon)}} \left(\frac{R(\varepsilon)}{M}\right)^{\frac{1}{R(\varepsilon)}}
\left(\frac{R^{-1}(\varepsilon)\theta_2(f)+\left(1-\frac 1M\right){\bf \Psi}(M)}{\mu(R(\varepsilon))\theta_1(f)}\right).
\]
In fact, one will try to saturate the  above condition, $i.e.$ to choose $\rho(\varepsilon,M) $ such that
\[
\frac{1-\rho(\varepsilon,M) }{\rho(\varepsilon,M) }\rho(\varepsilon,M) ^{-\frac{1}{2R(\varepsilon)}}= M\left(1-\frac{1}{R(\varepsilon)}\right)^{-1+\frac{1}{R(\varepsilon)}} \left(\frac{R(\varepsilon)}{M}\right)^{\frac{1}{R(\varepsilon)}}
\left(\frac{R(\varepsilon)^{-1}\theta_2(f)+\left(1-\frac 1M\right){\bf \Psi}(M)}{\mu(R(\varepsilon))\theta_1(f)}\right).
\]


As the function $\rho\mapsto \frac{1-\rho}{\rho}\rho^{-\frac{1}{2R(\varepsilon)}} $ is a decreasing homeomorphism from $(0,1) $ onto $(0,+\infty)$ this equation always has a solution $\rho=\rho(\varepsilon,M)$. Unfortunately it turns out to be  of little interest in its present form for practical implementation   since both  $\theta_i(f)$ are unknown.

\medskip
However,  as $R(\varepsilon)\rightarrow+\infty$ as $\varepsilon\rightarrow0$,
 and $\mu(R(\varepsilon)) \to  \widetilde c  $ as $\varepsilon \to 0$, we derive that  
\[
\frac{1-\rho(\varepsilon,M)}{\rho(\varepsilon,M)}\sim \frac{ (M-1) \mathbf{\Psi}(M)}{\widetilde c\, \theta_1(f)}
\quad i.e. \quad 
\rho(\varepsilon,M)\sim \frac{1}{1+ \frac{ {(M-1)}{\bf \Psi}(M)}{\widetilde c\, \theta_1(f)}} \quad \mbox{ as }\varepsilon \to 0.
\]

\noindent {\sc Step~4 }(\!{\em MSE, number of iterations and resulting complexity}):

\medskip
\noindent $\rhd$ Resulting $MSE$: From what precedes, we deduce that  after $n(\varepsilon, R(\varepsilon, M,\rho(\varepsilon))$ iterations, the MSE is lower than $\varepsilon^{-2}$.


\medskip
\noindent $\rhd$ Size: it follows from Equation~\eqref{eq:N_opt} in Step~1 combined with the  left inequality in Equation~\eqref{eq:ineqR} that
\begin{eqnarray*}
n(\varepsilon, R(\varepsilon), M,\rho(\varepsilon))& \sim &\left(1+ \frac{(M-1)\mathbf{\Psi}(M)}{\widetilde c\, \theta_1(f)}\right){\sigma_1^2(f)}\ \widetilde c\,  R(\varepsilon) \big(M^{\frac{R(\varepsilon)-1}{2}}\varepsilon^{\frac{1}{R(\varepsilon)}} \big)^{-1} \varepsilon^{-2} \\
&\precsim&\left(1+ \frac{(M-1)\mathbf{\Psi}(M)}{\widetilde c\, \theta_1(f)}\right){\sigma_1^2(f)}\widetilde c\,   R(\varepsilon)^2  \varepsilon^{-2} \\
&\sim&\frac{2}{\log M}\left({\widetilde c}\,+ \frac{(M-1)\mathbf{\Psi}(M)}{ \theta_1(f)}\right){\sigma_1^2(f)} \varepsilon^{-2}\log \Big(\frac{1}{\varepsilon}\Big)  \quad \mbox{ as }\varepsilon \to 0.
\end{eqnarray*}

\smallskip
\noindent $\rhd$ Complexity: Set $n(\varepsilon,M)= n(\varepsilon, R(\varepsilon), M,\rho(\varepsilon))$. The asymptotic resulting complexity satisfies
\[
K(n(\varepsilon,M),M) = n(\varepsilon,M)\Big(1+(M+1)\Big({R(\varepsilon)}-1\Big)\Big)\kappa_0
\sim  (M+1) R(\varepsilon)n(\varepsilon,M)\kappa_0 \quad \mbox{ as }\varepsilon \to 0
\]
 so that
\[
K(n(\varepsilon,M),M)\precsim \frac{2\kappa_0 (M+1)}{\log M}\left({\widetilde c}+ \frac{(M-1)\mathbf{\Psi}(M)}{ \theta_1(f)}\right){\sigma_1^2(f)} \varepsilon^{-2}\log \Big(\frac{1}{\varepsilon}\Big) \quad \mbox{ as }\varepsilon \to 0.
\]

\smallskip
\noindent $\rhd$ {\em  Initialization of the step:} it follows from~\eqref{eq:gamma_0}, the assumption made on $c_{_{R+1}}$  and the convergence of $R(\varepsilon)  \to +\infty$ that 
\begin{eqnarray*}
\gamma^*_1(\varepsilon)&\sim&  \widetilde c^{-1}M^{\frac{R(\varepsilon)}{2}}\,M^{-\frac 34}\quad \mbox{ as }\quad \varepsilon\to 0
\end{eqnarray*}
where we used that $\frac{R(R-1)}{2R+1}= \frac R2 -\frac 34 +\frac 34 \frac{1}{2R+1}$. Finally
using   the expression of $x(\varepsilon)$, we get 
\[
\gamma^*_1(\varepsilon)\sim  \widetilde c^{-1} \underbrace{M^{-\frac 34+\frac{\lceil x(\varepsilon,M)\rceil-x(\varepsilon,M)}{2}}}_{\in (M^{-\frac 14},M^{-\frac 34}]} \Big(\frac{\log M}{2}\Big)^{\frac 14}\exp{\Big(\sqrt{\frac{\log M\log\big(\frac{1}{\varepsilon}\big)}{2}}\Big)}  \Big(\log\big(\frac{1}{\varepsilon}\big)\Big)^{-\frac 14}\; \mbox{ as }\; \varepsilon\to 0.\quad \Box
\] 
%
%
%
%
\section{Numerical experiments}\label{sec:numerics}
\subsection{Practitioner's corner}\label{sec:practicorner}

In this section, we want to provide some helpful informations for some practical use of the optimized algorithm given in Theorem \ref{thm:optimiz}. Let  $\varepsilon>0$ denote the prescribed RMSE  and let $M$ be an integer greater than $2$. In what follows  we  aim at computing  $\nu(f)$ for a given function $f$ such that $f-\nu(f)$ is supposed to be a smooth enough coboundary.
 
  \paragraph{$\rhd$ The weights $\mathbf{W}^{(R)}_r)_{r=1,\ldots,R}$}. When the re-sizers are uniform they are computed by an instant closed form\eqref{eq:quniform}. Otherwise, they are given in full generality by  the $R$-tuple  of series~\eqref{eq:Wr} whose computation is also (almost)  instantaneous.  When $R=2,\,3$ one has again an instant closed form (see Examples below Lemma~\ref{lem:WWtilde}).

 \paragraph{$\rhd$ Computation of $R(\varepsilon,M)$.} We recall that $R(\varepsilon,M)=\lceil x(\varepsilon,M)\rceil$ where $x(\varepsilon,M)$ is the unique solution to $\frac{\log(M)}{2} x(x-1)+x\log x +\log({\varepsilon})=0$. For the computation of 
$x(\varepsilon,M)$, we use the classical (one-dimensional) zero search Newton algorithm. For standard values of $R$ and $M$, the reader may use Table \ref{valeursdexepsilon}. Finally, note that, ``though'' $\displaystyle R(\varepsilon)\sim \sqrt{\frac{2\log \big(1/\varepsilon\big)}{\log M}}$, one has   $\displaystyle \lim_{\varepsilon \to 0}R(\varepsilon)- \sqrt{\frac{2\log \big(1/ \varepsilon\big)}{\log M}}=-\infty$.

 \paragraph{$\rhd$ Values for $\mathbf{ \Psi}(M)$ and choice of $M$.} The quantity ${\bf \Psi}(M)$ appears in the size parameter $n(\varepsilon,M)$ (and in the complexity parameter $K(f,M)$  {given by~\eqref{eq:complexite}).
  Going back to the optimization procedure of the previous section, one remarks that for some fixed $R$ and $M$, one can replace $\mathbf{\Psi}(M)$ by $\frac{\Psi(R,M)}{R}$. This strategy leads to sharper bounds on the size parameter $n(\varepsilon,M)$ for a given RMSE  $\varepsilon$. We refer to  the first paragraph of Section \ref{subsec:numerics} for further investigations on this topic (see~\eqref{KepsRM} below and what precedes).  
Consequently, in Table~\ref{valeursdePsi}, we give some values  of ${\bf \Psi}(M)$,  but also of  $\frac{\Psi(R,M)}{R}$,  corresponding to some standard specifications  encountered in practical simulations. This also allows to check how $\frac{\Psi(R,M)}{R}$ varies for such low values of $R$ compared to ${\bf \Psi}(M)$. The conclusion is that ${\bf \Psi}(M)$ is an acceptable proxy of $\frac{\Psi(R,M)}{R}$.
 %
 %
\begin{table}[htbp]
\begin{center}
\begin{tabular}{|c|c|c|c|c|}
  \hline
  $\frac{\Psi(R,M)}{R}$ &  $R=2$ &  $R=3$ &   $R=4$ & $\mathbf{\Psi}(M)$ \\
  \hline   \hline
 $M=2$ &   2.133 &   2.591 &   2.674& 2.674 \\
\hline
  $M=3$ &  1.200 &    1.278 &   1.245 & 1.278 \\
\hline
  $M=4$ & 0.948 &   1.021  &  1.024 &  1.024\\
  \hline
\end{tabular}
\caption{\label{valeursdePsi} Values of $\Psi(R,M)$ and ${\bf \Psi}(M)$}
\end{center}
\end{table}

\paragraph{$\rhd$ Computation of  $n(\varepsilon,M)$.} The specification of  the size of the coarse level $n(\varepsilon,M)$ and, which is less important,  the {\em a priori} estimation of the global complexity, denoted $K(f,\varepsilon,M)$, both  require to estimate, at least theoretically, the parameters $\tilde{c}$, $\theta_1(f)$ and $\sigma_1^2(f)$. We will focus on their calibration in the next paragraph. To some extent, the estimation of $\theta_2(f)$ is less important and any way out of reach at a reasonable cost.

But even at this stage, it is inserting  to analyze their impact on $n(\varepsilon,M)$ in order to optimize the choice of the root $M$. To this end, we assume  for a moment that $C=\tilde{c}\,\theta_1(f)$ is known.   Going back to the sharper upper-bound of at our disposal, namely~\eqref{eq:complexite}, it suggests to minimize, for  fixed $C$,  the function
 \[
g_{_C}: M\longmapsto \frac{M+1}{\log M} \left (\frac{(M-1)\mathbf{\Psi}(M)}{C}+1\right).
\]
 Without going further, let us just note that $2{\bf \Psi}(3)\le {\bf \Psi}(2)$ so that $g_{_C}(3)\le g_{_C}(2)$ for any $C$ since $3/\log 2> 4/\log 3$ so that it seems that $M=3$ is always  a better choice than $M=2$.
But as emphasized in the next section~\ref{subsec:numerics} (first paragraph devoted to a ``toy'' Ornstein-Uhlenbeck setting), a sharper study of the complexity involving $\frac{\Psi(R,M)}{R}$ leads to temper the  answer. 

\paragraph{$\rhd$ Calibration of the parameters} 
{This calibration can be performed as a pre-processing phase based on a preliminary short Monte Carlo simulation, having in mind that only  rough estimates are needed.}
 
\smallskip
{
\noindent -- {\em Estimation of $\sigma_1^2(f)$ and $\theta_1(f)$}.  First, let us consider $\sigma_1^2(f)$. Through  an $L^2$-version  of~\eqref{tclbasic}, one deduces that 
for a family of independent random empirical measures  $(\nu_n^{(\ell)})_{\ell=1}^L$, namely
$$
\frac{1}{{\Gamma_n}}\sum_{\ell=1}^L\ES[(\nu_n^{(\ell)}(f)-\bar\nu_n^{(L)}(f))^2)]\xrightarrow{n\rightarrow+\infty}\sigma_1^2(f)\qquad\mbox{ as }\quad L,\, n\to +\infty
$$
where $\gamma_n=\gamzero  n^{-a}$ with $a> 1/3$ (say $a = \frac12$ in practice  to get rid  of the bias effect   even for small values of $n$) and $\bar\nu_n^{(L)}(f))=\frac 1L\sum_{\ell=1}^L\nu_n^{(\ell)}(f)$.
}

{As $\theta_1(f)=\frac{\sigma_1^2(f)}{\sigma_{2,2}^2(f)}$, it remains to provide an estimator of $\sigma_{2,2}^2(f)$. To do so we take advantage of the fact that $\sigma_{2,2}^2(f)$ is  the (normalized) asymptotic variance of $(\mu_n^{M,\gamma})_{n\ge1}$.  We thus may   use the same strategy as above. More precisely, under Assumption $\mathbf{(C_s)}$, we deduce from Propositions \ref{prop:dvpcorrectiveterms} and \ref{prop:contribmartmun} that 
 $$
 \frac{1}{\Gamma_n^{(2)}}\sum_{\ell=1}^L\ES[(\mu_n^{(\ell)}(f)-\bar\mu_n^{(L)}(f))^2)]\xrightarrow{n\rightarrow+\infty}\sigma_{2,2}^2(f)
 $$
if $\gamma_n=\gamzero  n^{-a}$ with $a>1/5$ (say $a = \frac14$ in practice  to get rid of  the bias effect   even for small values of $n$)  with $\bar\mu_n^{(L)}(f)=\frac 1L\sum_{\ell=1}^L\mu_n^{(\ell)}(f)$.
}

\smallskip
{\noindent --  {\em About  $\widetilde c$ and  $\theta_2(f)$}.  The coefficient $\widetilde c$ will probably always remain mysterious. On the other hand in practice what we really need is rather $|c_{_{R(\varepsilon)}}|^{\frac{1}{R(\varepsilon)}}$. However, under the assumption $\displaystyle \lim_{R\to +\infty} |c_{_R}|^{\frac 1R} = \widetilde c \!\in (0, +\infty)$ made on $c_{_R}$ in Theorem~\ref{thm:optimiz}, one can make the guess  from its very definition that its value is not too far from $1$ or is at least of order a few units. In particular, if  the coefficients $c_{_R}$ have  a polynomial growth or even $c_{_R} = O(\exp{|R|^{\vartheta_0}})$, $\vartheta_0\!\in [0,1)$, $\widetilde c =1$. If they have an exponential growth it remains finite (but possibly large). The point of interest is that, anyway, this value is much more stable than the first coefficient itself $c_1$ which would come out in a standard MLMC Langevin simulation framework (not investigated here).  
}

{The parameter $\theta_2(f)$ seems to be unaccessible as well, but for another reason: it is  the variance induced by a second order martingale. However as noticed in Section~\ref{subsec:numerics} (first paragraph), $\theta_2(f)$ is the {\em ratio of two variance terms} so that it seems not so much dependent on the magnitude of the diffusion coefficient (in fact it can be noted that the same property holds for $\theta_1(f)$).
 }
{\begin{Remarque} The numerical investigations of the next section show that the algorithm is very robust to the choice of the parameters. For simple practice, we thus recommend to get a rough estimation of $\sigma_1^2(f)$ and possibly of  $\theta_1(f)$    and to set  $\theta_2(f)=\widetilde{c}=1$.
\end{Remarque}
}

\subsection{Numerical tests}\label{subsec:numerics}
We propose in this section to provide some numerical tests of our algorithm.

\paragraph{Orstein-Uhlenbeck process: oracle and blind simulation.} We begin with the Ornstein-Uhlenbeck process in dimension $1$ solution to 
$$
dX_t=-\frac{1}{2} X_t dt+\sigma dW_t
$$
with $f(x)=x^2$.  We recall that this case is a toy example since whole the computations can be made explicit.
In particular, $\nu\sim{\cal N}(0,\sigma^2)$  so that $\nu(f)=\sigma^2$. Furthermore, $g(x)=x^2$ is
the unique solution (up to a constant) to the Poisson equation $f-\nu(f)=-{\cal L}g$   and it follows that 
$$\sigma_1^2(f)=\sigma_{2,2}^2(f)=4\sigma^4,\quad \textnormal{and}\quad\sigma_{2,1}^2(f)=5\sigma^4. $$
The reader can remark that in this case, the ratios $\theta_1(f)$ and $\theta_2(f)$ do not depend on $\sigma$. Even though this property can not be really 
generalized, it however emphasizes a stability of these parameters with respect to the variance of the model. 
The bias terms can also be computed: using that $\varphi_2(f)=\frac{1}{4} f$ and that $\varphi_\ell=0$ for $\ell\ge3$, we get   $c_{R+1}=\sigma^2/4^R$ (so that $\tilde{c}=1/4$).\smallskip

\noindent We want in this part to get a sharp estimate of the complexity for several choices of couples $(R,M)$. Following the optimization procedure, we go back to the definition of $n(\varepsilon, R, M,\rho)$ given in~\eqref{eq:N_opt}:
\begin{equation*}
n= n(\varepsilon, R, M,\rho) = \left\lceil\rho^{-(1+\frac{1}{2R})}\mu(R)  R\,\sigma_1^2(f)M^{-\frac{R-1}{2}}\varepsilon^{-2-\frac 1R}\right\rceil
\end{equation*}
and for each value of $R$ and $M$, we solve by a Newton method the following equation for $\rho\in[0,1]$:
\begin{equation}\label{eq:condrho}
\varepsilon^{\frac 1R} M^{\frac{R-1}{2}}R = \frac{1-\rho}{\rho}\rho^{-\frac{1}{2R}}\frac{\mu(R) \theta_1(f)}{R^{-1}\theta_2(f)+\left(1-\frac 1M\right)R^{-1}{\Psi}(R,M)}
\end{equation}
where the values of $\Psi(R,M)$ for $R,M=2,3,4$ are given in Table~\ref{valeursdePsi}.

We denote by $\rho^\star$ the solution of this equation. Then, the complexity $K(\varepsilon,M)$ (where we assume that $\kappa_0=1$) is given by
\begin{equation}\label{KepsRM}
K(\varepsilon,R,M)= \left(1+M\Big(1-\frac{1}{R}\Big)\right)n(\varepsilon, R, M,\rho^\star).
\end{equation}
This yields the following results for $\varepsilon=10^{-2}$:
\begin{table}[htbp]

\begin{minipage}[t]{.4\linewidth}
    \begin{tabular}{|c|c|c|c|}
    \hline
	$\sigma=1$&$R=2$&R=3&R=4\\
\hline
    \hline
    $M=2$ &  ${\bf 1.09*10^6}$&   $1.58*10^6$&  $2.55*10^6$\\
		 \hline
   $M=3$& $1.11*10^6$&    $1.43*10^6$&    $2.05*10^6$\\
	 \hline
   $M=4$ & $1.21*10^6$&    $1.57*10^6$&    $2.27*10^6$  \\
	 \hline
    \end{tabular}
\end{minipage}
\hskip 1.5cm 
\begin{minipage}[t]{.4\linewidth}
    \begin{tabular}{|c|c|c|c|}
    \hline
$\sigma=4$	&$R=2$&R=3&R=4\\
	\hline
   \hline
    $M=2$ &  ${ 7.02*10^8}$&   $5.23*10^8$&  $7.34*10^8$\\
		 \hline
   $M=3$& $7.17*10^8$&    ${\bf 4.76*10^8}$&    $6.10*10^8$\\
	 \hline
   $M=4$ & $7.56*10^8$&    $4.99*10^8$&    $6.55*10^8$  \\
	 \hline


    \end{tabular}
    \end{minipage}
     \caption{\label{complexexacte} $K(\varepsilon,R,M)$ for $\varepsilon=10^{-2}$}
\end{table}
On this example, we retrieve the property which says that $M=2$ is a good choice when $\tilde{c}\theta_1$ is small whereas $M=3$ can be greater when this quantity increases. However, as expected, the main parameter is the level $R$ of the method which increases when $\varepsilon\rightarrow0$. \smallskip

\noindent Taking only the first term of the expansion of the MSE for the crude procedure, the optimized complexity (with $\kappa_0=1$) for a MSE lower than $\varepsilon=10^{-2}$ is equal to $K(\varepsilon)=6.93*10^6$ and $K(\varepsilon)=1.77*10^9$  if $\sigma=1$ or $\sigma=4$ respectively.

 In Figure~\ref{fig1OU},
we compare numerically the evolution of {\bf   ML2Rgodic } with the crude algorithm for $\sigma=1$ and $\sigma=4$. Note that to obtain a rigorous comparison, the graphs are drawn in terms of the complexity, that  once again with a slight abuse of language, is the number of iterations of the Euler scheme involved by procedure.   
\begin{figure}[htbp]
   \begin{minipage}[c]{.46\linewidth}
      \includegraphics[width=8cm]{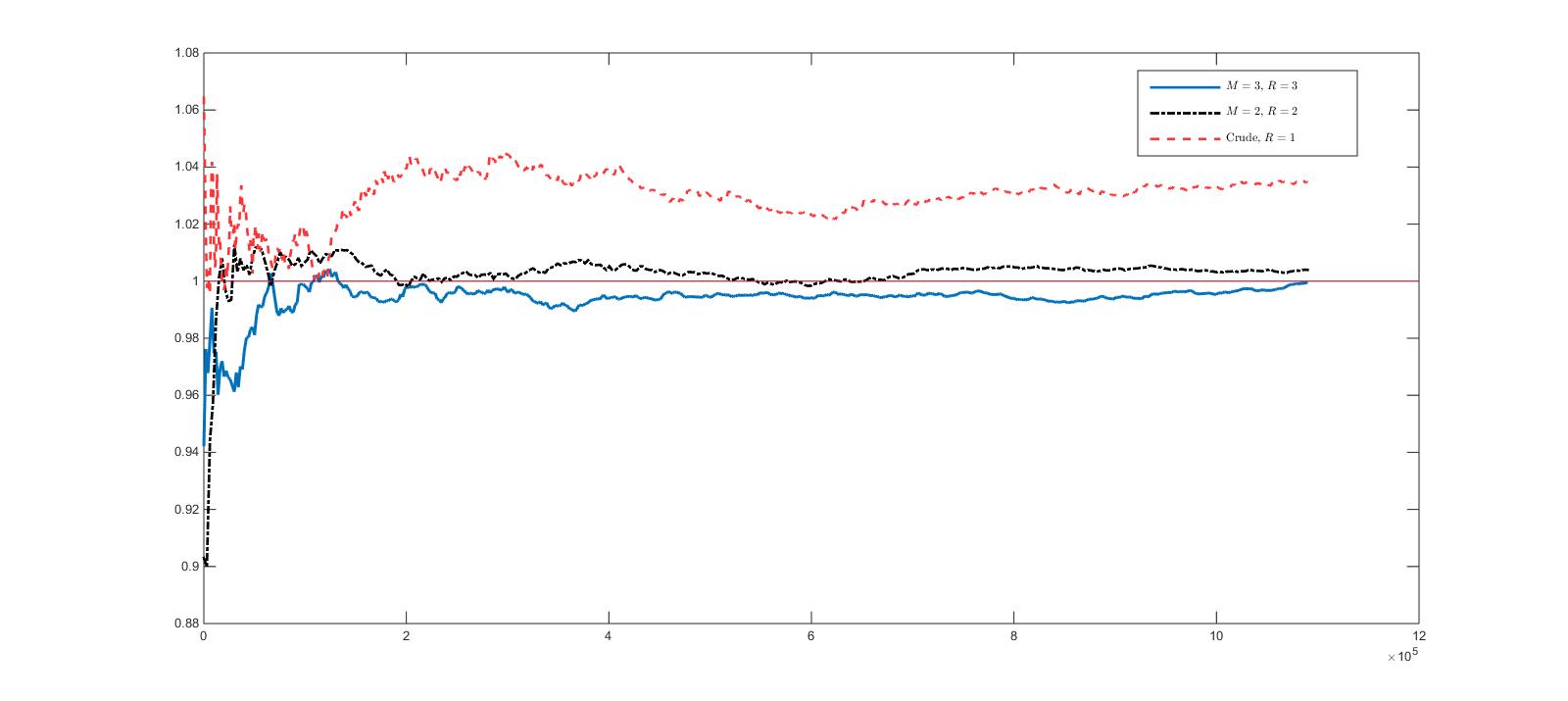}
   \end{minipage} \hfill
   \begin{minipage}[c]{.46\linewidth}
      \includegraphics[width=7.5cm]{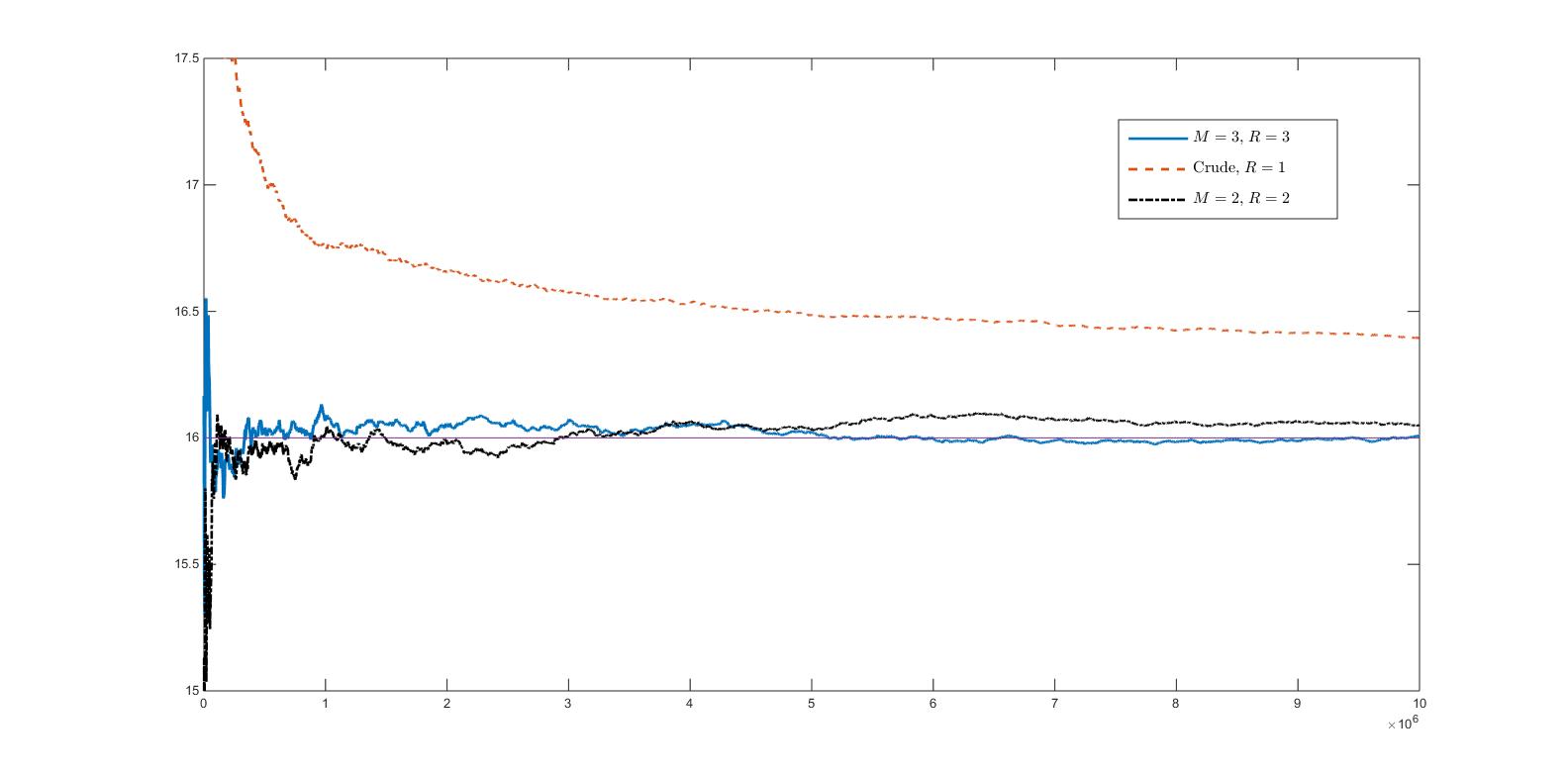}
   \end{minipage}
	\caption{\label{fig1OU} Comparison of the evolution in terms of the complexity of the ML2Rgodic with the crude algorithm}
\end{figure}
One remarks that the effect of the Multilevel-RR procedure is increased in the case $\sigma=4$ where the bias is larger. One also remarks in this case that, even though the algorithm is robust to the choice $M$ and $R$,  the best choice seems to be the one given in Table~\ref{complexexacte}.\smallskip

\noindent Of course, in practice, one can not make use of the exact parameters. As explained in Section~\ref{sec:practicorner}, it is possible to get a rough estimation of $\sigma_1^2(f)$ and $\theta_1(f)$ using the CLTS induced by the procedure. The coefficient $c_{R+1}$ can also be estimated but for this coefficient, this requires to use a Multistep method or the procedure ML2Rgodic itself with one more stratum than in the algorithm that we will implement after. Finally, the coefficient $\theta_2(f)$ seems to be impossible to estimate. This implies that the natural question that the practitioner may ask is: is it possible to get rid of the estimation of the above parameters ?\smallskip

\noindent To answer to this question, we propose in the case $\sigma=4$ to look at the dynamics of the procedure when we choose to fix
\begin{itemize}
\item $c_{R+1}=\theta_2(f)=1$ and to estimate  $\sigma_1^2(f)$ and $\theta_1(f)$,
\item $c_{R+1}=\theta_2(f)=\sigma_1^2(f)=\theta_1(f)$.
\end{itemize}
With these two choices of parameters and with $\varepsilon=10^{-2}$, we follow the procedure described in the previous section to estimate $\gamzero ^\star$, $R$, $\rho$ and $M$. Note that we again obtain $R=3$ and $M=3$ as an optimal choice. 
In Figure~\ref{figarticle3}, we thus compare the evolution of the previous method (with semi-estimated or not estimated) parameters and we can remark on this example that the algorithm seems to be very robust to the choice of the parameters.

\vskip -0.25cm
\begin{figure}[htbp]
\begin{center}
      \includegraphics[width=10cm]{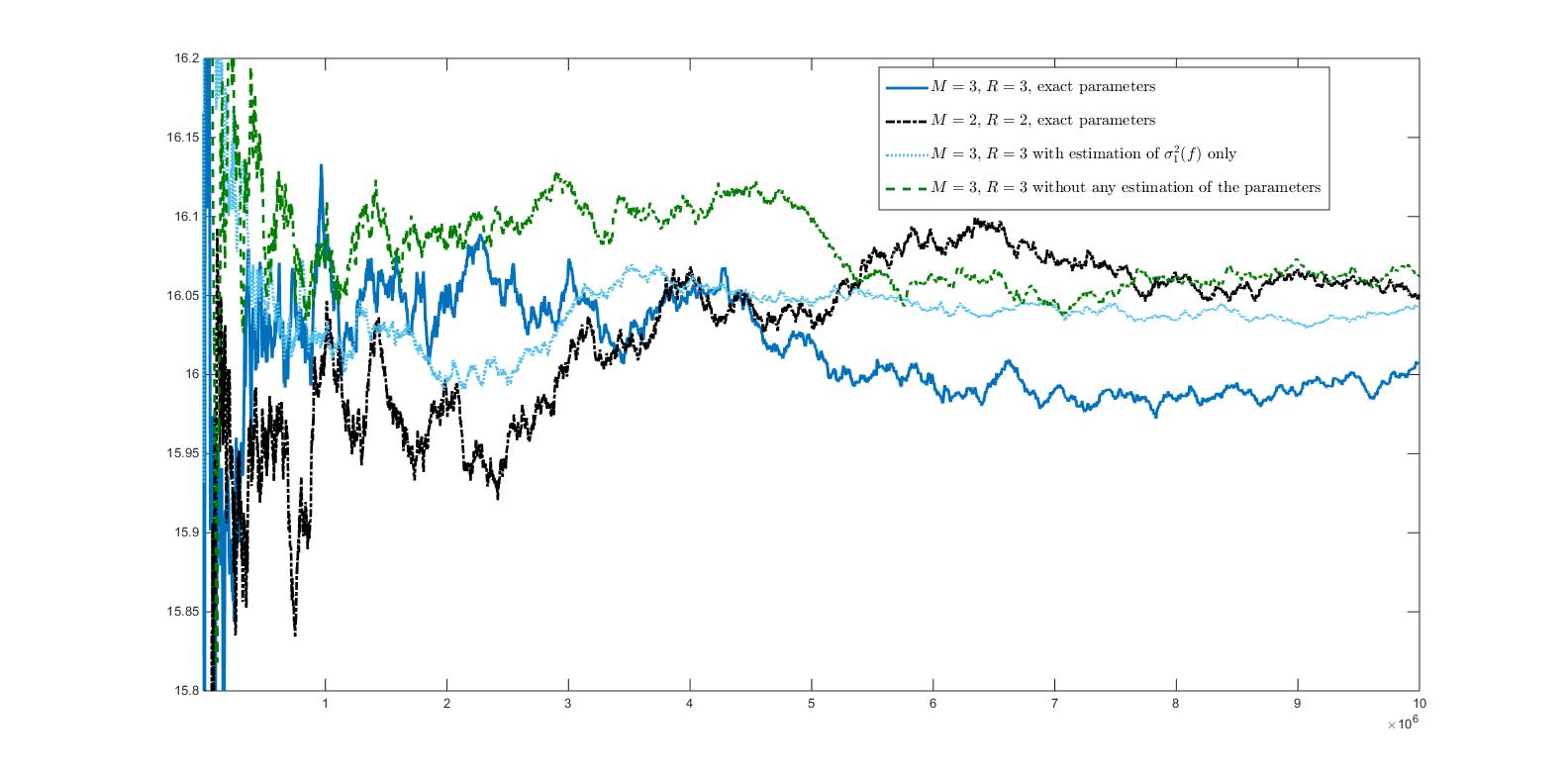}

\caption{\label{figarticle3} Orstein-Uhlenbeck process: Evolution of the algorithm in terms of the estimation of the parameters, Exact Value : $\nu(f)=16$.}
	\end{center}
\end{figure}

\vskip -0.25cm
%
%
\paragraph{Double-well potential.} We consider a second example in dimension $1$ 
$$
dX_t=-V_1'(X_t)dt+\sigma dW_t
$$
where $V_1(x)=x^2-\log(1+x^2)$ which is a non-convex potential (with two local minima in $-1$ and $1$) so that Assumption $\mathbf{(C_s)}$ is not fulfilled. However, Assumption $\mathbf{(C_w)}$ is true (see~\cite{LPPIHP}, Theorem 2.1). Let us also recall that the invariant distribution $\nu$ satisfies
$$
\nu(dx)=\frac{1}{Z_{V_1}}\exp\left(-\frac{V_1(x)}{2\sigma^2}\right)\lambda(dx)
$$
where $Z_{V_1}=\displaystyle \int_\ER \exp\left(-\frac{V_1(x)}{2\sigma^2}\right)\lambda(dx)$.

 We test the algorithm in this setting with $f(x)=x^2$ and $\sigma=2$. Figure~\ref{figarticle4} shows that ML2Rgodic is still efficient in this setting. The results are obtained using a rough estimation of $\sigma_1^2(f)$ and $\theta_1(f)$ and the other parameters are fixed to $1$. Once again, the evolution is compared with the crude algorithm with an optimized choice of $\gamzero ^*$ and the evolution is drawn as a function of the complexity. 

   \begin{figure}[htbp]
\begin{center}
      \includegraphics[width=10cm,height=6cm]{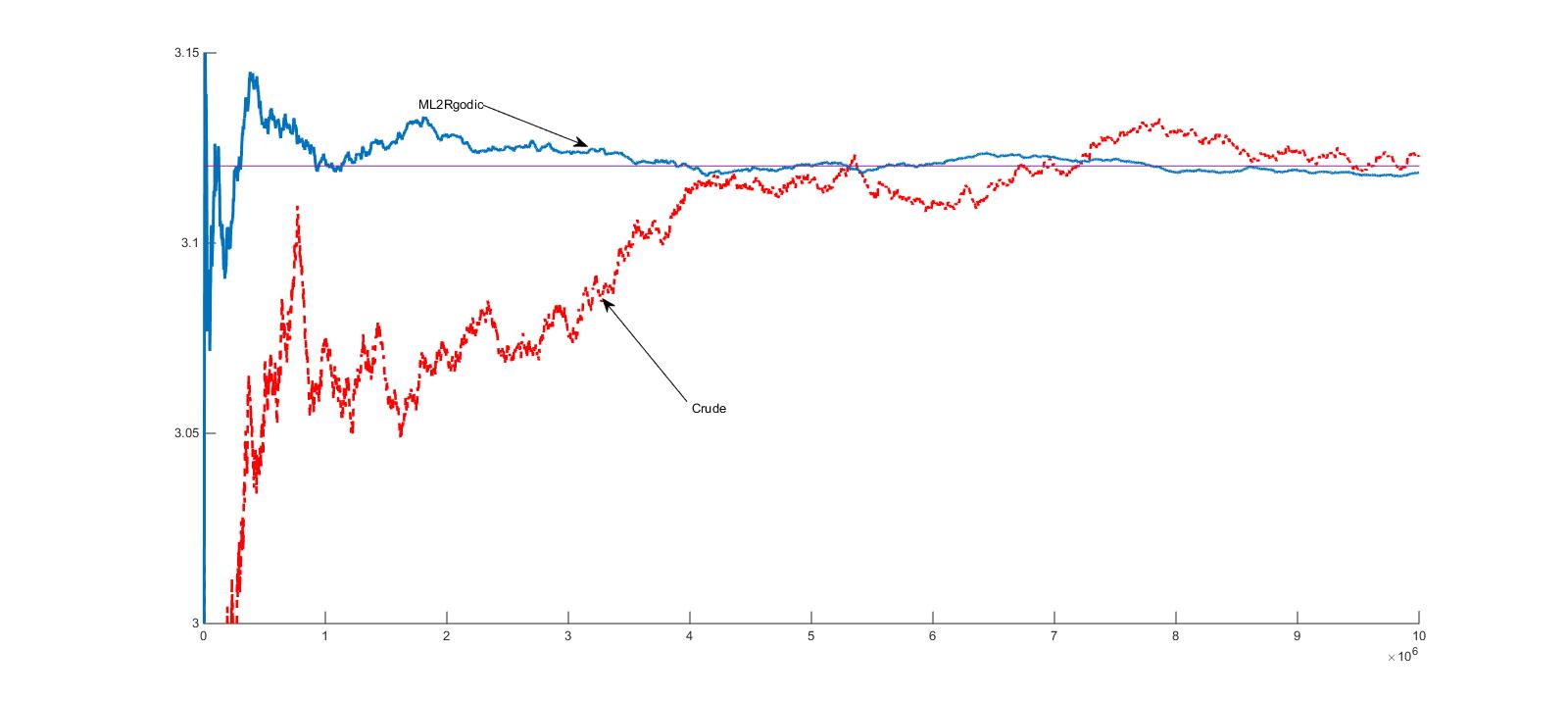}
   \vskip -0.5cm
	\caption{\label{figarticle4} Double-well potential: Approximation of $\nu(f)$ with $f(x)=x^2$, $\sigma=2$, Exact value : 3.1207.}
	\end{center}
\end{figure}

\paragraph{Statistical example (Sparse Regression Learning).} In~\cite{dalalyan-tsybakov}, the authors consider the problem of Sparse Regression Learning by Aggregation. For the sake of simplicity, we only recall here the case of linear regression: let $p$ denote the number of variables and $N$ the number of observations and suppose we are given $n$ couples of observations $({\bf X}_1,{\bf Y}_1)$, \ldots, $({\bf X}_N,{\bf Y}_N)$ where the vector ${\bf X}_i=(X_i^1,\ldots,X_i^p)$  is  the \textit{predictor} and the scalar ${\bf Y}_i$ is   the \textit{response}. Suppose that there exists $\theta_0\!\in\ER^p$ such that 
$$
\forall i\in\{1,\ldots,n\},\quad {\bf Y}_i={\bf X}_i\theta_0+\xi_i
$$
where $(\xi_i)_{i=1}^N$ denotes a sequence of $i.i.d.$ random variables with distribution ${\cal N}(0,\sigma^2)$ for a given (generally unknown) $\sigma>0$. Then, the classical question is: how to estimate $\theta_0$ ?
When $p\gg N$, the classical methods (such as the least-square method) do not work and it is necessary to introduce some alternative procedures.   The estimator of $\theta_0$ proposed by Dalalyan and Tsybakov -- called {\em EWA} (for Exponentially Weighted Aggregate)  -- is designed as follows:
$$ 
\hat{\theta}=\int _{\ER^p} \theta\,\pi_{V_2}(d\theta)
$$
where $\pi_{V_2}$ is the Gibbs probability measure defined by
$$
\pi_{V_2}(d\theta)=\frac{1}{Z_{V_2}}\exp\big(-V_2(\theta)\big)\lambda(d\theta)
$$
and    $Z_{V_2}$  is a normalizing coefficient and $V_2:\ER^p\mapsto\ER$ is  the potential defined for some given positive numbers $\alpha$, $\beta$ and $\tau$ by
$$
\forall \theta\in\ER^p,\quad V_2(\theta)=\frac{|\mathbf{Y}-\mathbf{X}\theta|^2}{\beta}+\sum_{j=1}^p \left(\log(\tau^2+\theta_j^2)+\omega(\alpha \theta_j) \right)
$$
with $\omega(\theta)=\theta^2\vee (2|\theta|-1)$.

As mentioned (and already numerically tested) in~\cite{dalalyan-tsybakov}, $\hat{\theta}$ is but  the expectation related to the invariant distribution of the following SDE 
$$
d\theta_t=-\nabla V_2(\theta_t) dt+\sqrt{2} dW_t.
$$
 It can subsequently be estimated through a Langevin Monte-Carlo procedure. The difficulty in this context is the fact that $p$ is potentially large so that the numerical computation needs some adaptations. More precisely,
in order to avoid an explosion of the Euler scheme, we need to impose the step to be not to large for small values of $n$. We thus assume in what follows that :
$$
\gamma_n=\min\Big(\frac{\gamma_1^\star}{ n^{a}},\frac{1}{p}\Big).
$$
Below, we test our   {\bf ML2Rgodic} estimator on a  \textit{compressed sensing} example given \cite{dalalyan-tsybakov} (see Example 1) with the parameters given in this paper. We fix~(\footnote{From a theoretical point of view, $\alpha$ should be a positive number such that $\alpha\le 1/(4p\tau)$.})
$$
\alpha=0,\quad \beta=4\sigma^2,\quad \tau=\frac{4\sigma}{{\rm Tr}({\bf X}^t{\bf X})^{\frac{1}{2}}}
$$
and the computations are achieved with  $p=500$, $N=100$ and $S=15$ where $S$ denotes the sparsity parameter, $i.e.$ the number $S$ of non-zero components of $\theta_0$  (of course we do not know which ones). 
Then, the matrices ${\bf X}$ and ${\bf Y}$ are generated  from simulated data as follows: in this compressed sensing setting, the matrix ${\bf X}$ has independent Rademacher entries with parameter $1/2$. The unknown $\theta_0$
is defined simply by $\theta_0(j)={\bf 1}_{j\le S}$, for every $j\in\{1,\ldots,p\}$. Finally, following again the parameters given in \cite{dalalyan-tsybakov}, we set $\sigma^2=S/9$.

\smallskip
 \noindent Denoting by $\hat{\theta}_n$ the approximation of $\hat{\theta}$ obtained after $n$ iterations of the scheme,  we depict in  Figure~\ref{figarticleEWA} the evolution of  $n\mapsto \|\hat{\theta}_n-\theta_0\|_2$. Note that $\|\hat{\theta}_n-\theta_0\|_2$ converges to $\|\hat{\theta}-\theta_0\|_2$ (which is not equal to $0$). We compare it with the crude procedure (taken with $a=1/3$ whereas for the ML2Rgodic procedure, $a=1/(2R+1)$ as usual). We can remark that the correction on the bias involved by the weighted Multilevel Langevin procedure strongly
improves the estimation of $\theta_0$. This remark is emphasized if we compare with the results of~\cite{dalalyan-tsybakov} based on an Euler scheme with constant step where the corresponding quantity is equal to $8.917$ (in this case, the constant step is about $(Np)^{-1}$).



\begin{figure}[htbp]
\begin{center}
      \includegraphics[width=10cm,height=5cm]{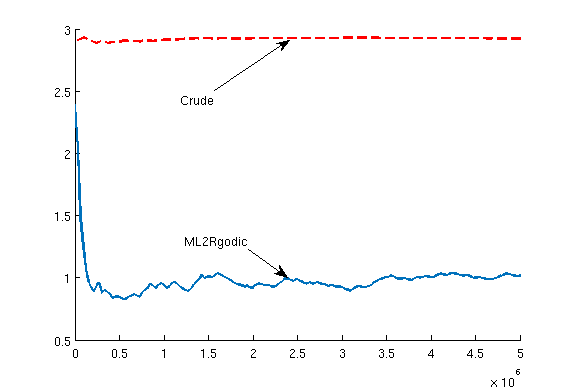}
   \vskip -0.5cm
	\caption{\label{figarticleEWA} Sparse Regression Learning: $n\mapsto \|\hat{\theta}_n-\theta_0\|_2$ for the Crude and  {\bf ML2Rgodic} ($R=3$) procedures.}
	\end{center}
\end{figure}

\bibliographystyle{alpha}
\small
\bibliography{bibli}

\def\cprime{$'$} \def\cprime{$'$} \def\cprime{$'$}
\begin{thebibliography}{LPP15}

\bibitem[Bha82]{bhatta82}
R.~N. Bhattacharya.
\newblock On the functional central limit theorem and the law of the iterated
  logarithm for {M}arkov processes.
\newblock {\em Z. Wahrsch. Verw. Gebiete}, 60(2):185--201, 1982.

\bibitem[Bil78]{billingsley-ergodic}
Patrick Billingsley.
\newblock {\em Ergodic theory and information}.
\newblock Robert E. Krieger Publishing Co., Huntington, N.Y., 1978.
\newblock Reprint of the 1965 original.

\bibitem[DT12]{dalalyan-tsybakov}
A.~S. Dalalyan and A.~B. Tsybakov.
\newblock Sparse regression learning by aggregation and {L}angevin
  {M}onte-{C}arlo.
\newblock {\em J. Comput. System Sci.}, 78(5):1423--1443, 2012.

\bibitem[Fri16]{Frikha}
N.~Frikha.
\newblock Multi-level stochastic approximation algorithms.
\newblock {\em Ann. Appl. Probab.}, 26(2):933--985, 2016.

\bibitem[Gil08]{GIL}
M.~B. Giles.
\newblock Multilevel {M}onte {C}arlo path simulation.
\newblock {\em Oper. Res.}, 56(3):607--617, 2008.

\bibitem[GT83]{gilbarg}
David Gilbarg and Neil~S. Trudinger.
\newblock {\em Elliptic partial differential equations of second order}, volume
  224 of {\em Grundlehren der Mathematischen Wissenschaften [Fundamental
  Principles of Mathematical Sciences]}.
\newblock Springer-Verlag, Berlin, second edition, 1983.

\bibitem[GT15]{garcia-trillos}
C.~A. Garc{\'{\i}}a~Trillos.
\newblock A decreasing step method for strongly oscillating stochastic models.
\newblock {\em Ann. Appl. Probab.}, 25(2):986--1029, 2015.

\bibitem[HH80]{hall-heyde}
P.~Hall and C.~C. Heyde.
\newblock {\em Martingale limit theory and its application}.
\newblock Academic Press, Inc. [Harcourt Brace Jovanovich, Publishers], New
  York-London, 1980.
\newblock Probability and Mathematical Statistics.

\bibitem[Kre85]{krengel-ergodic}
U.~Krengel.
\newblock {\em Ergodic theorems}, volume~6 of {\em de Gruyter Studies in
  Mathematics}.
\newblock Walter de Gruyter \& Co., Berlin, 1985.
\newblock With a supplement by Antoine Brunel.

\bibitem[Lem05]{lemairethese}
V.~Lemaire.
\newblock {\em Estimation r{\'e}cursive de la mesure invariante d'un processus
  de diffusion}.
\newblock Th\`{e}se de doctorat, Universit{\'e} de Marne-la-Vall{\'e}e
  (France), 2005.

\bibitem[Lem07]{lemaire2}
V.~Lemaire.
\newblock Behavior of the {E}uler scheme with decreasing step in a degenerate
  situation.
\newblock {\em ESAIM Probab. Stat.}, 11:236--247, 2007.

\bibitem[LP02]{LP1}
D.~Lamberton and G.~Pag{\`e}s.
\newblock Recursive computation of the invariant distribution of a diffusion.
\newblock {\em Bernoulli}, 8(3):367--405, 2002.

\bibitem[LP03]{LP2}
D.~Lamberton and G.~Pag{\`e}s.
\newblock Recursive computation of the invariant distribution of a diffusion:
  the case of a weakly mean reverting drift.
\newblock {\em Stoch. Dyn.}, 3(4):435--451, 2003.

\bibitem[LP13]{LEPA}
V.~Lemaire and G.~Pag\`es.
\newblock Multilevel {R}ichardson-{R}omberg extrapolation.
\newblock {\em Bernoulli (to appear in)}, 2013.

\bibitem[LPP15]{LPPIHP}
V.~Lemaire, G.~Pag{\`e}s, and F.~Panloup.
\newblock Invariant measure of duplicated diffusions and application to
  {R}ichardson-{R}omberg extrapolation.
\newblock {\em Ann. Inst. Henri Poincar\'e Probab. Stat.}, 51(4):1562--1596,
  2015.

\bibitem[Pag07]{PAG}
G.~Pag{\`e}s.
\newblock Multi-step {R}ichardson-{R}omberg extrapolation: remarks on variance
  control and complexity.
\newblock {\em Monte Carlo Methods Appl.}, 13(1):37--70, 2007.

\bibitem[Pan08]{panloup1}
F.~Panloup.
\newblock {Recursive computation of the invariant measure of a stochastic
  differential equation driven by a L{\'e}vy process}.
\newblock {\em Annals of Applied Probability}, 18(2):379--426, 2008.

\bibitem[PP09]{PP1}
G.~Pag{\`e}s and F.~Panloup.
\newblock Approximation of the distribution of a stationary {M}arkov process
  with application to option pricing.
\newblock {\em Bernoulli}, 15(1):146--177, 2009.

\bibitem[PP14]{PP3}
G.~Pag{\`e}s and F.~Panloup.
\newblock A mixed-step algorithm for the approximation of the stationary regime
  of a diffusion.
\newblock {\em Stochastic Process. Appl.}, 124(1):522--565, 2014.

\bibitem[PS94]{piccioni-scarlatti}
M.~Piccioni and S.~Scarlatti.
\newblock An iterative {M}onte {C}arlo scheme for generating {L}ie group-valued
  random variables.
\newblock {\em Adv. in Appl. Probab.}, 26(3):616--628, 1994.

\bibitem[PV01]{Parver1}
E.~Pardoux and A.~Yu. Veretennikov.
\newblock On the {P}oisson equation and diffusion approximation. {I}.
\newblock {\em Ann. Probab.}, 29(3):1061--1085, 2001.

\bibitem[Tal90]{talay}
D.~Talay.
\newblock Second order discretization schemes of stochastic differential
  systems for the computation of the invariant law.
\newblock {\em Stoch. Stoch. Rep.}, 29(1):13--36, 1990.

\end{thebibliography}

\small
\begin{appendix} 
\section{Proof of Lemma~\ref{lem:WWtilde}}\label{App:A}

Prior to the proof of Lemma~\ref{lem:WWtilde}, we need to prove this first technical lemma which will be used  to estimate in a precise way the coefficients $ \widetilde{{\mathbf W}}_{_{R+1}}$ and $\widetilde{{\mathbf W}}_{_{R+2}}$ involved in the asymptotic mean square error of the {\bf ML2Rgodic} estimator in Theorems~\ref{theo:CLT} and~\ref{L2theo}.

\begin{lemme} \label{lem:VanderStuf}Let $R\ge 2$ be an integer.and let   $(x_r)_{r=1,\ldots,R}$ be pairwise distinct real numbers. Then
the unique solution  $(y_r)_{r=1,\dots,R}$ to the   solution to the $R\times R$-Vandermonde system 
$$
\sum_{r=1}^{R}x_r^{\ell-1}y_r= c^{\ell-1},\;\ell=1,\ldots,R,
$$
\begin{eqnarray}
\label{eq:Vanderstuff00}
\mbox{is given by 
}\hskip 5cm  y_r &=&  \frac{\prod_{s=1, s\neq r}^{R}(x_r-c)}{\prod_{s=1, s\neq r}^{R}(x_r-x_s)}.\hskip 2cm
\\
\label{eq:Vanderstuff01}
 \mbox{Moreover }\hskip 4,25cm \sum_{r=1}^{R} y_r x_r^{R}&=& c^{R}-\prod_{r=1}^{R} (c-x_r) \hskip 2cm \\
\label{eq:Vanderstuff02}\mbox{ and }\hskip 5cm   \sum_{r=1}^{R} y_r x_r^{R}&=& c^R+\left(\sum_{r=1}^{R}x_r+c\right) \prod_{r=1}^{R} (c-x_r).\hskip 2cm
\end{eqnarray}
\end{lemme}

\begin{proof} The  above Vandermonde system ${\rm Vand}( x_r, r=1:R) \mathbf{w} = [0^{\ell-1}]_{\ell=1,R}$ can be explicitly solved by the Cramer formulas since its right hand side is of the form $[c^{\ell-1}]_{1\le \ell\le R}$ for some $c\!\in \R$. Namely
\[
y_r = \frac{{\rm det(Vand}(x_1, \ldots, x_{r-1}, c,x_{r+1}, \ldots, x_{_{R}}\mbox{)})}{{\rm det(Vand}( x_s, s=1:R\mbox{)})}, \; r=1,\ldots,R
\]
(the  column vector $[c^{\ell-1}]_{1\le \ell\le R}$ replaces the $r^{th}$ column of the original Vandermonde matrix).   Then, elementary computations show that it yields  the  announced  solutions.

To compute the next two sums, we start from the following canonical  decomposition of the rational fraction 
\[
\frac{1}{\prod_{r=1}^{R}(X-\frac{1}{x_r-c})} = \sum_{r=1}^{R} \frac{1}{(X-\frac{1}{x_r-c})\prod_{s\neq r}(\frac{1}{x_r-c}-\frac{1}{x_s-c})}.
\]
Setting $X=0$ yields after elementary computations
\[
\sum_{r=1}^{R}y_r(x_r-c)^{R} = (-1)^{R} \prod_{r=1}^{R} (x_r-c).
\]
Now, using that $(y_r)_{r=1,\ldots,R}$ solves the above Vandermonde system, we get 
\begin{eqnarray*}
\sum_{r=1}^{R}y_r(x_r-c)^{R}  &= &\sum_{r=1}^{R} y_r\sum_{k=0}^{R} \begin{pmatrix}R\cr k\end{pmatrix} (-1)^{R-k}x_r^kc^{R-k}\\
&=& \sum_{k=0}^{R}\begin{pmatrix}R\cr k\end{pmatrix} (-1)^{R-k} c^{R-k}
\underbrace{\sum_{r=1}^{R}y_rx_r^k }_{=\,c^{k}\mbox{\footnotesize if } k<R} =  \sum_{r=1}^{R}y_r x_r^{R} +c^{R}\big(\big(1-1\big)^{R}-1\big)
\end{eqnarray*}
so that
\[
\sum_{r=1}^{R} y_r x_r^{R} =c^{R}-  (-1)^{R}\prod_{r=1}^{R} (x_r-c)=c^{R}-\prod_{r=1}^{R} (c-x_r).
\]
The second identity follows likewise by differentiating the above rational fraction with respect to $X$ and then setting $X=0$ again.
\end{proof}

\noindent {\bf Proof of Lemma~\ref{lem:WWtilde}.} $(a)$ 
We introduce  the auxiliary variables and parameters 
\begin{equation}\label{eq:VanderR-1}
\overline {\mathbf W} _r = \left( \frac{q_1}{q_{r+1}}\right)^a \frac{W_{r+1}}{ M^{r-1 }},\; \quad x_r =  M^{-(r-1 )}\left(\frac{q_1}{q_{r+1}}\right)^a,\;r=1,\ldots,R-1.
\end{equation}
Then $(\mathbf W _r )_{1\le r\le R-1}$  is solution to the system~\eqref{eq:theSystem}  if and only if $(\overline {\mathbf W} _r )_{1\le r\le R-1}$  is solution to 
\[
\sum_{r=1}^{R-1} \overline {\mathbf W} _r x_r^{\ell-1} =  
\frac{1}{1- M^{-\ell}},\; \ell=1,\ldots,R-1.
\]
Expanding $\displaystyle \frac{1}{1- M^{-\ell}}=\sum_{k\ge 0}\frac{1}{ M^k}\frac{1}{ M^{k(\ell-1)} }$ yields by  linearity of  the above system that it suffices to solve   the sequence of $(R-1)\times(R-1)$-Vandermonde systems. 
\[
({\cal V}_k)\equiv \;\sum_{r=1}^{R-1} \overline {\mathbf W} _{k,r}  x_r^{\ell-1}=  M^{-k(\ell-1)}, \; \ell=1,\ldots,R-1, \;\; k\ge 0.
\]
As the $x_r$ are pairwise distinct,  $({\cal V}_k)$ has a unique  solutions  given by 
\[
\overline {\mathbf W}_{k,r} = \prod_{s=1,s\neq r}^{R-1}\frac{x_s- M^{-k}}{x_s-x_r},\; r=1,\ldots,R-1.
\]
with the usual convention $\prod_{\emptyset} =1$ Consequently, for every  $r=2,\ldots,R$, 
\[
\overline {\mathbf W} _r = \sum_{k\ge 0} \frac{1}{ M^k} \overline {\mathbf W} _{k,r} =\sum_{k\ge 0} \frac{1}{ M^k} \prod_{s=1,s\neq r}^{R-1}\frac{x_s- M^{-k}}{x_s-x_r},\; r=1,\ldots,R-1.
\]
Coming back to the weights of interest  finally yields the expected formula.

One derives   from  the definition~\eqref{eq:Wtilde} of $\widetilde{{\mathbf W}}_{_{R+1}}$, using the auxiliary variables, that
\[
\widetilde{{\mathbf W}}_{_{R+1}}=  q_1^{-aR}\big(1+( M^{-R}-1)\widetilde{\overline{\mathbf W}}_{_R}\big)
\quad\mbox{ with }\quad
\widetilde{\overline{\mathbf W}}_{_R}=\sum_{r=1}^{R-1}\overline{\mathbf W}_r  x_r^{R-1}
\]
and the  $x_r$ are given by~\eqref{eq:VanderR-1}. Following the lines of~$(a)$, we derive that 
\[
\widetilde{\overline{\mathbf W}}_{_R}=\sum_{k\ge 0}  \frac{1}{ M^k}\widetilde{\ \overline{\mathbf W}}_{R,k}
\]
where  the identity~\eqref{eq:Vanderstuff01}  established in the above  lemma~\ref{lem:VanderStuf} yields
 
\[
\widetilde{\overline{\mathbf W}}_{R,k}=  M^{-k(R-1)} - \prod_{r=1}^{R-1}( M^{-k}-x_r).
\]
Finally 
\[
\widetilde{{\mathbf W}}_{_{R+1}}= q_1^{-aR}\left(1+( M^{-R}-1)\sum_{k\ge 0}\frac{1}{ M^{kR}}\Big(1- \prod_{r=0 }^{R-2 }\Big(1- M^{k-r}\Big(\frac{q_1}{q_{r+2 }}\Big)^a\Big)\right).
\]
Noting that $\sum_{k\ge 0}\frac{1}{ M^{kR}} = \frac{1}{1- M^{-R}}$ completes the proof this claim.  The computation of $\widetilde{\mathbf{W}}_{_{R+2}}$ follows likewise, starting from the identity~\eqref{eq:Vanderstuff02}.

%
\smallskip
\noindent $(b)$  In the the starting system~\eqref{eq:theSystem} for the weights $q_r^{a(\ell-1)}$ no longer depends on $r$ and can be cancelled in each equation. This leads to the system
\[
{\mathbf W}_1=1,\quad  1 +( M^{-(\ell-1)}-1)\sum_{r=2}^R  M^{-(r-2 )(\ell-1)}{\mathbf W}_r=0,\; \ell=2,\ldots,R.
\]
After a standard Abel transform, we get that ${\mathbf W}_r= \mathbf{w}_{r} +\cdots+ \mathbf{w}_{_R}$ where the $\mathbf{w}_r$ are solution to the Vandermonde system
\[
\sum_{r=1}^R  M^{-(r-1 )(\ell-1)}\mathbf{w}_r = 0^{\ell-1},\; \ell=1, \ldots,R.
\]
Note that these weights corresponds to those coming out  when dealing with $ML2R$ for regular Monte Carlo (see~\cite{LEPA}) under a weak error expansion condition at rate $\alpha=1$.

As for the boundedness,  first note that the ``small'' weights $\mathbf{w} _r$ read $\mathbf{w}_r= b_{R_-r}/a_r$, $r=1,\ldots,R$, with
\[
a_r = \prod_{k=1}^r(1-M^{-k}) \quad\mbox{ and }\quad b_r = (-1)^r M^{-\frac{r(r-1)}{2}} a^{-1}_r.
\]
One straightforwardly checks that $a_r \downarrow a_\infty =   \prod_{k\ge 1}(1-M^{-k})>0 $ and $B_{\infty}=\sum_{r\ge 1} |b_r|<+\infty$. As a consequence
\[
\forall\, R\!\in \EN^*  , \; \forall\, r\!\in \{1,\ldots,R\},\quad |\mathbf{W}^{(R)}_r| \le \frac{B_{\infty}}{a_{\infty}}<+\infty.
\]

Finally, the same Abel transform shows that 
\[
\widetilde{\mathbf W}_{_{R+i}}= R^{a(R+i)}\sum_{r=1}^R  M^{-(r-1 )(R+i-1)}\mathbf{w}_r ,\; i=1,2, 
\]
and one concludes by formula~\eqref{eq:Vanderstuff01} and~\eqref{eq:Vanderstuff02} from Lemma~\ref{lem:VanderStuf}. 
$\quad\cqfd$
\section{An additional bias term}
In this part of the appendix, we focus the bias induced by the approximation
$$\frac{\Gamma_{n_r}^{(\ell)}}{\Gamma_{n_r}}\approx q_r^{-a\ell}\frac{\Gamma_n^{(\ell)}}{\Gamma_n}\quad\textnormal{(with $\gamma_n=\gamzero n^{-a}$, $a\in(0,1)$)},$$
that we use to build some universal weights $({\bf W}_r^{(R)})_{r=1,\ldots,R}$ (by universal, we mean that they do not depend on $n$). We have the following lemma:
\begin{lemme}\label{lem:bias1} Assume that $\gamma_n=\gamzero n^{-a}$ with $a\in(0,1)$. 
 
\smallskip
\noindent $(a)$ Let $\chi\in(0,1)$ and $L\in\EN$ such that $La<1$. Then, for every $n\ge n_0=\lceil \frac{6^{\frac{1}{1-a}}}{\chi}\rceil$,
 \begin{eqnarray}
\nonumber \left|\frac{\Gamma^{(\ell)}_{\lfloor \chi n\rfloor}}{\Gamma_{\lfloor \chi n\rfloor}}-\chi ^{-a(\ell-1)}\frac{\Gamma^{(\ell)}_{n}}{\Gamma_n}\right|&\le& 3\Big(1+\frac{1-a}{1-a(R+1)}\Big) \frac{\gamzero ^{\ell-1}}{n^{1-a}}\frac{\chi ^{-a\ell} }{\chi ^{1-a}-3n^{a-1}} \\
\label{Gamma-a-ell}&\le& \left(6\,\frac{2-aL}{1-aL}\gamzero ^{\ell-1}\chi ^{-1-a(\ell-1)}\right)\frac{1}{n^{1-a}}  .
 \end{eqnarray}
\noindent $(b)$ Set 
\begin{equation*}
{\rm Bias}^{(1)}\!(a,R,q,n)= \sum_{\ell=2}^{\ERR}\left[\!\Big[\frac{\Gamma_{n_1}^{(\ell)}}{\Gamma_{n_1}}-q_1^{-a(\ell-1)}\frac{\Gamma_{n}^{(\ell)}}{\Gamma_n}\Big]\mathbf{W}_1  \! +\! \sum_{r=2}^{R}m_{r,\ell} {\bf W}_r\Big[\frac{\Gamma_{n_r}^{(\ell)} }{\Gamma_{n_r}}-q_r^{-a(\ell-1)}\frac{\Gamma_{n}^{(\ell)} }{\Gamma_{n}}\Big]\right]c_\ell
\end{equation*}
where $m_{r,\ell}=(M^{-(\ell-1)}\!\!-1)M^{-(r-2)(\ell-1)}$. We have:
 \begin{equation*}
\nonumber  |{\rm Bias}^{(1)}\!(a,R,q,n)|\le \frac{C_{a,{\bf q}, r}}{n^{1-a}},
 \end{equation*}
 where
 $$
 C_{a,q,r}= 6\,\frac{2-a(R+1)}{1-a(R+1)}  \|\mathbf{W}\|_{\infty}\,q_*^{-1} \sum_{\ell=2}^R (\gamzero q_*^{-a})^{\ell-1}\left[1+\sum_{r=2}^R m_{r,\ell}\right]\,|c_{\ell}|
 $$
  with $q_*= \min_{1\le r\le R} q_r$ and $\|\mathbf{W}\|_{\infty}=\sup_{r\in\{1\ldots,R\},R\ge2} {\bf W}_r^{(R)}$. 
  
  Furthermore, if $q_1=\ldots=q_R=\frac{1}{R}$, then ${\rm Bias}^{(1)}\!(a,R,q,n)=0$.
 \end{lemme}
\begin{Remarque}\label{Rem:5.10} Note that since $a<1/2$, $n^{1-a}=o(n^{-\frac{1}{2}})$ so that this term is negligible at the first and second orders of the expansions obtained in this paper. Finally, it is worth noting that
 this term 
is equal to $0$ when the $q_i$ are equal to $\frac{1}{R}$, case where, in addition, the ${\bf W}_r$, $r=1\ldots,R$ have a simple closed form given by~\eqref{eq:quniform0} and ~\eqref{eq:quniform} in~Lemma~\ref{lem:WWtilde}.
\end{Remarque}
\begin{proof}
First, we derive by a comparison argument   with integrals $\int_0^n x^{-a}dx$ and $\int_1^{n+1}x^{-a}dx$ that 
\begin{equation}\label{eq:Gamma-a}
\frac{n^{1-a}-2}{1-a}\le \sum_{k=1}^n k^{-a} \le \frac{n^{1-a}}{1-a}, \; n\ge 1, a\!\in (0,1).
\end{equation}

Elementary computations then show that, for every $a\!\in \big(0, \frac{1}{R}\big)$, $\chi\!\in (0,1)$,  and every $n\ge 1$, every integer $\ell\!\in\{ 1,\ldots,R+1\}$  
\[
\left|\frac{\Gamma^{(\ell)}_{\lfloor \chi n\rfloor}}{\Gamma_{\lfloor \chi n\rfloor}}-\chi ^{-a(\ell-1)}\frac{\Gamma^{(\ell)}_{n}}{\Gamma_n}\right|\le \frac{3\gamzero ^{\ell}}{\Gamma_{\lfloor \chi n\rfloor}}\left(\frac{1}{1-a\ell}+ \frac{\chi^{-a\ell}}{1-a}\right)
\]
Using that $u\mapsto u^{1-a}$ is $(1-a)$-H\"older, we derive from the left inequality in~\eqref{eq:Gamma-a} that  $\Gamma_{\lfloor \chi n\rfloor}\ge \gamzero \frac{(\chi n)^{1-a}-3}{1-a}$ so that, for every $n\ge \frac{6^{\frac{1}{1-a}}}{\chi}$,  
 \begin{eqnarray}
\nonumber \left|\frac{\Gamma^{(\ell)}_{\lfloor \chi n\rfloor}}{\Gamma_{\lfloor \chi n\rfloor}}-\chi ^{-a(\ell-1)}\frac{\Gamma^{(\ell)}_{n}}{\Gamma_n}\right|&\le& 3\Big(1+\frac{1-a}{1-a(R+1)}\Big) \frac{\gamzero ^{\ell-1}}{n^{1-a}}\frac{\chi ^{-a\ell} }{\chi ^{1-a}-3n^{a-1}} \\
\label{Gamma-a-ell}&\le& 6\,\frac{2-a(R+1)}{1-a(R+1)}\frac{\gamzero ^{\ell-1}}{n^{1-a}} \chi ^{-1-a(\ell-1)} .
 \end{eqnarray}
 
 \noindent  Now, since $\|\mathbf{W}\|_{\infty}<+\infty$ (see Lemma~\ref{lem:WWtilde}$(b)$), we deduce by  plugging the above inequality in ${\rm Bias}^{(1)}\!(a,R,q,n)$ that, for every  $n\ge \frac{6^{\frac{1}{1-a}}}{q_*}$, 
 \begin{eqnarray}
\nonumber  |{\rm Bias}^{(1)}\!(a,R,q,n)| &\le &   6\,\frac{2-a(R+1)}{1-a(R+1)}  \frac{1 }{n^{1-a}} \|\mathbf{W}\|_{\infty}\,q_*^{-1} \sum_{\ell=2}^R (\gamzero q_*^{-a})^{\ell-1}\left[1+\sum_{r=2}^R m_{r,\ell}\right]\,|c_{\ell}|.
 \end{eqnarray}
 When $q_r = \frac1R$, $r=1,\ldots,R$,   
 \[
 {\rm Bias}^{(1)}\!(a,R,q,n)= \sum_{\ell=2}^{\ERR}\!\left[\frac{\Gamma_{n_1}^{(\ell)}}{\Gamma_{n_1}}-\bar q_1^{-a(\ell-1)}\frac{\Gamma_{n}^{(\ell)}}{\Gamma_n}\right]\left(\mathbf{W}_1  \! +\sum_{r=2}^{R} {\bf W}_r m_{r,\ell} \right)c_\ell = 0
 \]
since $\mathbf{W}$ is solution to~\eqref{eq:theSystem}. 
\end{proof}

\end{appendix}
\end{document}